\documentclass{amsart}
\usepackage[utf8]{inputenc}\usepackage{amsmath}
\usepackage{amsfonts}
\usepackage{amssymb}
\usepackage{enumitem}
\usepackage{comment}
\usepackage[english]{babel}
\usepackage{amsmath,amsthm}
\usepackage{amsfonts}
\usepackage{fancyhdr}
\usepackage{setspace}
\usepackage{color}
\usepackage[percent]{overpic}
\usepackage{fix-cm}
\usepackage{graphicx}
\usepackage[all]{xy}
\usepackage{mathrsfs}
\usepackage{mathdots}
\usepackage{amsbsy}
\usepackage{cancel}
\usepackage{leftidx}
\usepackage[hidelinks]{hyperref}
\usepackage{xcolor}
\usepackage{todonotes}
\usepackage{pst-node}
\usepackage{tikz-cd}

\usepackage[backend=biber,style=trad-alpha,sorting=nyt,natbib=true,giveninits=true,doi=false,isbn=false,url=false,eprint=false]{biblatex}
\AtEveryBibitem{\clearlist{language}}

\AtEveryBibitem{%
  \ifentrytype{book}{%
    \clearfield{pages}%
  }{%
  }%
}
\usepackage[left=1in,right=1in,top=1in,bottom=1in]{geometry}
\usepackage{BOONDOX-uprscr}
\usepackage{mathtools}
\mathtoolsset{showonlyrefs}

\hypersetup{
    linktocpage,
    colorlinks,
    linkcolor={magenta!100!black},
    citecolor={cyan!100!black},
    urlcolor={yellow!100!black}
}

\usepackage{tikz}

\let\Re\undefined

\DeclareMathOperator{\Re}{Re}

\DeclareMathOperator{\Tr}{tr}

\newtheorem{thm}{Theorem}[section]
\newtheorem{Athm}{Theorem}
\newtheorem*{thm*}{Theorem}
\newtheorem*{que*}{Question}
\newtheorem{prop}[thm]{Proposition}
\newtheorem*{prop*}{Proposition}
\newtheorem{defn}[thm]{Definition}
\newtheorem{lem}[thm]{Lemma}
\newtheorem{cor}[thm]{Corollary}
\newtheorem{Acor}[Athm]{Corollary}

\newtheorem{rem}[thm]{Remark}

\numberwithin{equation}{section}

\def\bb #1{\mathbb{#1}}

\def\cal #1{\mathcal{#1}}

\def\rm #1{\mathrm{#1}}
\def\sf #1{\mathsf{#1}}

\def\mod {\textnormal{Mod}(S)}
\def\Exmod {\textnormal{Mod}^{\pm}(S)}
\def\tech {\mathcal{T}(S)}
\def\h3 {\mathcal{H}_3(S)}

\newcommand{\nn}{\nonumber}

\newcommand{\rmd}{\mathrm{d}}

\title{the covariance metric in the Blaschke locus}
\author{Xian Dai and Nikolas Eptaminitakis}

\address{\newline Xian Dai \newline Ruhr-Universit\"at Bochum \newline Fakult\"at f\"ur Mathematik \newline e-mail: xian.dai@ruhr-uni-bochum.de \newline \newline Nikolas Eptaminitakis \newline Institut für Differentialgeometrie\newline Leibniz Universität Hannover \newline 
Welfengarten 1, 30167 Hannover, Germany,
\newline e-mail: 
nikolaos.eptaminitakis@math.uni-hannover.de}

\begin{document}

\maketitle

\begin{abstract}
    We prove that the  Blaschke locus has the structure of a finite dimensional smooth manifold away from the Teichm{\"u}ller space and study its Riemannian manifold structure with respect to the covariance metric introduced by Guillarmou, Knieper and Lefeuvre in \cite{GeodesicStretch}. We also identify some families of geodesics in the Blaschke locus arising from Hitchin representations for orbifolds and show that they have infinite length with respect to the covariance metric.
\end{abstract}

\section{Introduction}

Classical Teichmüller theory is a rich field which involves the interplay of tools from analysis, geometry and topology. One beautiful theorem dating back to the early twentieth century states that the Teichmüller space is a finite dimensional smooth contractible manifold (see for example \cite{Teichmuller}, \cite{AhlforsTeich}, \cite{Mike_Harmonic}). Among many different proofs of this fact, one approach, due to Fischer and Tromba (\cite{FischerTromba}), is based on global analysis and Riemannian geometry, by viewing the Teichmüller space as a space of isotopy classes of hyperbolic metrics on a closed connected oriented  surface $S$ with genus $\mathscr{G} \geq 2$. Many other interesting results can also be obtained from this Riemannian geometrical characterization: for example, the Weil-Petersson metric on the Teichmüller space is Kähler and has negative sectional curvature (see e.g. \cite[Section 5]{Tromba_Book}).

In this note, we will explore properties of a finite dimensional subspace of the space of isotopy classes 
of negatively curved metrics that contains the Teichm{\"u}ller space, using this Riemannian geometrical approach. 
Given a complex structure $J$ on $S$ and a holomorphic cubic differential with respect to $J$, one can lift them to a universal cover $\tilde{S}$ of $S$ and produce a parametrization $f:\tilde{S}\to \bb{R}^3$  of a hypersurface of special type arising from affine differential geometry, 
called a \emph{hyperbolic affine sphere} 
(see \cite{Loftin_Survey}, and also Section \ref{subsec:affinesphere}). 
A  hyperbolic affine sphere in $\bb{R}^3$ is a surface of constant negative \emph{affine mean curvature} (see Section \ref{subsec:affinesphere}). It  comes naturally equipped with an affine invariant Riemannian metric which descends to a uniquely determined negatively curved metric on $S$ in the conformal class of $J$, called a  $\emph{Blaschke metric}$.
We denote the space of Blaschke metrics on $S$ by $\cal{M}^B$ and the space of smooth hyperbolic metrics on $S$ by $\cal{M}_{-1}$.
A hyperbolic metric $\sigma\in \cal{M}_{-1}$ on  $S$ is a special case of a Blaschke metric: 
in this case, the hyperbolic affine sphere determined by the complex structure 
corresponding to $\sigma$ and the zero cubic differential can be taken to be the hyperboloid model of hyperbolic space; the descended Blaschke metric is exactly the hyperbolic metric $\sigma.$
Therefore, upon taking quotients by $\cal{D}_0$, the space of smooth diffeomorphisms isotopic to the identity, the Teichmüller space $\cal{T}(S)=\cal{M}_{-1}/\cal{D}_0$
is contained in the space  $\cal{M}^B/\cal{D}_0$, which we call the \emph{Blaschke locus}.

Our work is in two directions:
the first goal is to understand the topology and regularity of the Blaschke locus, which is finite dimensional.
The second is to study some of its Riemannian geometric properties with respect to a Riemannian metric constructed in \cite{GeodesicStretch}  on the space of isotopy classes of negatively curved metrics which restricts to the Weil-Petersson metric on $\cal{T}(S)$.

\smallskip

\subsection{Structure of the Blaschke locus and relation to higher Teichmüller theory.}

Our first result is concerned with the  topology and smooth structure of $\cal{M}^B/ \cal{D}_0$. The proof follows the spirit of Tromba's proof of the fact that the Teichmüller space is a smooth manifold (\cite[Corrolary 2.4.6]{Tromba_Book}).

\begin{Athm} [Theorem \ref{thm:contratible}, Theorem \ref{smoothnessBlaschkeLocus}] \label{Theorem A} 
 The Blaschke locus $\cal{M}^B/\cal{D}_0$ is a contractible space. Moreover, it has the structure of a smooth manifold of dimension $16\mathscr{G}-17$ away from the Teichmüller space $\cal{T}(S)$. 
\end{Athm}

To motivate our interest in the Blaschke locus and the idea behind the proof of Theorem \ref{Theorem A}, we now further explain the relation between the Blaschke locus and  Teichmüller space, from a different point of view. 
For this, we start with a brief exposition of closely relevant objects---the Hitchin components. Classically, besides viewing the Teichmüller space $\cal{T}(S)$ as a space of (equivalence classes of)  Riemannian metrics of constant negative curvature (or hyperbolic structures) on $S$, one can also view it as a space of (equivalence classes of) Riemann surface structures (complex structures) on $S$ or as a connected component of the space of (conjugacy classes of) representations of $\pi_1(S)$ into $\mathrm{PGL}(2,\mathbb{R})$.  The Hitchin component $\mathcal{H}_n(S)$, as an important example of higher rank Teichm{\"u}ller spaces (see \cite{InvitationHigherTeich} for a survey), generalizes $\cal{T}(S)$ from the representation theory viewpoint and is a connected component of the space of (conjugacy classes of) representations from $\pi_1(S)$ into $\mathrm{PGL}(n,\mathbb{R})$ for $n\geq 2$. 
When $n=2$, the Teichmüller space $\cal{T}(S)$ coincides with $\mathcal{H}_2(S)$ and embeds into all other Hitchin components $\mathcal{H}_n(S)$. When $n=3$, a complex analytical counterpart of $\mathcal{H}_3(S)$ was discovered independently by Loftin \cite{Loftin_AffineSphere} and Labourie \cite{FlatProjectiveCubicDifferentials} in analogy to the viewpoint of $\cal{T}(S)$ as spaces of Riemann surfaces. 
Let $Q_3(S)$ be the vector bundle over the Teichm{\"u}ller space $\cal{T}(S)$ whose fiber over a  complex structure (or, more accurately, an equivalence class of complex structures) is given by the vector space  of holomorphic cubic differentials on $S$, which by the Riemann-Roch theorem is finite dimensional (cf. \eqref{eq:Q_n_def}).
They showed that there exists a mapping class group equivariant homeomorphism between the Hitchin component $\mathcal{H}_3(S)$ and the holomorphic cubic differentials vector bundle $Q_3(S)$. 
One can then ask whether there is a natural Riemannian geometrical generalization of the Teichmüller space $\cal{T}(S)$. 
As hinted previously, the Blaschke locus $\cal{M}^B/\cal{D}_0$ as a space of negatively curved Riemannian metrics, even though not in bijection to $\mathcal{H}_3(S)$, plays this role. Modulo an $S^1$ action on the vector bundle $Q_3(S)$ (which identifies
a holomorphic cubic differential $q$ with $e^{2\pi i\theta}q$ for any $\theta\in [0,1)$), using the holomorphic data $Q_3(S)$ as a bridge, one obtains the following mapping class group equivariant homeomorphisms
\begin{equation}\label{eq:homeo}
\mathcal{H}_3(S)/S^1 \stackrel{\text{homeo}}{\simeq} Q_3(S)/S^1 \stackrel{\text{homeo}}{\simeq} \cal{M}^B/\cal{D}_0,
\end{equation}
 where the $S^1$ action on $\mathcal{H}_3(S)$ is simply obtained by pullback of the $S^1$ action on $Q_3(S)$. 

The above three objects are generalizations of $\cal{T}(S)$ from different viewpoints (representation theoretic, complex analytic and Riemannian geometrical respectively). A more detailed characterization of them and their relations will be described in Section \ref{Sec:BlaschkeLocus}. In particular, the first homeomorphism is proved and implied from \cite{Loftin_AffineSphere} and \cite{FlatProjectiveCubicDifferentials}. The second bijection is first shown in \cite{OuyangTamburelli}. We explain the second homeomorphism in Proposition \ref{prop,homeo}. These identifications will be crucial for the proof of Theorem \ref{Theorem A}.

\smallskip

\subsection{The covariance metric in the Blaschke locus.} 
As mentioned before, we also study the Riemannian geometry of the
Blaschke locus $\cal{M}^B/\cal{D}_0$
with respect to the \emph{covariance metric} $G(\cdot,\cdot)$  introduced by
Guillarmou, Knieper and Lefeuvre (\cite{GeodesicStretch})\footnote{This Riemannian metric is referred to as the pressure metric in \cite{GeodesicStretch}.}.
This metric extends the Weil-Petersson metric
from the Teichm{\"u}ller space (viewed as the space of isotopy classes of hyperbolic metrics) to a Riemannian metric on the space of isotopy classes of  metrics of variable negative curvature, using techniques originating from the study of the X-ray transform on closed Anosov  manifolds (\cite{guillarmou_anosov}, \cite{Guill-lef}).
In our case, starting with the simple observation that the extended mapping class group is a subgroup of the group of isometries for the covariance metric (Proposition \ref{prop,mappingClassGrpInv}), the mapping class group equivariant homeomorphisms from $\cal{H}_3(S)/S^1$ to $\cal{M}^B/ \cal{D}_0$ in \eqref{eq:homeo} allow us to identify certain families of covariance metric geodesics in $\cal{M}^B/\cal{D}_0$ arising from special orbifold Hitchin representations (Section \ref{orbifolds}). 
Briefly,  let a two-dimensional  orbifold $Y$ be given as a quotient of a Riemann surface $X_J$ (with underlying smooth surface structure $S$) by a finite diffeomorphism group $\Sigma$.
One can then associate to $Y$ a special one-parameter family of representations in  $\mathcal{H}_3(S)$, denoted by $\mathcal{H}_3(Y)$, as a fixed point set of the group action  of $\Sigma$ on $\mathcal{H}_3(S)$, with $\Sigma$ understood as a subgroup of the extended mapping class group (see Section \ref{subsec:MCGHitchin}). We show

\begin{Athm} [Lemma \ref{CircleActionY}, Theorem \ref{thm,geodesics}] \label{Theorem 1}
Let $Y$ be a non-orientable orbifold of negative Euler characteristic with orientation double cover $Y^{+}$ given by a sphere with 3 cone points of respective orders $m_1\geq 3$, $m_2\geq 3$ and $m_3\geq 4$. Then $\mathcal{H}_3(Y)/\mathbb{Z}_2$ is homeomorphic to a half line and embeds as a  geodesic (unparametrized) in $\mathcal{M}^B/\mathcal{D}_0$ with respect to the covariance metric $G(\cdot,\cdot)$, where the $\mathbb{Z}_2$ action on $\mathcal{H}_3(Y)$ is induced from the $S^1$ action on $\mathcal{H}_3(S)$. 
\end{Athm}

\noindent These geodesics, which are homeomorphic to half lines, have starting points in Teichm{\"u}ller space $\cal{T}(S)$ and eventually leave all compact sets of $\mathcal{M}^B/\mathcal{D}_0$ (see Theorem \ref{thm, singularFlat}). We further proceed in Section \ref{infiniteLength} to estimate their covariance metric lengths. We show in Corollary \ref{CircleActionY},

\begin{Athm} [Corollary \ref{CircleActionY}]
The covariant metric geodesics  in $\mathcal{M}^B/\mathcal{D}_0$  corresponding to $\mathcal{H}_3(Y)/\mathbb{Z}_2$ given in Theorem \ref{Theorem 1}  have infinite length. 
\end{Athm}

\noindent In fact, our proof works more generally for any curve in $\mathcal{M}^B/\mathcal{D}_0$  parametrized by a ray starting from $\cal{T}(S)$  in a fixed fiber of the bundle $Q_3(S)/S^1$, using identification \eqref{eq:homeo},

\begin{Acor} [Theorem \ref{thm,infiniteLength}]
Let $\sigma $ be a hyperbolic metric on $S$ and $q$ be a nonzero  cubic differential which is holomorphic with respect to the complex structure determined by $\sigma$.
Then the curve $\{[g_t]\}_{t\geq0} \subset \cal M^B/\cal D_0$, where $g_t\subset \cal M^B$ satisfies  Wang's equation \eqref{WangEquation} with cubic differential $\sqrt{t}q$,
has infinite length with respect to the  covariance metric.
\end{Acor}

It remains as a question whether there are other candidates for finite covariance metric length paths leaving all compact sets of $\mathcal{M}^B/\mathcal{D}_0$ but not in the Teichm{\"u}ller space $\cal{T}(S)$. In general, incomplete paths in moduli spaces indicate meaningful geometric phenomena. For instance, in the Teichm{\"u}ller space $\cal{T}(S)$, Wolpert exhibits some incomplete paths for the Weil-Petersson metric in \cite{Wolpert_Noncomplete}. These are paths realizing ``pinched Riemann surfaces''.

In the Appendix \ref{furtherEstimates} we end with some further estimates of the covariance metric in $\mathcal{M}^B/\mathcal{D}_0$.
An explicit formula (Proposition \ref{prop: FiberDirection}) for the covariance metric $G(\cdot,\cdot)$ at a point in $\cal{T}(S)$ with tangent vectors corresponding to a direction tangential to the fiber of $Q_3(S)/S^1$ is given as a direct application of \cite[Lemma A.1, Remark A.2]{GuilMon}. We hope that  this formula can be further simplified in the future. 

\smallskip 

\subsection{ Outline of the proofs.} 
We briefly discuss our proofs of the main theorems in the sequel. Both for regularity results and the study of the covariance metric in $\mathcal{M}^B/ \cal{D}_0$, an important tool used in our investigation is a single partial differential equation, called Wang's equation (\ref{WangEquation}). It underlies the second identification \eqref{eq:homeo} and has natural connection to Blaschke metrics and the theory of affine differential geometry. 

For regularity results, the ideas are as follows:
\begin{itemize}
    \item The proof  of the smoothness of the Blaschke locus $\cal{M}^B/\cal{D}_0$ relies on constructing smooth charts for it away from $\cal{T}(S)$, and is modeled  on the  construction of charts for the Teichmüller space outlined in \cite[Section 2.4]{Tromba_Book}.    
    There, the key observation is that the finite dimensional space of transverse traceless (= divergence free and trace free) symmetric two tensors with respect to a fixed hyperbolic metric can be locally identified with a slice of smooth hyperbolic metrics inside the Hilbert manifold of hyperbolic metrics of fixed Sobolev regularity. This slice locally parametrizes $\cal{T}(S)$, thus providing a natural local coordinate for it.
    For us, the coordinates for the Blaschke locus are constructed via local identification with the vector bundle of holomorphic cubic differentials over those slices for $\cal{T}(S)$. 
    
    \item The link between local slices for the bundle of holomorphic cubic differentials and local slices for the space of Blaschke metrics, viewed as a subset of the Banach manifold of negatively curved metrics of $C^{k,\alpha}$ regularity, is given by  Wang's equation (\ref{WangEquation}). 
    It allows us to produce smooth diffeomorphisms between those slices, by means of the implicit function theorem and the inverse function theorem for Banach spaces.
    
\end{itemize}

To study the covariance metric in $\cal{M}^B/\cal{D}_0$, we use a mixture of global geometry concerning mapping class groups actions and local estimates using tools from partial differential equations:

\begin{itemize}
    \item The equivariance of the  homeomorphisms in \eqref{eq:homeo} with respect to the mapping class group action allows one to preserve the fixed point sets of actions of certain subgroups of it  on the different spaces. By showing that the covariance metric is extended mapping class group invariant, some one-dimensional fixed points sets arising from geometric symmetry in the Hitchin components pull back by   \eqref{eq:homeo} to geodesics in the Blaschke locus.
    Then, analytical tools can be applied.
    \item To estimate the lengths of the  geodesics above,  Wang's equation is again a key, together with some standard techniques from the theory of partial differential equations. These include the existence of supersolutions and subsolutions for Wang's equation, previously obtained by Loftin (Proposition \ref{prop:SuperSub}), and some maximum principle arguments. 
\end{itemize}

\smallskip

\subsection{ Structure of the article.} The article is organized as follows. In Section \ref{sec: Prelim}, we recall some fundamental results from Teichmüller theory and Weil-Petersson geometry. We then introduce Higgs bundles, Hitchin components and Hitchin maps. We also include a short discussion on orbifolds and orbifold representations. 
Section \ref{sec: MappingClassGrp} is devoted to explaining the actions of the extended mapping class group on various mathematical objects from different areas. This will play an important role in the proofs of Section \ref{sec:geodesics}. Section \ref{sec:PressureMetric} contains an exposition on the covariance metric introduced in \cite{GeodesicStretch} in the space of negatively curved metrics. We also show in this section that the covariance metric is extended mapping class group invariant. In Section \ref{Sec:BlaschkeLocus}, we introduce Blaschke metrics, the Blaschke locus, and explain the identification \eqref{eq:homeo}. We also discuss some important results concerning the regularity and topology of  Blaschke locus.  In Section \ref{sec:geodesics}, we prove some results about geodesics in the Blaschke locus with respect to the covariance metric and estimate their lengths. Finally, we end with some further estimates of the covariance metric in the Blaschke locus near the Teichmüller space $\cal{T}(S)$ in  Appendix \ref{furtherEstimates}.

\smallskip

\subsection*{Acknowledgements.}
The authors would like to thank Kiril Datchev, Dan Fox, Gerhard Knieper, Qiongling Li and Gabriele Viaggi for helpful discussions, and Alex Nolte for  the reference \cite{bers}.
We would also like to thank the referee for carefully reading the manuscript and providing constructive feedback.
X. Dai was funded by the Deutsche Forschungsgemeinschaft (DFG, German Research Foundation) – Project-ID 281071066 – TRR 191
and by the DFG under Germany's excellence strategy exc 2181/1 -
390900948 (the Heidelberg structures excellence cluster).

\section{Preliminaries} \label{sec: Prelim}

This section develops the background material we will need in later sections. A reader familiar with this material may  skip it. We begin in Section \ref{subsec: teich} with an exposition on the classical Teichm{\"u}ller space and the Weil-Petersson metric. Then in Section \ref{subsec,HitchinComponents}, we introduce some basics on Higgs bundles and Hitchin components. We conclude with a discussion of orbifolds and orbifold representations in Section \ref{orbifolds}.

\subsection{Teichm{\"u}ller space and Weil-Petersson metric} \label{subsec: teich}

In this subsection, we will discuss Teichm{\"u}ller space from the viewpoint of Riemannian geometry, initiated by Tromba and Fischer \cite{FischerTromba}.

Let $S$ be a closed orientable smooth surface of genus $\mathscr{G}\geq 2$. We denote by $\mathcal{M}$ the space of smooth Riemannian metrics on $S$, by $\mathcal{M}_-$ the subspace of negatively curved smooth Riemannian metrics, and by $\mathcal{M}_{-1}$ the subspace of hyperbolic metrics on $S$. 
We also  denote by $\cal{D}$ be the diffeomorphism group of $S$, 
 by $\mathcal{D}^+$  the group of orientation preserving diffeomorphism on $S$ (when $S$ is given an orientation), and by $\cal{D}_0$  the normal subgroup of $\cal{D}$ consisting of smooth diffeomorphisms isotopic to the identity.

\begin{defn} \label{def,teich}

The \emph{Teichm{\"u}ller space}, denoted as $\cal{T}(S)$, is the quotient space $\cal{M}_{-1}/\cal{D}_0$, where the right action of $\mathcal{D}_0$ on $\mathcal{M}_{-1}$ is given by
\begin{align*}
    \mathcal{M}_{-1} \times \mathcal{D}_0 &\to \mathcal{M}_{-1},\\
    (\sigma, \psi) & \mapsto \psi^*\sigma.
\end{align*}
We denote the equivalence class of $\sigma\in \cal{M}_{-1}$ by $[\sigma]\in \cal{M}_{-1}/\cal{D}_0$.
\end{defn}

Equivalently,  the Teichm{\"u}ller space $\cal{T}(S)$ is the space of (oriented) complex structures on $S$ up to $\mathcal{D}_0$-action. According to another viewpoint, the Teichm{\"u}ller space $\cal{T}(S)$ is a connected component of the representation space $\textnormal{Hom}(\pi_1(S), \textnormal{PGL}(2,\mathbb{R}))/\textnormal{PGL}(2,\mathbb{R})$. This will be discussed in Section \ref{subsec,HitchinComponents}.

\begin{rem}\label{rem:orientation}
We remark that with Definition \ref{def,teich} above, the Teichmüller space does not ``see'' the orientation on $S$, in the sense that if $\psi$ is an orientation reversing isometry for a metric $\sigma$, then $[\psi^*\sigma]=[\sigma]\in \cal{T}(S)$. 
When $\cal{T}(S)$ is viewed as the space of oriented hyperbolic structures modulo $\cal{D}_0$, which is a point of view often taken in the literature (see e.g. \cite{WP_isometry}), if $\psi:S\to S$ is an orientation reversing isometry for a hyperbolic metric $\sigma$ and $\mathfrak{G}$ is the positively oriented hyperbolic structure that $\sigma$ determines, then the hyperbolic structure obtained by pulling back the charts of $\mathfrak{G}$ by $\psi$, $\rm{mod}\; \cal{D}_0$, is an element of the Teichm\"uller space of $\overline{S}$, where $\overline{S}$ has the opposite orientation from $S$.
\end{rem}

Given a Riemann surface $X_J$ with complex structure $J$, we denote by $K_J$ the canonical line bundle associated to $X_J$, which is the $(1,0)$-part of the complexified cotangent bundle $T^*X_J^{\mathbb{C}}=\mathbb{C}\otimes_{\mathbb{R}}T^*X_J$.
Further, we denote by $H^{0}(X_{J}, K_{J}^d)$ the space of $J$-holomorphic differentials of order $d$,
and by $H^{0}(X_{[J]}, K_{[J]}^d)$ the space of their $\mathcal{D}_0$-equivalence classes. 
Explicitly, if $\psi\in \cal{D}_0$, the $J$-holomorphic differential $q\in H^{0}(X_{J}, K_{J}^d)$ and the $\psi^*J$-holomorphic differential $\psi^*q\in H^{0}(X_{\psi^*J}, K_{\psi^*J}^d)$  represent the same point in $H^{0}(X_{[J]}, K_{[J]}^d)$, denoted as $[q]$.
The action of the groups $\mathcal{D}$ and $\mathcal{D}_0$ on complex structures and holomorphic differentials will be explained in detail in Section \ref{sec: MappingClassGrp}. 
It is well known that the cotangent space of the Teichm{\"u}ller space $\cal{T}(S)$ at $[\sigma]$ can be identified with the space of quadratic differentials $H^{0}(X_{[J]}, K_{[J]}^2)$ on the Riemann surfaces $X_{[J]}=(S, [J])$ where $[J]$ is associated to the ($\mathcal{D}_0$-equivalence class of) hyperbolic metrics $[\sigma]$. An extensively studied Riemannian metric on the Teichm{\"u}ller space $\cal{T}(S)$, defined using holomorphic quadratic differentials, is the following \emph{Weil-Petersson metric}.

\begin{defn}

The Weil-Petersson metric is a cometric on $\cal{T}(S)$ defined by
$$\big\langle [q_1],[q_2] \big\rangle_{\scriptscriptstyle  WP}([\sigma])=  \textnormal{Re}\int_{X_J} \frac{q_1\overline{q_2}}{\sigma^2} d v_{\sigma},$$
where $[\sigma] \in \cal{T}(S)$ and $[q_1],[q_2]$  are holomorphic quadratic differentials with respect to $[J]$ and $\sigma,J, q_1,q_2$ are representatives picked from their equivalence classes so that $q_1, q_2$ are holomorphic with respect to $J$.

\end{defn}

It is well known that the Weil-Petersson metric is K\"ahler (\cite{Ahlfors_Kahler}) and negatively curved (\cite{Ahlfors_Curvatures}, \cite{Tromba_Book}, \cite{Wolpert_Curvatures}). The isometry group of the Weil-Petersson metric is the mapping class group \cite{WP_isometry}. Also, although the Weil-Petersson metric is not complete  (\cite{Chu_incomplete}, \cite{Wolpert_Noncomplete}), it exhibits many nice properties of complete negatively curved metrics (see \cite{Wolpert_Noncomplete}, \cite{Wolpert_NielsenProblem}).
In  Section \ref{subsec: TeichPressure} we will discuss a different interpretation of the  Weil-Petersson metric (Theorem \ref{thm,Pressure=WP}).

\subsection{Hitchin components and Hitchin map}\label{subsec,HitchinComponents}

In this subsection, using Higgs bundles techniques, we introduce the \emph{Hitchin component}, which is a connected component of the representation space
$$\textnormal{Rep}(\pi_1 S,\mathrm{PGL}(n,\mathbb{R})):= \textnormal{Hom}(\pi_1 S,\mathrm{PGL}(n,\mathbb{R}))/\mathrm{PGL}(n,\mathbb{R}),$$
for $n\geq 2$, as a generalization of the classical Teichm{\"u}ller space $\cal{T}(S)$. 
\subsubsection{Higgs bundles, Hitchin components and Hitchin Sections} \label{subsec: HitchinComp}

In this subsection, $n\geq 2$ is an integer. 
The  discussion here holds for much wider classes of groups, but for our purposes we will restrict to the group $G=\mathrm{PGL}(n,\mathbb{R})$.

Let $\mathfrak{sl}(n,\bb{R})=\mathfrak{so}(n,\mathbb{R})\oplus sym^0(n,\bb{R})$ be the Cartan decomposition of $\mathfrak{sl}(n,\mathbb{R})$, where  $sym^0(n,\bb{R})$ is the set of $n \times n$ symmetric matrices of trace zero. Denote $ sym^0(n,\bb{C})=sym^0(n,\bb{R}) \otimes \bb{C}$.  The following definition is a special case of \cite[Definition 3.14]{Orbifolds}.

\begin{defn}\label{def,Higgsbundel}
A $\mathrm{PGL}(n,\mathbb{R})$-Higgs bundle on a Riemann surface $X_J$ is a pair $(\cal{E},\varphi)$, where
\begin{itemize}
    \item $\cal{E}$ is a holomorphic Lie algebra bundle with typical fiber $\mathfrak{sl}(n,\mathbb{C})$ and structure group $\mathrm{PO}(n,\mathbb{C})$,
    \item $\varphi \in H^0(X_J; K_J \otimes ad_{sym^0(n,\bb{C})}(\cal{E}))$ is a holomorphic section, called the \emph{Higgs field}. 
\end{itemize}
Here by $ad_{sym^0(n,\bb{C})}(\cal{E})$ we mean the bundle of symmetric adjoint endomorphisms of $\cal{E}$, i.e. endomorphisms of $\cal{E}$ locally of the form $ad_\xi: \mathfrak{sl}(n,\mathbb{C})\to\mathfrak{sl}(n,\mathbb{C}) $ for some $\xi\in sym^0(n,\bb{C})$. 
\end{defn}

\noindent The space of gauge equivalence classes of $\mathrm{PGL}(n, \mathbb{R})$-Higgs bundles with some ``good'' conditions forms the \emph{moduli space of $\mathrm{PGL}(n, \mathbb{R})$-Higgs bundles}, denoted by $\mathcal{M}_{Higgs}(\mathrm{PGL}(n,\mathbb{R}))$. (The ``good'' conditions are polystablity conditions for the Higgs bundles. One can find an introduction in \cite{Baraglia_thesis}, for example.)

Suppose that the holomorphic vector bundle $\cal{E}$ has holomorphic structure $\bar{\partial}_\cal{E}$ and is equipped with a Hermitian metric $H$. Recall that the \emph{Chern connection} $A_H$ of $\cal{E}$ is the unique connection that is compatible with $H$ and satisfies
$A_{H}^{0,1}=\bar{\partial}_{\cal{E}}$.  The following is important.

\begin{thm}[{Hitchin \cite{Hitchin87theself-duality}, Simpson \cite{Simpson-Variation-of-Hodge}}] \label{rem FlatConnection}
Let $(\cal{E}, \varphi)$ be a polystable $\mathrm{PGL}(n,\bb{R})$-Higgs bundle and let $H$ be a Hermitian metric on $\cal{E}$. A connection $D=A_H+ \varphi+ {\varphi}^{*H}$ on $(\cal{E}, \varphi, H)$ is flat if and only if  the following \emph{Hitchin equation} is satisfied,
\begin{align}
    F_{A_H}+[\varphi,{\varphi}^{*H}]=0,
   \label{eq HitchinEquation}
\end{align}
where $F_{A_H}$ is the curvature of the connection $A_H$. Moreover, the holomorphic vector bundle $\cal{E}$ admits a Hermitian metric $H$ satisfying the Hitchin equation if and only if $(\cal{E},\varphi)$ is polystable.
\end{thm}

\noindent We will come back to a special case of the Hitchin equation (Eq. \eqref{WangEquation}), which is of central importance for this note, in Section \ref{subsec:affinesphere}.

Hitchin \cite{LieGrpandTeichmullerSpace} further introduces the Hitchin component using Higgs bundles and the Hitchin section. We briefly discuss the Hitchin components and the Hitchin section here and refer the reader to section 2 of \cite{Baraglia_thesis} for a more comprehensive exposition.
Let $\mathfrak{g}_{\mathbb{C}}$ be $\mathfrak{sl}(n,\bb{C})$ and let $\mathfrak{g}$ be $\mathfrak{sl}(n,\bb{R})$ which is a split real form fixed by an antilinear Lie algebra involution $\tau$ of $\mathfrak{g}_{\mathbb{C}}$. Given a principal 3-dimensional subalgebra $\mathfrak{s}=\text{span}\{x, e, \tilde{e}\}$ of $\mathfrak{sl}(n,\mathbb{C})$ consisting of a semisimple element $x$ and regular nilpotent elements $e$ and $\tilde{e}$ with commutation relations
$$[x,e]=e, \qquad [x,\tilde{e}]=-\tilde{e}, \qquad [e,\tilde{e}]=x,$$
the Lie algebra 
$\mathfrak{sl}(n,\mathbb{C})$ decomposes into a direct sum of irreducible subspaces under the adjoint representation of $\mathfrak{s}$:
$$\mathfrak{sl}(n,\mathbb{C})= \bigoplus\limits_{i=1}^{n-1} V_i.$$
We take $e_1,\cdots, e_{n-1}$ as the highest weight elements of $V_1, \cdots, V_{n-1}$, where $e_1=e$. Another decomposition is 
$$\mathfrak{sl}(n,\mathbb{C})= \bigoplus\limits_{d=-n+1}^{n-1} \mathfrak{g}_{\mathbb{C}}^{(d)},$$
where $\mathfrak{g}_{\mathbb{C}}^{(d)}$ is the subspace of $\mathfrak{sl}(n,\mathbb{C})$ on which $ad_x$ acts with eigenvalue $d$. Associated to this decomposition is a natural Lie algebra bundle

$$\cal{E}_{can}:=\bigoplus\limits_{d=-n+1}^{n-1} \mathfrak{g}_{\mathbb{C}}^{(d)}\otimes K_J^d.$$
This is a common choice of the holomorphic bundle $\cal{E}$ in Definition \ref{def,Higgsbundel}. With this defined, we can  introduce the Hitchin section in the setting of $\mathcal{M}_{Higgs}(\mathrm{PGL}(n,\mathbb{R}))$.

\begin{defn} \label{HitchinSection}
A Hitchin section $s_J$ is a map from $\bigoplus\limits_{i=2}^{n} H^{0}(X_J,K_J^i)$ to $\mathcal{M}_{Higgs}(\mathrm{PGL}(n,\mathbb{R}))$ defined as follows:
for $q=(q_{2},q_{3},\cdots, q_{n})\in \bigoplus\limits_{i=2}^{n} H^{0}(X_J,K_J^i)$, the image $s_J(q)$ is a Higgs bundle $\cal{E}_{can}$
with its Higgs field $\varphi(q) \in H^0(X, K_J \otimes ad_{sym^0(n,\bb{C})}(\cal{E}_{can}))$ given by
$$\varphi(q)=\tilde{e}+q_{2}e_{1}+q_{3}e_{2}+ \cdots q_{n}e_{n-1}.$$
\end{defn}

Hitchin in \cite{LieGrpandTeichmullerSpace} shows that any Higgs bundle in the image of the Hitchin section $s_J$ has the associated flat connection $D$ (Theorem \ref{rem FlatConnection}) with holonomy in $\mathrm{PGL}(n,\mathbb{R})$ (See \cite[ Section 3]{Orbifolds}). This leads to the following important definition of the \emph{Hitchin component}:

\begin{defn} [\cite{LieGrpandTeichmullerSpace}]\label{def,HitchinComponent}
When $n\geq2$, the Higgs bundles in the image of the Hitchin section $s_J$ are stable and have holonomy in $\mathrm{PGL}(n,\mathbb{R})$.   Moreover, the corresponding representations form a connected component of the representation space $\textnormal{Rep}(\pi_1 S,\mathrm{PGL}(n,\mathbb{R}))$.
     This connected component is homeomorphic to a Euclidean space of dimension $(2\mathscr{G}-2)(n^2-1)$ and is called the \emph{Hitchin component}, denoted as $\mathcal{H}_n(S)$.
\end{defn}

 An element in $\mathcal{H}_n(S)$ is a conjugacy class of representations, called a (conjugacy class of a) \emph{Hitchin representation}. Given a representation $\rho$, we denote its conjugacy class by $[\rho]$. When $n=2$, the Hitchin section exactly parametrizes the Teichm{\"u}ller space, i.e., 
 one has $\mathcal{H}_2(S)=\mathcal{T}(S)$. When $n>2$, one can choose a principal $3$-dimensional subalgebra $\mathfrak{s}\cong \mathfrak{sl}(2,\mathbb{C})$ of the form $\mathfrak{s}=\text{span}\{x, e, \tilde{e}\}$ which is $\tau$-invariant, and it induces an inclusion $\mathfrak{sl}(2,\mathbb{R}) 	\hookrightarrow \mathfrak{sl}(n, \mathbb{R})$.  This induces a group homomorphism $\kappa_{\mathbb{C}}$ from $\mathrm{PGL}(2,\mathbb{C})\simeq\mathrm{Int}(\mathfrak{sl}(2,\mathbb{C}))$ to $\mathrm{PGL}(n,\mathbb{C})\simeq\mathrm{Int}(\mathfrak{sl}(n,\mathbb{C}))$ (see \cite[Section 2.2]{Orbifolds}). By restricting the groups respectively to $\mathrm{PGL}(2,\mathbb{R})$ and $\mathrm{PGL}(n,\mathbb{R})$, one obtains the \emph{principal representation}
$\kappa: \mathrm{PGL}(2,\mathbb{R}) \to \mathrm{PGL}(n,\mathbb{R})$ and an embedding of the Teichm{\"u}ller space $\mathcal{T}(S)$ in the Hitchin component $\mathcal{H}_n(S)$. In other words, the principal representation $\kappa$ sends $ [\rho_0]\in\mathcal{T}(S)$ to $[\rho]=[\kappa\circ \rho_0]\in\mathcal{H}_n(S)$. We often call $[\rho]=[\kappa\circ \rho_0]$ (a conjugacy class of) \emph{Fuchsian representations} of $\mathcal{H}_n(S)$ and the space of conjugacy classes of Fuchsian representations is called the \emph{Fuchsian locus} of $\mathcal{H}_n(S)$.

\begin{rem} 
We note that in the literature, the Hitchin component is usually defined as a component of $ \textnormal{Hom}(\pi_1 S,\mathrm{PSL}(n,\mathbb{R}))/\mathrm{PSL}(n,\mathbb{R})$ (\cite{LieGrpandTeichmullerSpace}, \cite{FlatProjectiveCubicDifferentials}). In this note, we define $ \cal{H}_n(S) $ up to $\mathrm{PGL}(n,\mathbb{R})$ conjugacy (to work with orbifolds and orientation reversing maps). The differences between these definitions are further discussed in Remark \ref{rem PGLvsPSL}.\end{rem}

\subsubsection{Hitchin map}\label{sussec:HitchinMap}
The holonomy maps of flat connections $D$ associated to Higgs bundles in the images of Hitchin section $s_{J}$ induce a homeomorphism $\displaystyle  H_{J}: \bigoplus\limits_{i=2}^{n} H^{0}(X_{J},K_{J}^i) \to \mathcal{H}_n(S)$. This map describes a parametrization of the Hitchin component $\mathcal{H}_n(S)$ by holomorphic differentials and is often called the \emph{Hitchin parametrization}. 
However, one drawback of the Hitchin parametrization is that it depends on a specific choice of complex structure $J$. In particular, it breaks the invariance with respect to the mapping class group action, which will be extensively discussed in Section \ref{sec: MappingClassGrp}. An attempt towards a mapping class equivariant construction is the following, due to Labourie \cite{CrossRatiosEnergyFunctional}. Let $Q_n(S)$ denote the vector bundle over the Teichm{\"u}ller space  $\tech$ whose fiber over $X_{[J]}$ is 
\begin{equation}\label{eq:Q_n_def}
    Q_n(S)\big|_{[J]}= H^{0}(X_{[J]}, K_{[J]}^3) \oplus \cdots \oplus H^{0}(X_{[J]}, K_{[J]}^n). 
\end{equation}
 Then consider the \emph{Hitchin map} $\mathbf{H}$ defined as
 \begin{equation}
\begin{aligned}
    \mathbf{H}: Q_n(S)  &\longrightarrow \mathcal{H}_n(S)\\ 
       ([J, q_3,\cdots , q_n]) &\mapsto  H_{J}(0, q_3, \cdots ,q_n).
\end{aligned}\label{labourie_map}
\end{equation}
 Here $J$ is a representative of the equivalence class $[J]$ and $q_i$ are $i$-th order differentials holomorphic with respect to $J$ that are representatives from $[q_i]$. The image of $H_J$ is obtained by taking holonomy of the flat connection $D$ associated to the Higgs bundle $s_J(0, q_3, \cdots q_n)$ (Recall Definition \ref{HitchinSection}). 

\begin{rem}\label{equivariantAction}
The vector space $ H^{0}(X_{J},K_{J}^i)$ can be identified with the space $H^{0}(X_{[J]}, K_{[J]}^i)$.  Via the map $H_J$, the holonomy defined using the flat connection associated to $([J, q_3, \cdots, q_n])$ induces representations well defined up to conjugation.  
\end{rem}

The Hitchin map is always a surjective mapping class group equivariant map. A natural question to ask is whether the Hitchin map is a homeomorphism. A recent result from Sagman and Smille (\cite{CounterLabouire}) shows that it fails to be injective and therefore not a homeomorphism when $n\geq 4$. However, in this note we will only focus on the case $n=3$, for which the homeomorphism result is well known. 

\begin{thm}[{\cite[Theorem 2]{Loftin_AffineSphere}, \cite[Theorem 1.0.2]{FlatProjectiveCubicDifferentials}}] \label{thm, HtichinMap}
The Hitchin map $$ \mathbf{H}: Q_3(S) \to \h3 $$ is a mapping class group equivariant homeomorphism, where the mapping class group actions on $Q_3(S)$ and on $\mathcal{H}_3(S)$ (as outer automorphism group action) are the left actions which will be explained in Section \ref{sec: MappingClassGrp}.  
\end{thm}

\begin{rem}\label{rmk_topology}
It will be useful for later to remark that the bundle $Q_n(S)$ over $\cal{T}(S)$, with the latter viewed as $\cal{M}_{-1}/\cal{D}_0$ and with the fiber over each $[\sigma]$ consisting of holomorphic differentials with respect to the positively oriented complex structure determined by $[\sigma]$, has the natural $C^{\infty}$ topology (which is the quotient topology descended from the $C^\infty$ topology of objects without taking quotients by the $\mathcal{D}_0$ action). For holomorphic differentials, the $C^{\infty}$ topology is equivalent to the compact open topology due to Weierstrass' Theorem. Therefore a sequence of pairs of hyperbolic metrics and holomorphic differentials  $\{(\sigma_k, q_k)\}_{k\geq 0}$ converges to a pair of hyperbolic metric and holomorphic differential $(\sigma, q)$ if the hyperbolic metrics $\sigma_k$ converge to $\sigma$ in $C^{\infty}$ topology and the lifts of holomorphic differentials $q_k$ to the universal covers $\mathbb{D}$ converge uniformly on compact subsets of $\mathbb{D}$ to the lift of $q$.
\end{rem}

\subsection{Orbifolds and orbifolds representations}\label{orbifolds}

\subsubsection{Orbifolds}
An orbifold is a space that is locally modelled on $\mathbb{R}^n$ modulo finite group actions. In the special case that all of these finite groups are trivial, we obtain a manifold. Otherwise,  orbifolds  have singularities. For a general introduction on orbifolds, we refer the reader to \cite[Chapter 13]{ThurstonBooks}. Much of the presentation in this subsection follows \cite[Section 2]{Orbifolds}.
We restrict our discussion to $n=2$. Let $Y$ be a closed connected smooth orbifold of dimension $2$. There are three types of singularities of $Y$:
\begin{enumerate}
    \item $p$ is a cone point of order $m$: there is a neighborhood of $p$ that is isomorphic to $\mathbb{R}^2/ \mathbb{Z}_m$ where $\mathbb{Z}_m$ acts on $\mathbb{R}^2$ by rotation.
    \item $p$ is a mirror point: there is a neighborhood of $\mathbb{R}^2/ \mathbb{Z}_2$ where  $\mathbb{Z}_2$ acts by reflection in the $y$-axis. 
    \item $p$ is a corner reflector of order $n$: there is a neighborhood of $p$ that is isomorphic to $\mathbb{R}^2/D_n$ where $D_n$ is the dihedral group of order $2n$, with presentation
    $$\langle a, b: a^2=b^2=(ab)^n=1 \rangle.$$
\end{enumerate}
For a $2$-dimensional orbifold $Y$, we will denote by $k$ the number of cone points (of respective orders $m_1, \cdots, m_k$) and by $l$ the number of corner reflectors (of respective orders $n_1, \cdots, n_l$). We will also denote by $\tilde{Y}$ the orbifold universal cover of $Y$. The \emph{orbifold fundamental group} is denoted by $\pi_1(Y)$, which is defined to be the group of deck transformations of the universal cover $\tilde{Y}$. We say $Y$ \emph{orientable} if its underlying topological space $|Y|$ is orientable and if $Y$ has only cone points as singularities.
The \emph{Euler characteristic} of an orbifold $Y$ is defined as
$$\chi(Y)= \chi(|Y|)- \sum_{i=1}^{k} (1-\frac{1}{m_i})-\frac{1}{2}\sum_{j=1}^{l} (1-\frac{1}{n_j}).$$
We will assume in this note that $Y$ has negative Euler characteristic:  $\chi(Y)<0$.
We say that a 2 dimensional orbifold is a \emph{good orbifold} if it has some covering orbifold which is a surface. Every orbifold of negative Euler characteristic is a good orbifold. It can be seen as a quotient of a closed orientable surface and has a presentation defined as follows:

\begin{defn}
A presentation of  a closed connected orbifold $Y$ is a triple $(S, \Sigma, \varphi)$, where $S$ is a  smooth closed connected orientable surface, $\Sigma$ is a finite subgroup of $\cal{D}$ and $\varphi$ is an orbifold isomorphism $\varphi: Y \to S/ \Sigma$. We then write $Y \simeq [S/ \Sigma]$, with $\varphi$ ignored. 
\end{defn}

\begin{rem}\label{rem:OrbifoldComplexStructure}
If $X_J=(S,J)$ is a Riemann surface and $\Sigma$ acts on $X_J$ by holomorphic or anti-holomorphic maps, then $Y=Y_J$ inherits the ``complex structure'' from $X_J$ and we denote it as $Y \simeq [X_J/ \Sigma]$. For nonorientable orbifolds, these are called \emph{orbifold dianalytic structures}. Precisely, an orbifold dianalytic structure on $Y$ is an orbifold complex structure on its orientable  double cover $Y^+$ with $\bb{Z}/2\bb{Z}$ action given by an anti-holomorphic involution, see \cite[Section 5.1]{Orbifolds}. 
\end{rem}

\subsubsection{Hitchin representations for orbifolds}

Thurston studied the space of hyperbolic structures on a closed $2$-orbifold $Y$ of negative Euler characteristic \cite[Chapter 13]{ThurstonBooks}. This is called the \emph{Teichm{\"u}ller space} of $Y$, denoted as $\mathcal{T}(Y)$. Similarly to the case of closed surfaces, by taking the holonomy representations of hyperbolic structures on $Y$, this space of hyperbolic structures on $Y$ is identified with a connected component of the representation space 
$$\textnormal{Rep}(\pi_1 Y,\mathrm{PGL}(2,\mathbb{R})):= \textnormal{Hom}(\pi_1 Y,\mathrm{PGL}(2,\mathbb{R}))/\mathrm{PGL}(2,\mathbb{R}).$$
Similarly to the closed surface case, one can define \emph{Fuchsian representations} and the \emph{Fuchsian locus} for orbifolds using the principal representation
$\kappa: \mathrm{PGL}(2,\mathbb{R}) \to \mathrm{PGL}(n,\mathbb{R})$. A (conjugacy class of) representations $[\rho]: \pi_1 Y \to \mathrm{PGL}(n,\mathbb{R})$ is called a (conjugacy class of) \emph{Fuchsian representations} if there exists $[\rho_0]\in \mathcal{T}(Y)$ such that $[\rho]=[ \kappa\circ \rho_0]$. The space of (conjugacy class of) Fuchsian representations is called the \emph{Fuchsian locus} of $\textnormal{Rep}(\pi_1 Y,\mathrm{PGL}(n,\mathbb{R}))$.

The \emph{Hitchin component} of the orbifold $Y$ is defined using Fuchsian representations.

\begin{defn}
The \emph{Hitchin component} of $Y$, denoted as $\textnormal{Hit}(\pi_1 Y, \mathrm{PGL}(n,\mathbb{R}))$,  is the connected component of $\textnormal{Rep}(\pi_1 Y,\mathrm{PGL}(n,\mathbb{R})):= \textnormal{Hom}(\pi_1 Y,\mathrm{PGL}(n,\mathbb{R}))/\mathrm{PGL}(n,\mathbb{R})$ that contains the Fuchsian locus.  An element in $\textnormal{Hit}(\pi_1 Y, \mathrm{PGL}(n,\mathbb{R}))$ is called a Hitchin representation.
\end{defn}

\begin{rem}[{\cite[Remark 2.5]{Orbifolds}}] \label{rem PGLvsPSL}
If $Y$ is orientable (for instance if $Y=S$ is a closed orientable surface), then any Fuchsian representation of $\pi_1(Y)$ is in fact contained in $\mathrm{Hom}(\pi_1 Y , \mathrm{PSL}(n,\mathbb{R}))$. It may happen that there are two Hitchin components (for example, when $n$ is even) if we consider such representations up to $\mathrm{PSL}(n,\mathbb{R})$-conjugacy. These representations in  two connected components are  related by an inner  automorphism of $\mathrm{PGL}(n,\mathbb{R})$. Therefore $\mathrm{PGL}(n,\mathbb{R})$-conjugacy identifies these components. On the other hand, when $Y$ is nonorientable, the images of Fuchsian representations of $\pi_1(Y)$ are in $\mathrm{PGL}(n,\mathbb{R})$ (and may not be able to be restricted to $\mathrm{PSL}(n,\mathbb{R})$). 
\end{rem}

\section{Mapping class group actions}\label{sec: MappingClassGrp}
We give an exposition on how the extended mapping class group acts on various mathematical objects.

\subsection{Extended mapping class group action on Riemannian metrics}
\label{subsec: MappingClassGrpMetrics}

Let $S$ be a closed orientable smooth surface of genus $\mathscr{G}\geq 2$ as in Section \ref{subsec: teich}. Recall that
  $\mathcal{M}_-$ denotes the space of negatively curved smooth Riemannian metrics on $S$, on which the diffeomorphism groups $\cal{D}$, $\mathcal{D}^+$ (when $S$ is oriented), and $\cal{D}_0$ act by pullback.
The \emph{extended mapping class group} is given by the quotient group 
$\textnormal{Mod}^{\pm}(S):= \mathcal{D}/ \mathcal{D}_0$.
Two smooth diffeomorphisms $f_1$ and $f_2$ represent the same point in $\textnormal{Mod}^{\pm}(S)$ if and only if  $f_1$ is smoothly isotopic to $f_2$. In other words, the extended mapping class group $\textnormal{Mod}^{\pm}(S)$ is the group of isotopy classes of elements in $\cal {D}$.
When $S$ is given an orientation, we also denote by $\mod= \mathcal{D}^+/ \mathcal{D}_0$ the 
\emph{mapping class group} which is the group of isotopy classes of all orientation-preserving smooth diffeomorphisms of $S$. The mapping class group $\textnormal{Mod}(S)$ is an index 2 subgroup of $\textnormal{Mod}^{\pm}(S)$. Given $\psi\in\cal {D}$, we denote by $[\psi]$ the induced element in $\textnormal{Mod}^{\pm}(S)$. 
The action of $\cal D$ on $\mathcal{M}_-$ by pullback induces an action of $\textnormal{Mod}^{\pm}(S)$ on the quotient space $\mathcal{M}_-/ \mathcal{D}_0$. 
This space  will be discussed further in Section \ref{subsec: negMetrics}. 
Note that this action of $\textnormal{Mod}^{\pm}(S)$ is a right action on $\mathcal{M}_-/ \mathcal{D}_0$. To be consistent with the outer automorphism group action introduced in the next subsection, people also often use the induced left action of $\textnormal{Mod}^{\pm}(S)$ on $\mathcal{M}_-/ \mathcal{D}_0$ given by $[\psi]\cdot [g]=[(\psi^{-1})^*g]$. 

\subsection{Extended mapping class group action on Hitchin components}\label{subsec:MCGHitchin}

\subsubsection{Outer automorphism group}\label{subsec outer auto} 
The \emph{outer automorphism group} $\textnormal{Out}(\pi_1 S)$ is defined as the quotient
$$\textnormal{Out}(\pi_1 S)= \textnormal{Aut}(\pi_1 S)/ \textnormal{Inn}(\pi_1 S),$$
where $\textnormal{Aut}(\pi_1 S)$ is the group of automorphisms of $\pi_1 S$ and $\textnormal{Inn}(\pi_1 S)$ denotes the group of all \emph{inner automorphisms}:  for any $h\in \pi_1 S$, the associated inner automorphism is defined by
\begin{align*}
    I_h: \pi_1 &S \to \pi_1 S\\
    &g \mapsto hgh^{-1}.
\end{align*}
By the Dehn-Nielsen-Baer Theorem \cite[Theorem 8.1]{FarbMappingClassGroup}, the extended mapping class group $\textnormal{Mod}^{\pm}(S)$ is isomorphic to the outer automorphism group $\textnormal{Out}(\pi_1 S)$. There is a natural left action of $\textnormal{Out}(\pi_1 S)$ on $\mathcal{H}_n(S)$. Given a representation $\rho\in \mathcal{H}_n(S)$ and any $\psi \in \textnormal{Out}(\pi_1 S)$,  define
$$ \psi \cdot \rho= \rho \circ \psi^{-1}.$$
This action preserves the Hitchin component $\mathcal{H}_n(S)$ (see  the paragraph after Lemma 2.8 in \cite{Orbifolds}). As the extended mapping class group $\textnormal{Mod}^{\pm}(S)$ is identified with the group $\textnormal{Out}(\pi_1 S)$, we obtain a natural action of $\textnormal{Mod}^{\pm}(S)$ on the Hitchin component $\mathcal{H}_n(S)$. In particular, this action on $\mathcal{H}_2(S)=\mathcal{T}(S)$ corresponds to the left extended mapping class group action on $\mathcal{M}_{-1}/\mathcal{D}_0$ by (inverse)  pullback defined in Section \ref{subsec: MappingClassGrpMetrics}.

\subsubsection{Outer automorphism group and orbifolds}

In Section \ref{orbifolds}, we presented a 2-dimensional closed connected smooth orbifold $Y$ of negative Euler characteristic as a quotient of a closed orientable surface $S$ by a finite subgroup $\Sigma\leq \cal{D}$. This implies the existence of a short exact sequence:
$$1\to \pi_1 S \to \pi_1 Y \to \Sigma \to 1.$$
In particular, $\pi_1 S$ is a normal subgroup of $\pi_1 Y$ of finite index  and $\Sigma \simeq \pi_1 Y/ \pi_1 S$.

The finite group $\Sigma\leq \cal{D}$ yields a subgroup $\Sigma\leq\textnormal{Mod}^{\pm}(S)$ which is isomorphic to a subgroup of $\textnormal{Out}(\pi_1 S)$ by the Dehn-Nielsen-Baer Theorem.
Because  $\textnormal{Out}(\pi_1 S)$ acts on the Hitchin component $\mathcal{H}_n(S)$ from the left by (inverse)  precomposition, one obtains an action of $\Sigma$ on $\mathcal{H}_n(S)$. We will denote by $\textnormal{Fix}_{\Sigma} \mathcal{H}_n(S)$ the fixed locus of the action $\Sigma$.
The following theorem from \cite{Orbifolds} about the relation between Hitchin representations for orbifolds and the outer automorphism group action on $\cal{H}_n(S)$ will be important later:

\begin{thm}[{\cite[Theorem 2.12]{Orbifolds}}] \label{thm,orbifoldLocus}
Given a closed connected 2-orbifold of negative Euler characteristic $Y$ and a presentation $Y\simeq [S/\Sigma]$, the map $\rho \mapsto \rho|_{\pi_1 S}$ induces a homeomorphism $j: \textnormal{Hit}(\pi_1 Y, \mathrm{PGL}(n, \mathbb{R})) \to \textnormal{Fix}_{\Sigma}\mathcal{H}_n(S)$ between $\textnormal{Hit}(\pi_1 Y, \mathrm{PGL}(n, \mathbb{R}))$ and the $\Sigma$-fixed locus in $\mathcal{H}_n(S)$.
\end{thm}

\subsection{Extended mapping class group action on holomorphic differentials} \label{MappingClassHolomorphicDiff}

In this subsection, we first explain how the extended mapping class group acts  on the space of sections of holomorphic differentials in general. Then we focus on surface diffeomorphisms that are holomorphic or antiholomorphic with respect to a certain complex structure, and discuss the extended mapping class group action they induce. This will naturally lead to an exposition on the \emph{equivariant structure of Hitchin fibration} from \cite[Section 4.1]{Orbifolds}.

Given a complex structure $J \in C^\infty(S;\textnormal{End}(TS))$, a diffeomorphism $\psi\in \cal{D}$ acts on $J$ from the right as: $(\psi^* J)_{x}= (d\psi^{-1})_{\psi(x)} \circ J_{\psi(x)} \circ d\psi_{x}$. In particular, we say $\psi$ is holomorphic with respect to $J$ if $\psi^* J=J$; We say $\psi$ is anti-holomorphic with respect to $J$ if $\psi^* J=-J$. 
Upon taking a quotient by $\mathcal{D}_0$ action, one obtains an action of the extended mapping class group $\textnormal{Mod}^{\pm}(S)= \mathcal{D}/ \mathcal{D}_0$ on isotopy classes of complex structures.

The action of $\psi \in \mathcal{D}$ naturally induces a left action of $\psi$ on all powers of the canonical bundles $K^d_J$, denoted by $\kappa_\psi$.
Let $X_J=(S,J)$ be the Riemann surface with the complex structure $J$. We can define an action of $\psi$ on a holomorphic section $s\in H^0(X_J, K^d_J)$ by $\psi^* s:= {\kappa_{\psi}}^{-1} \circ s \circ \psi$ as illustrated by the following diagram: 
    \[ \begin{tikzcd}
K_{\psi^*J}^{d} \arrow{r}{\kappa_{\psi}}  & K_{J}^d    \\%
(S, \psi^*J) \arrow{r}{\psi} \arrow[swap]{u}{\psi^* s}&  (S,  J) \arrow[swap]{u}{s}
\end{tikzcd}
\]
More explicitly, for $p\in S$ and $v\in (T X^{\mathbb{C}}_{\psi^*J})^{(1,0)}$, we have 
$$(\psi^* s)(p)\big(v,\cdots, v\big)=s(\psi(p))\big(d\psi (v),\cdots, d\psi (v)\big).$$
Here $d\psi$ denotes the complexified differential $d\psi: T_p X_{\psi^* J}^{\mathbb{C}} \to T_{\psi(p)} X_{J}^{\mathbb{C}}$. 
The pull back $\psi^* s$ is a holomorphic section on $H^{0}(X_{\psi^* J}, K_{\psi^* J}^d)$. 
After taking $\mathcal{D}_0$ quotient, we obtain a right  $\textnormal{Mod}^{\pm}(S)$ action on ($\mathcal{D}_0$-equivalence classes of) holomorphic $d$-differentials. Often, we also use the induced left action by inverse pull back given by $[\psi ]\cdot [s]= [(\psi^{-1})^* s]$.

\smallskip

Now we focus on diffeomorphisms that are either holomorphic or antiholomorphic with respect to a fixed complex structure $J$ and explain their actions on holomorphic differentials. Suppose that $\Sigma$ is a finite subgroup of $\cal{D}$.  We fix an orientation for $S$ and a $\Sigma$-invariant Riemannian metric $g$ on $S$. Denote by 
$J=J_{g}$ the complex structure associated to $g$.
Then a map $\psi\in \Sigma$ is holomorphic with respect to $J$ if it preserves the orientation of $S$; otherwise, it is anti-holomorphic. The bundle isomorphism $\kappa_{\psi}$ from $K_{\psi^*J}$ to $K_{J}$ defined before induces a bundle automorphism $\tau_{\psi}:K_J \to K_J$. More explicitly, when $\psi$ is orientation preserving, we let $\tau_{\psi}= \kappa_{\psi}$, which is clearly a bundle automorphism. 
When $\psi$ is orientation reversing, given $\xi=(p,w)\in K _{J}^{d}|_{p}$, we let $\tau_{\psi} \xi:= \overline{\kappa_{\psi}\xi}$.
Since $\kappa_{\psi}$ maps the bundle $K_{\psi^*(-J)}$ to $K_{-J}$ and $\psi^*(-J)=J$, it is also clear in this case that $\tau_{\psi}:K_J \to K_J$ is a bundle automorphism. This introduces an action of $\tau_{\psi}$ on all tensor powers $K^d_J$. The definition of the $\tau_\psi$ action here then coincides with the $\tau_\psi$ action in \cite[Section 4.1]{Orbifolds}.

The family $\tau=(\tau_\psi)_{\psi\in \Sigma}$ forms a \emph{$\Sigma$-equivariant structure} of the holomorphic  bundle $K^d_J$ in the following sense:
\begin{enumerate}
    \item  For all $\psi \in \Sigma$, the following diagram commutes (i.e.  $\pi\circ \tau_\psi= \psi \circ \pi$),
    \[ \begin{tikzcd}
K_J^{d} \arrow{r}{\tau_{\psi}} \arrow[d,"\pi", xshift=0.1ex] & K_J^{d}  \arrow[d,"\pi", xshift=-0.7ex] \\%
(S, J) \arrow[u,"s",bend left=10, dashed, xshift=-1.3ex] \arrow{r}{\psi}&  (S, J)\arrow[u,"\psi^A \cdot s" right, bend right=10, dashed, xshift=1.3ex]
\end{tikzcd},
\]
where $\pi:K_J^{d} \to (S,J) $ is the projection. (The map $\psi^A$ will be explained in the next paragraph.)
    \item The bundle map $\tau_\psi$ is fiberwise $\mathbb{C}$-linear if $\psi$ is holomorphic with respect to $J$ and fiberwise $\mathbb{C}$-antilinear if $\psi:X \to X$ is anti-holomorphic with respect to $J$.
    \item $\tau_{id}= \textnormal{Id}_{K^d_J}$ and $\tau_{\psi_1\psi_2}=\tau_{\psi_1} \tau_{\psi_2}$.
\end{enumerate}
Note that since  
$\psi$ and $\tau_{\psi}$ are simultaneously holomorphic or anti-holomorphic, the composition $ \tau_{\psi} \circ s \circ \psi^{-1}$ is again a holomorphic section in $H^{0}(X_{J}, K_{J}^d)$. We will denote  this (left) action by $ \psi^{A} \cdot s:= \tau_{\psi} \circ s \circ \psi^{-1}$ to emphasize that this is an automorphism and distinguish it from the pull back action $\psi^{*} s$ and its induced left action $\psi\cdot s=(\psi^{-1})^*s$ by inverse pull back.

In general, suppose $X_J$ is a Riemann surface with complex structure $J$ so that the finite group $\Sigma$ acts on $X_J$ by holomorphic or anti-holomorphic maps. Then $\psi\in \Sigma$ act on $H^{0}(X_{J}, K_{J}^d)$ as an automorphism $\psi^A:H^{0}(X_{J}, K_{J}^d) \to H^{0}(X_{J}, K_{J}^d) $.  The following Proposition in \cite{Orbifolds} will be important later.

\begin{prop}[{\cite[Lemma 4.3]{Orbifolds}}]
\label{SigmaEquiHitchinSection}
 Under the above assumptions, the Hitchin parametrization $H_{J}:\bigoplus\limits_{i=2}^{n}H^0(X_{J}, K^i_{J}) \to \mathcal{H}_{n}(S) $ is $\Sigma$-equivariant and induces a homeomorphism $$\textnormal{Fix}_{\Sigma}\big(\bigoplus\limits_{i=2}^{n}H^0(X_{J}, K^i_{J})\big) \simeq \textnormal{Fix}_{\Sigma}\mathcal{H}_{n}(S)$$ for any integer $n\geq 2$.
\end{prop}

\section{Covariance metric on the space of negatively curved Riemannian Metrics}\label{sec:PressureMetric}

In this section, which follows closely \cite[Section 2]{GeodesicStretch}, we will define the covariance metric introduced there on the space of negatively curved metrics. The material here works for $n$-dimensional Riemannian manifolds with Anosov geodesic flows, though we only discuss it in the setting that we are interested in, namely that of a closed orientable smooth surface $S$ with genus $\mathscr{G}$ at least $ 2$.

 \subsection{Function Spaces}
If $M$ is a closed manifold, we will denote by $\mathcal{D}'(M):=(C^\infty(M))'$ the space of distributions, and by $\cal{H}^s(M)$ the Sobolev space of order $s\in \bb{R}$.
The latter can be defined as 
\begin{equation*}
    \cal{H}^s(M) :=\{u \in \mathcal{D}'(M)\textnormal{ }|\textnormal{ } (1-\Delta_{{g}_0})^{s/2}u\in L^2_{g_0}(M) \textnormal{ } \},    
\end{equation*}
where $\Delta_{g_0}$ denotes the negative Laplace-Beltrami operator of a fixed Riemannian metric $g_0$ and $L^2_{g_0}(M)$ is the space of square integrable functions with respect to the probability measure  induced by the volume density $\rmd v_{g_0}$ of $g_0$, i.e. with respect to $\frac{\rmd v_{g_0}}{\mathrm{Vol}(M,g_0)}$.
An alternative useful characterization of $\cal{H}^k(M)$ for integer $k\geq 0$  is given by
\begin{equation*}
    \cal{H}^k(M)=\{u\in \cal {D}'(M)| V_1\cdots V_mu\in  L^2_{g_0}(M), \quad 0\leq m\leq k, \quad   \text{ for any }  V_j\in \mathfrak{X}(M) \},
\end{equation*}  
where $\mathfrak{X}(M)$ denotes smooth vector fields on $M$.
A choice of $g_0$ induces an inner product
\begin{equation*}\label{SobolevInner_product}
	\langle u,v\rangle_{\cal{H}^s_{g_0}(M)}=\langle (1-\Delta_{g_0})^{s/2} u, (1-\Delta_{g_0})^{s/2}v\rangle_{ L^2_{g_0}(M)}=\langle (1-\Delta_{g_0})^{s} u, v\rangle_{ L^2_{g_0}(M)},
\end{equation*}
making  $ \cal{H}^s(M)$ into a Hilbert space (the last equality follows by self-adjointness). 
In our applications, $M$ will typically be either $S$ or $T^1S_g$, where 
$T^1S_g$ is the unit tangent bundle with respect to a Riemannian metric $g$ on $S$.
Whenever we write $L^2(T^1S_g)$, it will be understood that the measure used to define the $L^2$ inner product and norm is the Liouville measure of $g$ normalized to have total mass 1,  denoted by $\mu_g^L$.
Notice that if $f\in C^\infty(S)$ and $\pi_0:T^1S_g\to S$ is the natural projection we have $$\int_{T^1S_g}\pi_0^*f \rmd\mu_{g}^L= \frac{1}{\mathrm{Area}(S,g)} \int_S  f \rmd v_g.$$

For $k\in \bb{N}_0$ and  $\alpha\in (0,1)$, we will also make use of H\"older spaces
\begin{equation*}
	C^{k,\alpha}(M):=\{u\in  C^k(M)| V_1\cdots V_k u\in  C^{\alpha}(M), \text{ for any }  V_j\in \mathfrak{X}(M) \}
\end{equation*}
where
\begin{equation}
    C^{\alpha}(M)=C^{0,\alpha}(M)=\big\{u\in C^0(M)\;|\; \sup_{x\neq y} \frac{|u(x)-u(y)|}{\textrm{dist}_{g_0}(x,y)^\alpha}<\infty\big\}.
\end{equation}
Upon fixing a norm, $C^{k,\alpha}(M)$ becomes a Banach space.
Spaces of sections of smooth vector bundles with Sobolev or H\"older regularity are defined using local trivializations.
We remark that for  both Sobolev and Hölder spaces, a different choice of background metric $g_0$ does not change the spaces themselves, and different choices of metrics result in equivalent norms.

\smallskip

\subsection{\texorpdfstring{The $\Pi^g$  and Variance operators}{The Pi  and Variance operators}} \label{subsec: Pi}
 
For the rest of this discussion, fix a negatively curved metric $g$ on $S$.
In \cite{guillarmou_anosov}, an operator $\Pi^g: \cal{H}^s(T^1S_{g}) \to \cal{H}^{-s}(T^1S_{g})$ is constructed using microlocal tools. When restricted to $f\in C^{\infty}(T^1S_g)$ satisfying the mean zero property (that is, $\langle f, 1\rangle_{L^2(T^1S_{g})} =0$), it is given by

\begin{equation}
    \begin{aligned}
        \Pi^g:\;&C^{\infty}(T^1S_g)
    \to \mathcal{D}'(T^1S_g),\nn
    \\
    \langle \Pi^g f, f'\rangle &= \lim\limits_{T\to\infty} \int_{-T}^{T}\langle  f\circ \phi_t, f'\rangle_{L^2(T^1S_{g})} \mathrm{d}t. 
    \end{aligned}
\label{eq:Pi}
\end{equation}
Here $\phi_t$ is the geodesic flow of $g$ on $T^1S_g$ and the convergence of the integral on the right hand side is guaranteed by exponential decay of correlations for $\phi_t$ (see \cite{Liverani}).

From another perspective, $\Pi^g$ is related to the \emph{Variance} and \emph{Covariance}, arising from the Thermodynamic formalism.
\begin{defn}  \label{def,Variance}
The variance of $f\in C^{\infty}(T^{1}S_g)$  which is of mean zero with respect to $\mu^L_g$ is defined as
\begin{equation}
    \mathrm{Var}(f, \mu^L_g)= \lim_{T\to\infty} \frac{1}{T} \int_{T^{1}S_g}  \left( \int_{0}^{T} f(\phi_t(x))\mathrm{d}t \right) ^2 \mathrm{d} \mu^L_g(x)
    \label{eq:Var}
\end{equation}
and the covariance of two $\mu^L_g$-mean zero functions $f_1,f_2\in C^{\infty}(T^{1}S_g)$ is given by
\begin{equation}
    \mathrm{Cov}(f_1,f_2, \mu^L_g)= \lim_{T\to\infty} \frac{1}{T} \int_{T^{1}S_g}  \left( \int_{0}^{T} f_1(\phi_t(x))\mathrm{d}t \right) \left( \int_{0}^{T} f_2(\phi_t(x))\mathrm{d}t \right) \mathrm{d} \mu^L_g(x).
    \label{eq:CoVar}
\end{equation}
\end{defn}
\noindent A crucial fact is that if $f\in C^{\infty}(T^{1}S_g)$ is of mean zero, one can use the $\phi_t$-invariance of the Liouville measure $\mu^L_g$  to obtain
\begin{equation}
    \langle \Pi^g f, f\rangle=\mathrm{Var}(f, \mu^L_g);\label{var-Pi}
\end{equation}
A proof can be found in   \cite[Proposition 4]{pollicott_DerivativeOfEntropy}.
Note that \eqref{var-Pi} is always nonnegative.
Similarly, one can check that $\langle \Pi^g f_1, f_2\rangle=\mathrm{Cov}(f_1,f_2, \mu^L_g)$ if $f_1$, $f_2$ are of mean zero. %

The following Theorem is proved by Guillarmou in \cite{guillarmou_anosov}. We state it in our setting:

\begin{thm} [{\cite[Theorem 1.1]{guillarmou_anosov}, \cite[Section 4A]{Thibault_MicrolocalAnalysis}}] \label{thm,PiProperty}
For all $s>0$, the operator $\Pi^g: \cal{H}^s(T^1S_{g}) \to \cal{H}^{-s}(T^1S_{g})$ is bounded and self-adjoint and satisfies the following properties:
\begin{enumerate}
    \item $\Pi^g$ is positive in the sense that $\langle \Pi^g f, f\rangle \geq 0$ for all real-valued $f \in \cal{H}^s(T^1S_{g})$.
    \item If $f$ and $X^gf$ belong to $\cal{H}^s(T^1S_g)$ , then $\Pi^g X^gf=0$, where $X^g$ is the geodesic vector field on $T^1S_g$. 
    \item If $f\in \cal{H}^s(T^1S_g)$ with $\langle f,1 \rangle_{L^2(T^1S_g)}=0$, then $\Pi^g f=0$ if and only if there exists a solution $w\in \cal{H}^s(T^1S_g)$ to the cohomological equation $X^gw=f$, and $w$ is unique modulo constants.  
    \item  $\Pi^g 1 = 0$.
    \item \label{thm4.2it5}  If $f\in \cal{H}^s(T^1S_{g})$, then  $\langle \Pi^g f,1\rangle=0$. 
\end{enumerate}
\end{thm}

These properties of $\Pi^g$ are important for introducing the covariance metric from \cite[Section 2]{GeodesicStretch}. In particular, the fact that $\Pi^g 1 = 0$ allows one to use \eqref{eq:Pi} to make sense of $\Pi^g$ for all $C^{\infty}(T^{1}S_g)$: once one notices that
$\Pi^gf=\Pi^g (f-\langle f, 1\rangle_{L^2(T^1S_{g})})$, the following is immediate. 
\begin{lem} \label{defn,PiDefn}
The operator $\Pi^g:C^{\infty}(T^1S_g)
\to \mathcal{D}'(T^1S_g) $ is given by
\begin{equation*}
 \langle \Pi^g f, f'\rangle := \lim\limits_{T\to\infty} \int_{-T}^{T}\langle P_{g}(f)\circ \phi_t, f'\rangle_{L^2(T^1S_{g})} \mathrm{d}t, 
\end{equation*}
where $P_{g}: C^{\infty}(T^1S_g) \to C^{\infty}(T^1S_g)$ is the projection operation defined by 
$$P_{g}(f):=f -\langle f, 1\rangle_{L^2(T^1S_{g})}.$$
\end{lem}

\subsection{\texorpdfstring{Symmetric tensors on a surface $S$}{Symmetric tensors on surface S}}\label{subsec,symmetricTensors}

We denote by $\mathsf{S}_m(S)$  the space of
smooth  symmetric $m$-tensors on $S$ (and we use $\sf{S}^{k,\alpha}_m(S)$ when we need $C^{k,\alpha}$ regularity).
If $f\in \sf{S}_m(S)$ and $m\in \mathbb{N}_0=\{0,1,2,\dots\}$, we denote by $\pi_m^*f\in C^{\infty}(TS)$ the map $\pi_m^*f(x,v):=f_x(v,\cdots, v)$. So $\pi_m^*$ converts a symmetric $m$-tensor to a function on the tangent bundle $TS$. 
A Riemannian metric $g$ naturally induces a scalar product $\langle \cdot, \cdot \rangle$ on $\mathsf{S}_m(S)$ (see for example \cite[Section 2A1]{Thibault_MicrolocalAnalysis}). The formal adjoint of $\pi_m^*$ is given by declaring
\begin{align*}
\langle f, \pi_{m*}h \rangle_{L^2_{g}(S)}= \langle \pi_m^*f, h \rangle_{L^2(T^1S_{g})},
\end{align*}
where $f\in \sf{S}_m(S)$ and $h\in C^{\infty}(T^1S_g)$. 

Fix a smooth Riemannian metric $g$ on $S$ for the rest of this discussion. 
We denote by $D_g$ the symmetrization of the covariant derivative with respect to the Levi-Civita connection $\nabla_g$. One has the following relation \cite[Lemma 2.3]{Thibault_MicrolocalAnalysis} between the geodesic vector field $X^g$ on $T^1 S_g$ and the operator $D_g$, 
\begin{equation}\label{eqtn,XandD}
X^g \pi^*_m=\pi^*_{m+1}D_g.
\end{equation}
The formal adjoint of $D_g$ is the negative of the divergence operator on symmetric $m$-tensors, i.e. $D_g^*:=- \Tr_g\circ \nabla_g:\sf{S}_m(S)\to \sf{S}_{m-1}(S)$.
Note that the negative Laplace-Beltrami operator on functions is given by $\Delta_g=-D^*_gD_g$.  %\noteD{I am confused again... Is it postive or negative for geometric laplacian?}
%\noteN{What people usually call the geometric Laplacian is the one we have in \eqref{eq:laplace_in_coords}, so no minus. Here we wrote $D^*$ for the formal adjoint of $D$, so this comes with a $-$ because of the integration by parts. Then you include one more minus to get the geometric Laplacian. Everywhere in the paper we used the geometric Laplacian I think.
%}

We will be mostly interested in symmetric 2-tensors, either smooth or of Hölder regularity, and for those, the following $L^2$-orthogonal decompositions will be useful.
Any tensor $f\in \sf{S}_{2}^{k,\alpha}(S)$ can be decomposed uniquely  as a sum
\begin{equation}\label{eqtn, Divergences}
  f= D_g \chi + v,  
\end{equation}
where $\chi \in \sf{S}_{1}^{k+1,\alpha}(S)$
and $v\in  \sf{S}_{2}^{k,\alpha}(S)$ satisfies $D_g^*v=0$. 
The component $D_g \chi$ is often called the \emph{potential part} of $f$, whereas the divergence free component $v$ is called \emph{solenoidal}.
 Another helpful decomposition of $f\in \sf{S}_{2}^{k,\alpha}(S)$ is the one into a \emph{conformal} and a \emph{trace free} part with respect to $g$: 
\begin{equation}\label{eqtn,traceDecomp}
 f=f_1+f_2,   
\end{equation}
where  $f_1=\frac{1}{2}(\Tr_g f) g$ is conformal to $g$ and $f_2= f-\frac{1}{2}(\Tr_g f)  g$ is trace free.

\subsection{\texorpdfstring{The space $\mathcal{M}_-/\mathcal{D}_0$ }{The space M-/D0 }}\label{subsec: negMetrics}

Here we summarize some useful facts found in \cite[Section 2.3]{GeodesicStretch}.
Recall that in Section \ref{subsec: teich}, we defined $\mathcal{M}$  as the space of smooth Riemannian metrics on a closed orientable smooth surface $S$ of genus $\mathscr{G}\geq 2$ and $\mathcal{M}_-$ as the open subspace (in the $C^\infty$ topology) of negatively curved Riemannian metrics. The space $\mathcal{M}$  is a smooth Fr{\'e}chet  manifold whose tangent space at $g\in \mathcal{M}$ can be naturally identified with the space $\sf{S}_{2}(S)$ of smooth symmetric $2$-tensors. Since $\cal M_-$ is an open subset of $\cal M$ in the $C^\infty$ topology, it has the same tangent space. The group $\cal{D}_0$ of smooth diffeomorphisms isotopic to the identity is a Fr{\'e}chet Lie group (\cite[Section 4.6]{Hamilton}) and acts on $\mathcal{M}$ on the right by pull back. This action is smooth and proper on $\cal{M}$ (\cite{EbinRiemannianMetrics1}, \cite{EbinRiemannianMetrics2}). 
Moreover, for negatively curved metrics, this action is free \cite{FrankelFreeAction}.
By Ebin's slice theorem  (see \cite{EbinRiemannianMetrics2}, \cite{slice_survey}),    for any ${g_0}
\in \mathcal{M}_-$, there exists a neighborhood $\mathcal{U}$ of ${g_0}$ in $\mathcal{M}_-$, a neighborhood $\mathcal{V}$ of
$ \textnormal{Id} \in \mathcal{D}_0$, and a Fr{\'e}chet submanifold $\mathcal{W}$ of $\mathcal{M}_-$ containing ${g_0}$ such that
\begin{equation}\label{eq:slice}
    \begin{aligned}
    &\mathcal{W} \times \mathcal{V}  \to \mathcal{U} \\
    & (g,\psi) \mapsto \psi^*g
    \end{aligned}
\end{equation}
 is a diffeomorphism of smooth  Fr{\'e}chet manifolds and such that
\begin{equation}
 T_{g_0}\mathcal{W}= \{v\in  \sf{S}_{2}(S)| D_{g_0}^*v=0 \}.\label{tangent_space} 
\end{equation}
Because the tangent space of the orbit space $g \cdot \mathcal{D}_0\subset \mathcal{M}_- $ at $g$ consists of elements of the form $\mathcal{L}_Y g$ for $Y\in \mathfrak{X}(S)$, the fact that $\frac{1}{2}\cal L_Yg=D_g(Y^\flat)$ implies 
that it consists exactly of the potential symmetric 2-tensors.
Therefore, $T_{g_0}\cal W$ and $T_{g_0}(g\cdot \cal D_0)$ are mutually orthogonal with respect to the $L_{g_0}^2$ inner product. For $g$ near $g_0$, one has $T_g\cal W\cap T_g(g\cdot \cal D_0)=\{0\}$.
Since the action of $\cal D_0$ on $\mathcal{M}_-$ is proper and free, together with the diffeomorphism property of \eqref{eq:slice}, a neighborhood of $[g]$ in $\mathcal{M}_-/\cal D_0$ can be identified with a neighborhood of $g$ in $\cal W$. The set $\mathcal{M}_-/\cal D_0$ therefore inherits from $\cal W$ the structure of a Fr\'echet manifold with its tangent space at $[g_0]$ identified with \eqref{tangent_space}.

We will also make use of the space $\cal{M}_-^{k,\alpha}$ of $C^{k,\alpha}$ negatively curved metrics, where $k\in \mathbb{N}$ is large and $\alpha\in (0,1)$, which is a Banach manifold with $T_g\cal{M}_{-}^{k,\alpha}=\sf{S}_2^{k,\alpha}(S)$ at each $g$. 
The group $\cal{D}_0^{k+1,\alpha}$ of $C^{k+1,\alpha}$ diffeomorphisms isotopic to the identity acts continuously, freely and properly on $\cal{M}_-^{k,\alpha}$, but not smoothly (see \cite{Tromba_Book}).
Thus  we cannot use the quotient manifold theorem to give $\cal{M}_-^{k,\alpha}/\cal{D}_0^{k+1,\alpha}$ the structure of a smooth manifold.
However, we note that the $\cal{D}_0^{k+1,\alpha}$ orbit through any $C^\infty$ metric in $\cal{M}^{k,\alpha}_-$ is a smooth submanifold of the latter with tangent space at $g_0$ consisting of the $C^{k,\alpha}$ potential symmetric 2-tensors with respect to $g_0$.

With all these understood, we can proceed to introduce the covariance metric on the space of negatively curved metrics in the next subsection.

\subsection{The covariance metric on \texorpdfstring{$\mathcal{M}_-/\mathcal{D}_0$}{M-/D}} \label{subsec: PressureMetric}

To construct the covariance metric (\cite{GeodesicStretch}) on $\mathcal{M}_-/\mathcal{D}_0$, we first produce a bilinear form $G$ on $\mathcal{M}_-$ and on $\mathcal{M}_-^{k,\alpha}$. It might be tempting to think that $\pi_{2 *}\circ\Pi^g \circ\pi_2^* $  induces a positive definite symmetric bilinear form on $T_g(\mathcal{M}_-/ \mathcal{D}_0)$ by Theorem \ref{thm,PiProperty}.
However, even though $\Pi^g$ is positive (Theorem \ref{thm,PiProperty}),  positive definiteness of the bilinear form associated with $\pi_{2 *}\circ\Pi^g \circ\pi_2^* $ does not follow (see Remark \ref{rem,ConformalDegenerate}). 
To tackle this problem, the authors in \cite{GeodesicStretch} consider the
operator 
\begin{equation}
\begin{aligned}
  \Pi^g_m:&  \sf{S}_m(S)\to \sf{S}_m^{-\infty}(S), \\
    \Pi^g_m:&=\pi_{m*}(\Pi^g+ \textbf{1}\otimes\textbf{1} )\pi_{m}^*,
\end{aligned}
\end{equation}
where the operator $\textbf{1}\otimes\textbf{1}:C^{\infty}(T^1S_g)\to C^{\infty}(T^1S_g)$ projects a function $f\in C^{\infty}(T^1S_g)$ onto its mean $\langle f, 1\rangle_{L^2(T^1S_{g})}$ and $\sf{S}_m^{-\infty}(S):=(\sf{S}_m(S))'$.  
The operator $\Pi_m^g$ can be  thought of as an analog of the \emph{normal operator of the X-ray transform}, defined on a compact manifold with boundary having sufficiently good geometric properties, which averages the X-ray transform of a tensor field over all geodesics passing through a point (see for example \cite{paternain2023}). On the closed surface $(S,g)$, the \emph{X-ray transform} is defined as 
\begin{equation}\label{xray}
    I_m^g:\sf{S}_m(S)\to \ell^\infty(\cal{C}),\quad f\mapsto I_m^gf(c)=\frac{1}{L(c)} \int_0^{L(c)}\pi_m^*f (\phi_t(z))\mathrm{d}t,
\end{equation}
where $\cal{C}$ denotes the set of closed orbits of the geodesic flow and in \eqref{xray} a closed orbit $c$ of primitive length $L(c)$ is parametrized as $\phi_t(z)$, where $t\in [0,L(c)]$ and $z\in c$.
The fact that on $S$ the space of all geodesics does not have a manifold structure necessitates the more complicated construction of $\Pi_m^g$, compared to the normal operator.

Now one can define a bilinear form on $T_g\mathcal{M}_-$ as follows:

\begin{defn}\label{def:bilinear}
Define a bilinear form $G_g(\cdot, \cdot)$ on $T_g\mathcal{M}_-$ by 
$$G_g(h_1,h_2):= \langle\Pi_2^{g}h_1, h_2\rangle_{L^2_g(S)}$$
for $h_j \in T_g\mathcal{M}_-\simeq \sf{S}_{2}(S)$ and $j=1,2$. 
\end{defn}

\begin{lem}
The bilinear form $G_g(\cdot, \cdot)$ satisfies
\begin{align*}
    G_g(h,h)=\textnormal{Var}(P_g(\pi_2^*h),\mu_g^L)+\langle\pi_2^*h,1\rangle^2_{L^2(T^1S_g)},
\end{align*}
for $g\in\mathcal{M}_-$ and $h\in T_g\mathcal{M}_-$, where Var is the variance defined in Definition \ref{def,Variance}.
\end{lem}
\begin{proof}
We have
\begin{align*}
G_g(h,h)&= \langle\Pi_2^{g}h, h\rangle_{L^2_g(S)}\\
&=\langle\Pi^{g} \pi_2^*h, \pi_2^*h\rangle+\langle\pi_2^*h,1\rangle_{L^2(T^1S_g)}\langle\pi_2^*h,1\rangle_{L^2(T^1S_g)}\\
&=\langle\Pi^{g}\big(P_g(\pi_2^*h) \big),\pi_2^*h\rangle+\langle\pi_2^*h,1\rangle_{L^2(T^1S_g)}\langle\pi_2^*h,1\rangle_{L^2(T^1S_g)}\hspace{.5in}&&\text{(by Lemma \ref{defn,PiDefn})}\\
&=\langle\Pi^{g}\big(P_g(\pi_2^*h) \big),P_g(\pi_2^*h)\rangle+\langle\pi_2^*h,1\rangle_{L^2(T^1S_g)}\langle\pi_2^*h,1\rangle_{L^2(T^1S_g)}. \hspace{.5in}&&\text{(By Theorem \ref{thm,PiProperty}(\ref{thm4.2it5}).)}
\end{align*}
\noindent The result follows immediately because the first term coincides with $\textnormal{Var}(P_g(\pi_2^*h),\mu_g^L)$.
\end{proof}

\begin{cor}\label{cor:corr}
The bilinear form $G_g(\cdot, \cdot)$ also satisfies
\begin{align*}
    G_g(h_1,h_2)&= \mathrm{Cov}(P_g(\pi_2^*h_1),P_g(\pi_2^*h_2),\mu_g^L)+\langle\pi_2^*h_1,1\rangle_{L^2(T^1S_g)}\langle\pi_2^*h_2,1\rangle_{L^2(T^1S_g)}\\
    &= \mathrm{Cov}(P_g(\pi_2^*h_1),\pi_2^*h_2,\mu_g^L)+\langle\pi_2^*h_1,1\rangle_{L^2(T^1S_g)}\langle\pi_2^*h_2,1\rangle_{L^2(T^1S_g)}.
\end{align*}
\end{cor}
\begin{proof}
For a proof, see \cite[Corollary 2.21 and Definition 2.22]{Dai}.
\end{proof}

\begin{prop}\label{prop,smoothness}
The bilinear form $G_g(\cdot, \cdot)$ is defined  on the Banach manifold $\mathcal{M}_-^{k,\alpha}$ for fixed large $k$ and $\alpha\in (0,1)$ and is $C^{k-3}$ there, 
in the sense that for any smooth local sections $h_1,h_2\in C^\infty(\cal{U};\sf{S}_{2}^{k,\alpha}(S))$, where $\cal{U}\subset  \cal{M}^{k,\alpha}_{-}$ is an open neighborhood of a metric $g_0$, the function 
\begin{equation}\label{eq:smoothness_form}
    \cal{U}\to \bb{R},\qquad  g  \mapsto G_g(h_1(g),h_2(g))
\end{equation}
is $C^{k-3}$.
Similarly, if $\cal{U}\subset \cal{M}_-$ and $h_1,h_2\in C^\infty(\cal{U};\sf{S}_{2}(S))$ with $h_1,h_2:\cal{M}_-^{k,\alpha}\to\sf{S}_{2}^{k,\alpha}(S)$ smooth for all $k$, then \eqref{eq:smoothness_form} is smooth on $\cal{U}$.
\end{prop}

\begin{proof}
We will use the expression in Corollary \ref{cor:corr}, which also makes sense on $\cal{M}_{-}^{k,\alpha}$, as the following proof shows.
For $\cal{U}\subset \cal{M}_{-}^{k,\alpha}$ and for $i=1,2$, consider the map (see \cite[Section 1.2]{GeodesicStretch} for the last equality) 
\begin{equation}
    F_i(g):\cal{U}\to \bb{R}, \quad  g\mapsto \langle\pi_2^*h_i(g),1\rangle_{L^2(T^1S_g)}=\int_{T^1S_g} \pi_2^*h_i(g)\mathrm{d}\mu^L_g
    = \frac{1}{{\mathrm{Area}(S,g)}}\int_{S} \mathrm{tr}_g h_i(g) \mathrm{d}v_g.
\end{equation}
The integral factor can be written locally in coordinates $x^1, x^2$ as $\displaystyle
\int \sum\limits_{s,t=1}^2g^{st}(h_i(g))_{st}\sqrt{\det(g)}dx$.  Since for all $s, t =1,2$, the maps $g\mapsto g^{st}$ and $g\mapsto \sqrt{\det(g)}$ are smooth  into $ C^{k,\alpha}(S)$, so is the integrand.
Then the integral defines a bounded linear map from $C^{k,\alpha}(S)$ (actually even from $C^0(S)$) into $\bb{R}$, so the composition is smooth. 
The area is given locally by $\int \sqrt{\det(g)}dx$, so it is also smooth in $g$.
So the maps $F_i$ and their product are smooth.

Then we show that $g\mapsto\mathrm{Cov}(P_g(\pi_2^*h_1(g)),P_g(\pi_2^*h_2(g)),\mu_g^L)$ is $C^{k-3}$. From the  thermodynamic formalism, we know (see e.g.\cite[Proposition 4.11]{BookZetaFunctionPollicott} ,\cite[Remark 2.25]{Dai}) that
\begin{equation}\label{eq:pressure_function}
    \mathrm{Cov}(P_g(\pi_2^*h_1(g)),P_g(\pi_2^*h_2(g)),\mu_g^L)= \frac{\partial^2 \textbf{P}(-J^{u}_{g}+s\pi_2^*h_1(g) + t\pi_2^*h_2(g), X^g)}{\partial t \partial s}\bigg|_{s=t=0}.
\end{equation}
Here $\textbf{P}(\cdot, X^g)$ 
is the pressure function  with respect to the geodesic flow of $g$ (see e.g. \cite{BookZetaFunctionPollicott}), which is determined by the geodesic vector field $X^g\in \mathfrak{X}^{k-1,\alpha}(TS)\subset \mathfrak{X}^{k-1}(TS)$ (note that this inclusion is smooth). The Liouvile measure $\mu_g^L$ is the equilibrium state of $-J^{u}_{g}$, where $J^{u}_{g}$ is the unstable Jacobian of the geodesic flow generated by $X^{g}$ (\cite[Section 4 and Section 5]{BowenRuelle}). 
The pressure is real analytic in the first component (see \cite[Proposition 4.8]{BookZetaFunctionPollicott}, \cite[page 377]{McMullen_Pressure}). By \eqref{eq:pressure_function} and polarization, to show that $g\mapsto \partial_1^2\mathbf{P}(0,X^g)(\pi_2^*h_1(g), \pi_2^*h_2(g))$  is $C^{k-3}$, it suffices to show that $$f(t,g):=\textbf{P}( t\pi_2^*h_2(g), X^g): \mathbb{R} \times \cal{M}^{k,\alpha}_{-} \to \mathbb{R}$$ is $C^{k-1}$.
Since we know 
$$\cal{U}\ni g\mapsto \pi_2^* h_j(g)\in C^{k,\alpha}(TS)\quad \text{ and }\quad \cal{U}\ni g\mapsto X^g \in \mathfrak{X}^{k-1}(TS) $$ are smooth, from \cite[Theorem C(a)]{RegularityEntropyContreras} one knows that upon fixing $t=t_0$,  the map $\textbf{P}( t_0\pi_2^*h_2(g), X^g)$ is $C^{k-1}$ respect to $g$. Since varying $t$ does not change the background flows and corresponding subshifts of finite type \cite[Lemma 5.1]{RegularityEntropyContreras}, the proof of  \cite[Theorem C]{RegularityEntropyContreras} (page 110) can be adjusted to show that $\textbf{P}( t\pi_2^*h_2(g), X^g)$ is jointly $C^{k-1}$ for both parameters $t$ and $g$ from the real analytic dependence of the pressure on the first component.
\end{proof}

It is important for our purposes that the bilinear form $G_g(\cdot,\cdot)$ is positive definite on $T_g \mathcal{M}_-\cap \textnormal{ker} D_g^*$:

\begin{lem}[{\cite[Lemma 2.1]{GeodesicStretch}}] \label{lem, nondegeneracy}
Given $h\in T_g \mathcal{M}_-\cap \textnormal{ker} D_g^*$, then
$$G_g(h,h)\geq 0.$$
Moreover, $G_g(h,h)= 0$ if and only if $h=0$.
\end{lem}

\noindent Another important criterion for the bilinear form $G(\cdot,\cdot)$ to descend to $\mathcal{M}_-/\mathcal{D}_0$ is the following Lemma

\begin{lem} \label{lem, KillPotential}
Suppose $h_1= D_gp \in T_g\mathcal{M^-}$ is a potential tensor with  $p\in \sf{S}_1(S)$. Then 
$$G_g(h_1,h_2)=0$$
for any $h_2\in \sf{S}_{2}(S)$.
\end{lem}

\begin{proof}
We write $\langle\Pi_2^{g}h_1,h_2\rangle_{L^2_g(S)}=\langle \Pi^g \pi^*_2 (D_gp), \pi^*_2h_2 \rangle + \langle \pi_2^*(D_gp),1 \rangle_{L^2(T^1S_g)} \langle \pi_2^*h_2,1 \rangle_{L^2(T^1S_g)}$. By equation \eqref{eqtn,XandD} and Theorem \ref{thm,PiProperty}, we know that $ \pi^*_2 D_gp=X^g\pi^*_1p$ and that $X^g\pi^*_1p$ is in the kernel of the $\Pi^g$ operator. 
Also $\langle \pi_2^*D_gp,1 \rangle_{L^2(T^1S_g)}=\langle p, D_g^* \pi_{2*} 1\rangle_{L^2_g(S)}$, so since $\pi_{2*}1=\frac{1}{2}g$ and $D_g^*=-\textnormal{Tr}_g \nabla_g$, we conclude that $G_g(h_1,h_2)=0$.
\end{proof}

Now we are able to introduce the Riemannian metric from  \cite{GeodesicStretch}.

\begin{prop} [{\cite[Proposition 3.9]{GeodesicStretch}}]
\label{def:pressure}
The bilinear form $G$ produces a Riemannian metric on the quotient space $\mathcal{M}_-/\mathcal{D}_0$, called the covariance metric. Given $[g] \in \mathcal{M}_-/\mathcal{D}_0$, we denote the covariance metric at $[g]$ by $G_{[g]}(\cdot,\cdot)$.
\end{prop}

\begin{proof}
The proof of this Proposition can be found in \cite[Proposition 3.9]{GeodesicStretch} and we only repeat it for its importance. 
For a fixed $g_0 \in \mathcal{M}_-$, we identify a neighborhood of $[g_0]\in\mathcal{M}_-/\mathcal{D}_0$ with a slice $\mathcal{W}\subset\cal{M}_-$ (a Fr{\'e}chet submanifold)  passing through $g_0$.
We verify positive definiteness.
For $g\in  \cal{W}$ near $g_0$, let $h\in T_g \mathcal{W}$. Since we can decompose $h=\mathcal{L}_Yg+h'$, where $Y\in \mathfrak{X}(S)$ and $D_g^* h'=0$, by Lemma \ref{lem, KillPotential} we obtain $G_g(h,h)=G_g(h',h') \geq 0$ with equality exactly when $h'=0$, by Lemma \ref{lem, nondegeneracy}. If $h'=0$, the fact that $T_g \mathcal{W}\cap\{\mathcal{L}_Yg|Y\in \mathfrak{X}(S)\}=\{0\}$ yields $h=0$. 
We notice that the argument does not depend on which slice $\mathcal{W}$ we use to identify $\mathcal{M}_-/\mathcal{D}_0$ by Lemma \ref{lem, KillPotential}. So $G_{[g]}(\cdot,\cdot)$ is a well-defined Riemannian metric on the quotient space $\mathcal{M}_-/\mathcal{D}_0$. 
\end{proof}

\begin{rem} \label{rem,ConformalDegenerate} 
From the proof of Lemma \ref{lem, KillPotential}, we notice that $\langle\Pi^{g} \pi_2^*h, \pi_2^*h\rangle=\textnormal{Var}(P_g(\pi_2^*h),\mu_g^L)$ also descends to a bilinear form on $\mathcal{M}_-/\mathcal{D}_0$. However it is not positive definite and therefore does not give a Riemannian metric on $\mathcal{M}_-/\mathcal{D}_0$. For example, consider a family of conformal metrics $\{g_t\}_{t\in\mathbb{R}}\in \mathcal{M}_-$ given by $g_t=tg$ for some $g\in\mathcal{M}_-$. Since $h=\dot{g}_0=g$ is divergence free (and therefore not a potential tensor), we have $d\pi_{\cal{M}_{-}}h=d\pi_{\cal{M}_{-}}g\in T_{[g_0]}(\mathcal{M}_-/ \mathcal{D}_0)$ is nonzero (where $\pi_{\cal{M}_{-}}:\cal{M}_{-}\to{\cal{M}_{-}}/\cal{D}_0 $ is the quotient map). But $P_g(\pi_2^*h)=P_g(\pi_2^*g)=0$, so $\langle\Pi^{g} \pi_2^*h, \pi_2^*h\rangle=0$. 
\end{rem}

Next we show that the extended mapping class group is an isometry subgroup of the covariance metric. Recall that the action of $\cal D$ on $\mathcal{M}_-$ by pullback induces a right action of the extended mapping class group $\Exmod$ on the space $\mathcal{M}_-/ \mathcal{D}_0$. 
In the following proposition, if $[\psi]\in \Exmod$, we write this action as $\theta_{[\psi]}:\mathcal{M}_-/\cal D_0 \to \mathcal{M}_-/\cal D_0$, $[g]\mapsto \theta_{[\psi]}([g])= [\psi^*g]$.
Then $\theta_{[\psi]}$ is smooth with smooth inverse (note that $\Exmod$ is discrete and $[g]\mapsto [\psi^*g]$ is smooth, as one can see by writing $[g]$ and $[\psi^*g]$ in terms of slices $\cal{W}$ and $\psi^*\cal{W}$ as in \eqref{eq:slice}).
It is actually an isometry with respect to the covariance metric $G_{[g]}(\cdot, \cdot)$.

\begin{prop}[Isometry subgroup]\label{prop,mappingClassGrpInv} 
The covariance metric is invariant under the extended mapping class group action on $\mathcal{M^-}/ \mathcal{D}_0$. In other words, the extended mapping class group is an isometry subgroup of the covariant metric. Explicitly, given $[g]\in \mathcal{M}_-/ \mathcal{D}_0$ and $\widehat{h}_j\in T_{[g]} (\mathcal{M}_- / \mathcal{D}_0) $, for $j=1,2$, and an element $[\psi] \in \Exmod$, we have
\begin{equation}\label{eq:isometry_pres}
G_{\theta_{[\psi]}([g])}(d\theta_{[\psi]}\widehat{h}_1,d\theta_{[\psi]}\widehat{h}_2)=G_{[g]}(\widehat{h}_1,\widehat{h}_2).
\end{equation}
\end{prop}

\begin{proof}
First observe that if $\widehat{h} \in T_{[g]}(\mathcal{M}_-/\cal {D}_0)$ and   $h\in T_{g}\mathcal{M}_-$ satisfies $d\pi_{\mathcal{M}_-} h = \widehat{h}$ for some $g\in [g]$, then 
$d\theta_{[\psi]}\widehat{h}= d\pi_{\mathcal{M}_-} (\psi^* h)$, for any
$\psi\in [\psi]$.
Indeed, let  $g_t\in  \mathcal{M}_-$ be a curve with $h=\frac{d}{dt}{g}_t \big|_{t=0}$, so that $\frac{d}{dt}[g_t]\big|_{t=0}=\widehat{h}$.
Then 
$$d\theta_{[\psi]}\widehat{h}=\frac{d}{dt}\big(\theta_{[\psi]}([g_t])\big)\big|_{t=0}=\frac{d}{dt}\pi_{\mathcal{M}_-}(\psi^*g_t)\big|_{t=0}=d\pi_{\mathcal{M}_-} (\psi^* h).$$
This and the definition of the covariance metric imply that it suffices to show
\begin{equation}\label{eq:isometry}
G_{\psi^{*}g}(\psi^*h_1, \psi^*h_2)=G_g(h_1,h_2)    
\end{equation}
for $h_j\in T_g\cal {M}_-$  satisfying $d\pi_{\mathcal{M}_-} h_j=\widehat{h}_j$ and $\psi\in \cal{D}$, since in that case we have
\begin{equation}\label{eq:computation_isometry}
    G_{[g]}(\widehat{h}_1,\widehat{h}_2)=G_g(h_1,h_2)=G_{\psi^{*}g}(\psi^*h_1, \psi^*h_2)= G_{\theta_{[\psi]}([g])}(d\theta_{[\psi]}\widehat{h}_1,d\theta_{[\psi]}\widehat{h}_2).
\end{equation}
Note here that the $h_j$ are determined by $\widehat{h}_j$ up to the addition of a tensor field which is vertical with respect to the quotient map, that is, a potential tensor. The validity of \eqref{eq:computation_isometry} is independent of the choice of $h_j$ by Lemma \ref{lem, KillPotential}.

To show \eqref{eq:isometry}, recall that
$$G_g(h_1,h_2)=\lim _{T\to\infty}\frac{1}{T}\int_{-T}^T\langle P_g(\pi_2^*h_1)\circ \phi_t,\pi_2^*h_2\rangle_{L^2(T^1S_g)} \mathrm{d}t+\langle\pi_2^*h_1,1\rangle_{L^2(T^1S_g)}\langle\pi_2^*h_2,1\rangle_{L^2(T^1S_g)},$$
where $P_g(\pi_2^*h_1)(x)=\pi_2^*h_1(x)-\langle \pi_2^*h_1,1\rangle_{L^2(T^1S_g)}$.
The Liouville probability measures of $g$ and $\psi^*g$ are related by pushforward, i.e., $\mu^L_{\psi^*g}= (\psi^{-1}_*)_* \mu_g^L$, where $\psi^{-1}_* :T^1S_g\to T^1S_{\psi^*g}$ is  the induced diffeomorphism by $\psi$. 
Also, by the naturality of the Levi-Civita connection, the geodesic flows $\phi_t^g$ and $\phi_t^{\psi^*g}$ of $g$ and $\psi^*g$ respectively are related by conjugation by $\psi$, that is, $\phi_t^{\psi^*g}=\psi_*^{-1}\circ \phi^g_t \circ \psi_*: T^1 S_{\psi^*g} \to T^1 S_{\psi^*g}$. 
A simple change of variable then yields \eqref{eq:isometry}. 
\end{proof}

\begin{rem}
    Although we have only shown the above theorem for the right pull back action of the extended mapping class group on $\mathcal{M}_- / \mathcal{D}_0$, it naturally also holds for its induced left action introduced in Section \ref{subsec: MappingClassGrpMetrics}.
\end{rem}

\subsection{The special case of \texorpdfstring{$\mathcal{T}(S)$}{T(S)}}
\label{subsec: TeichPressure}

Fischer and Tromba \cite{FischerTromba}, using Riemannian geometry and non-linear analysis, reprove the classical result that $\cal{T}(S)=\cal{M}_{-1}/\cal{D}_0$ is  a $C^{\infty}$ finite dimensional contractible manifold.
Its tangent space at $[\sigma]\in \cal{M}_{-1}/\cal{D}_0$ is isomorphic to
\begin{equation} \label{TangentTeich}
    \sf{S}_2^{\sigma,TT}(S)=\{ h\in \sf{S}_{2}(S)|  \Tr_\sigma h=0, D_\sigma^*h=0\},
\end{equation}
where $\sigma\in [\sigma]$.
More specifically, given $\sigma_0\in \cal{M}_{-1}$ and $k\gg 1$, $\alpha\in(0,1)$, one can construct a local slice $\cal{S}\subset \cal{M}_{-1}\subset \cal{M}_{-1}^{k,\alpha}$ passing through $\sigma_0$, identified with a neighborhood of $[\sigma_0]\in \cal{M}_{-1}/\cal{D}_0$, so that $T_{\sigma_0}\cal{S}= \sf{S}_2^{\sigma_0,TT}(S)$ (see \cite[Theorem 2.4.2]{Tromba_Book} and Section \ref{subsec:smoothnessBlaschke})\footnote{
In \cite{Tromba_Book}, the slice $\cal{S}$ is a submanifold of the Hilbert manifold $\cal{M}_{-1}^s$ of hyperbolic metrics with fixed Sobolev regularity, though the proofs work similarly in the Hölder case. }.
 Moreover, for $h\in T_\sigma \cal{S}$ the decomposition \eqref{eqtn, Divergences} holds with $v\in \sf{S}_2^{\sigma,TT}(S)$. The space $\sf{S}_2^{\sigma,TT}(S)$ is related to holomorphic data on the Riemann surface $X_J$, where $J$ is the complex structure determined by $\sigma$ and a choice of orientation:

\begin{thm}[{\cite[Theorem 8.9]{FischerTromba}}]\label{rem: Req}
For $\sigma \in \cal M_{-1}$, there is a canonical isomorphism between the spaces $H^{0}(X_J, K_J^2)$ and $\sf{S}_2^{\sigma,TT}(S)$, given by $q\mapsto \Re(q)$. Thus by the Riemann-Roch theorem, the space $\sf{S}_2^{\sigma,TT}(S)$ is of real dimension $6 \mathscr{G}-6$.
\end{thm}

\noindent One can then simplify the formula of the covariance metric when restricting to the Teichm{\"u}ller space  $\mathcal{T}(S)=\cal{M}_{-1}/\cal{D}_0$ and show that it restricts to a scale of the Weil-Petersson metric there.
\begin{cor}
On Teichm{\"u}ller space $\mathcal{T}(S)$, the covariance metric at $[\sigma]\in\mathcal{T}(S)$ satisfies  
\begin{equation}
    G_{[\sigma]}(\widehat{h},\widehat{h})=\textnormal{Var}(\pi_2^*h,\mu_{\sigma}^L),
\end{equation}
where $\widehat{h}\in T_{[\sigma]}\mathcal{T}(S)$, $\sigma\in [\sigma]$ and $h$ is the lift of $\widehat{h}$ in  $\sf{S}_2^{\sigma,TT}(S)$.
\end{cor}

\begin{proof}
We have $\langle \pi_2^*h, 1 \rangle_{L^2(T^1S_{\sigma})}= \mathrm{Area}(S,\sigma)^{-1} \int_{S} \Tr_\sigma h\textnormal{ }  \mathrm{d}v_\sigma =0$ (see \cite[Section 1.2]{GeodesicStretch}). Since $h\in \sf{S}_2^{\sigma,TT}(S)$, one concludes that  $\pi_2^*h$ is of mean zero.
\end{proof}

\noindent Combining the above discussions, we obtain:

\begin{thm} [{\cite[Theorem 1.5]{McMullen_Pressure}, \cite{Bridgeman_Pressure}, \cite[Theorem 6.3.1]{VariationAlongFuchsian}}] \label{thm,Pressure=WP}
Given $\widehat{h}\in T_{[\sigma]}\mathcal{T}(S)$, consider the lift $h\in \sf{S}_2^{\sigma,TT}(S)$ given by the real part of a holomorphic quadratic differential $q$ (Theorem \ref{rem: Req}). Then 
$$ G_{[\sigma]}(\widehat{h},\widehat{h})=\textnormal{Var}({\mathfrak{R}} (q),\mu_{\sigma}^L)= C \langle [q],[q] \rangle_{\scriptscriptstyle  WP}([\sigma]) $$
Here ${\mathfrak{R}} (q):=\pi^*_2(\textnormal{Re}(q))\in C^{\infty}(T^1S_{\sigma},\mathbb{R})$.
The constant $C$ only depends on the topology of $S$.
\end{thm}

\section{The Blaschke locus in \texorpdfstring{$\mathcal{M}_-/ \mathcal{D}_0$}{M-/D0}}
\label{Sec:BlaschkeLocus}

This section discusses the Blaschke locus $\mathcal{M}^B /\mathcal{D}_0$. In Subsection \ref{subsec:affinesphere}, we introduce some basic concepts from affine differential geometry. Then in Subsection \ref{subsec, BlaschkeHitchin}, we introduce the Blaschke locus and discuss its relation with the holomorphic vector bundle $Q_3(S)$ (Subsection \ref{subsec,HitchinComponents}). We then investigate regularity of the Blaschke locus $\mathcal{M}^B /\mathcal{D}_0$ in Subsection \ref{subsec:smoothnessBlaschke} in the spirit of Tromba \cite{Tromba_Book}, finishing with a discussion on the topology of the Blaschke locus in Subsection \ref{subsec:TopologyBlaschke}. Wang's equation (\ref{WangEquation}) will be the key for our study in this and the next sections.

\subsection{Affine differential geometry and Blaschke metrics}\label{subsec:affinesphere}

In this subsection we give a brief introduction on \emph{affine spheres} and  \emph{Blaschke metrics}.
Those are objects arising from affine differential geometry, which is the study of {affine differential invariants}, namely differential properties of hypersurfaces of $\mathbb{R}^{n+1}$ which are invariant under all volume preserving affine transformations. Standard references for affine differential geometry are \cite{Loftin_Survey} and \cite{AffineDifferentialGeometry}. The space of Blaschke metrics, which include hyperbolic metrics as special examples, will be the object of study in what follows.

A basic construction in affine differential geometry associates to a hypersurface $L$ of $\bb{R}^{n+1}$ a transverse vector field $\xi\pitchfork L$, the \emph{affine normal} vector field, which is an affine differential invariant. An \emph{affine sphere} is a hypersurface $L$ whose affine normal lines are concurrent at a point, the \emph{center}. 
We outline here the construction of an affine sphere in the special case of $\mathbb{R}^3$ and its associated affine differential invariants.
Let $L$ be a hypersurface in $\mathbb{R}^3$. 
A choice of a transverse vector field $\xi: L\to \mathbb{R}^3$ yields a decomposition $\mathbb{R}^3=T_{p}L \oplus \langle\xi(p)\rangle$ for any $p\in L$, where $\langle\xi(p)\rangle$ stands for the line spanned by $\xi(p)$. This allows one to decompose the standard flat affine  connection $D$ on $\bb{R}^3$ into tangential part $\nabla$ and normal part as follows,
\begin{equation}\label{affineDecompositionI}
    D_X Y=\nabla_X Y+ g(X,Y)\xi,
\end{equation} 
\begin{equation}\label{affineDecompositionII}
    D_X \xi= -B(X)+\tau(X) \xi, 
\end{equation}
where $X$ and $Y$ are tangent vector fields to $L$ and $B=B_{\xi}$ is an endomorphism of $TL$. 
Observe that $\nabla$ is a torsion-free connection on $TL$, so $g$ is a symmetric 2 tensor.  By restricting to the case where $L$ is strictly convex and $\xi$ points towards the convex side of the surface $L$, we can assume $g$ is positive definite. Further, by imposing the conditions  $\tau \equiv0$ and  $\det(\xi, X_1, X_2)^2 \equiv 1$ for any $g$-orthonormal frame $(X_1, X_2)$, one
determines a unique transversal vector field $\xi$ on $L$, which  is an affine differential invariant and is called the \emph{affine normal} of $L$. The endomorphism $B$ is then called the \emph{affine shape operator}. Moreover, one can check that the vector field $\xi$ being concurrent to a point is equivalent to the affine shape operator $B$ being a nonzero multiple of identity: $B=HI$ for some constant $H\in\mathbb{R}\setminus \{0\}$, the \emph{affine mean curvature}. 
We will focus on the case where $H=-1$, in which case $L$ is a \emph{hyperbolic affine sphere}\footnote{By applying a translation, one can always assume that the center of the hyperbolic affine sphere is the origin.}.

We will need two other affine differential invariants associated to the affine sphere $L$. One of them is the affine second fundamental form $g$, which is symmetric and positive definite. It yields a Riemannian metric which is an affine differential invariant on $L$, called the \emph{Blaschke metric}. 
The second is a cubic form $A$ on $TL$ known as the \emph{Pick form}: it is given by taking the difference $\nabla-\nabla^g$, where $\nabla^g$ denotes the Levi-Civita connection of the Blaschke metric, and lowering an index via $g$. If we use the conformal class of the Blaschke metric to regard $L$ as a Riemann surface, then the Pick form $A$ is the real part of a cubic differential $q=\tilde{q}(z)dz^3$ 
on $L$, which is called the \emph{Pick differential}. 

The map $f=\xi: L \to \mathbb{R}^3$ provides an immersion of $L$ into $\mathbb{R}^3$ if the integrability conditions for the structure equations \eqref{affineDecompositionI} and \eqref{affineDecompositionII} are satisfied\footnote{While $f(L)$ is the immersed affine sphere we obtain through this construction, we often abuse notation and call $L$ the affine sphere.}. 
Written in complex coordinates $z$ determined by the conformal class of the Blaschke metric $g$, the integrability conditions (see \cite[Section  5]{Loftin_Survey}) for $f$ are the following partial differential equations,
\begin{align}
     \tilde{q}_{\bar{z}}=0,\label{holomoprhic}
 \\*
    K(g)=-1+2|q|^2_{g}.\label{WangEquation}
\end{align}
The first equation \eqref{holomoprhic} simply requires the Pick differential $q$ to be holomorphic. The second equation (\ref{WangEquation}) is an second order partial differential equation in $g$, where $K(g)$ denotes the Gaussian curvature of $g$ and $|q|_{g}^2=\frac{|q|^2}{g^3}$ is the pointwise $g$ norm of the holomorphic cubic differential $q$. It is called  \emph{Wang's equation} in the affine sphere literature \cite{Wang91}. This equation is  of key importance in this note.

\begin{prop}[{\cite{Loftin_thesis}}]\label{pr:wang}
Wang's equation (\ref{WangEquation}) admits a unique smooth solution given a holomorphic cubic differential $q$ on a compact Riemann surface.
\end{prop}

\noindent An interesting aspect of affine sphere theory is its connection with Higgs bundle theory introduced in Section \ref{subsec,HitchinComponents}.  
Wang's equation can be viewed as a special case of the Hitchin equation for a $\mathrm{PGL}(3,\bb{R})$ Higgs bundle: let $J$ be a complex structure on $S$ and denote by $\sigma$ the hyperbolic metric associated to $J$. Let $q$ be a holomorphic cubic differential with respect to the complex structure $J$. The Hitchin equation  \eqref{eq HitchinEquation} for the Higgs bundle $s_J(0,2q)$ in fact reduces to a single scalar equation, which is Wang's equation \eqref{WangEquation} on $S$ associated to $(J,q)$ (see for example \cite[Section 9]{FlatProjectiveCubicDifferentials} and \cite[Section 6.2]{QionglingIntro}). 

Returning to the closed oriented surface $S$ with genus $\mathscr{G} \geq 2$ we started with, we denote $\tilde{S}$ its universal cover. The punchline of the whole discussion is the following important Theorem.

\begin{thm} [{\cite[Theorem 3.1, Theorem 3.5]{Wang91}}]\label{thm:affine}
Given a complex structure $J$ on $S$, any holomorphic cubic differential $q$ on $X_J=(S,J)$ determines a complete hyperbolic affine sphere $f: \tilde{S} \to \mathbb{R}^3$ that admits discrete and properly discontinuous subgroup action in $\mathrm{SL}(3,\mathbb{R})$ so that the quotient topologically is $S$. Its Blaschke metric is given by the solution of Wang's equation on $\tilde{S}$ which descends to $S$. 
Conversely, any complete hyperbolic affine sphere $f: \tilde{S} \to \mathbb{R}^3$ that admits a discrete and properly discontinuous subgroup action in $\mathrm{SL}(3,\mathbb{R})$ with quotient topologically given by $S$ defines a complex structure $J$ given by the conformal class of its Blaschke metric  and a holomorphic cubic differential $q$ with respect to this complex structure on $S$.
\end{thm}

\begin{rem}
A last remark that we want to make about the affine sphere theory is its relation with hyperbolic geometry and Teichm\"uller theory. In the special case in which the holomorphic cubic differential $q\equiv0$, Wang's equation reduces to the classical curvature equations for metrics of constant curvature $-1$. Therefore a hyperbolic metric is a special example of a  Blaschke metric. In these cases, the affine spheres obtained are universal covers of hyperbolic surfaces viewed in the hyperboloid model as mentioned in the introduction.
\end{rem}

\subsection{The Blaschke locus \texorpdfstring{$\mathcal{M}^B /\mathcal{D}_0$}{M\^B / D0} } \label{subsec, BlaschkeHitchin}

In this section, we define the Blaschke locus first as a set and then study its manifold structure.

\begin{defn}
We define $\mathcal{M}^B$ to be the space of Blaschke metrics on $S$. The quotient space of $\mathcal{M}^B$ up to $\mathcal{D}_0$-action is denoted by  $\mathcal{M}^B /\mathcal{D}_0$, called the Blaschke locus.
\end{defn}

\noindent The Blaschke locus $\mathcal{M}^B/\mathcal{D}_0$ is a subset of the space $\mathcal{M}_- /\mathcal{D}_0$ due to the following Proposition: 

\begin{prop}[{\cite[Theorem 5.1]{PolonomialCubics}, \cite[Lemma 3.3]{OuyangTamburelli}}]\label{prop:neg_curv}  
Any Blaschke metric $g$ is strictly negatively curved, i.e. $K(g) <0$.  
\end{prop}

Fix a choice of orientation on $S$. Denote by $J(\sigma)$  the  complex structure determined by $\sigma$ and the orientation. Let
\begin{equation}\label{eq:Q3}
    \tilde{Q}_3(S):=\bigsqcup_{\sigma\in \cal{M}_{-1}}H^0(X_{J(\sigma)}, K_{J(\sigma)}^3),
\end{equation}
viewed as a subset of $\cal{M}_{-1}\times \sf{S}_3(S)^{\mathbb{C}}$, where the superscript $\mathbb{C}$ denotes complexification and $X_{J(\sigma)}=(S, J(\sigma))$ is the Riemann surface with the complex structure $J(\sigma)$. Let 
\begin{equation}\label{eq:g_tilde}
  \widetilde{\mathbf{g}}:  \tilde{Q}_3(S)\to\cal{M}^B\subset  \cal{M}_{-}
\end{equation}
be the map that assigns to a pair $(\sigma,q)$ the solution  of Wang's equation \eqref{WangEquation} on the Riemann surface $X_{J(\sigma)}$.
Now $\cal{D}_0$ (or $\cal{D}^{+}$) act on $\tilde{Q}_3(S)$ and $ \cal{M}_{-1}$ in a natural way from the right by pullback (and similarly from the left by inverse pullback).
Consider the equivalence relation on $\tilde{Q}_3(S)$ given by $(\sigma,q)\sim (\sigma',q')$ if and only if for some $\psi\in \cal{D}_0$, one has $ \sigma'=\psi^*\sigma$ and $q'=\psi^* q$. We henceforth identify  $Q_3(S)$ with $\tilde{Q}_3(S)/\cal{D}_0 $ by viewing $Q_3(S)$ as a bundle over $\cal{M}_{-1}/\cal{D}_0$, for a fixed choice of orientation. 
With some abuse of notation we denote elements in $\tilde{Q}_3(S)$ by either $(\sigma,q)$ or $(J,q)=(J(\sigma),q)$. This is allowed because positively oriented complex structures are in one-to-one correspondence with hyperbolic metrics.

\begin{prop}\label{prop_equiv_bl}
The map $\widetilde{\mathbf{g}}$ is equivariant with respect to the action of $\psi\in \cal{D}^+$ by pullback: $\psi^*(\widetilde{\mathbf{g}}(J,q))= \widetilde{\mathbf{g}}\big(\psi^*( J,q)\big)$. Similarly, it is equivariant with respect to the left action of $\cal{D}^+$ on $\tilde{Q}_3(S)$ and $\cal{M}^B$ by inverse pullback.
Thus after taking a quotient by the $\cal{D}_0$ action, we obtain a well defined mapping class group equivariant map 
\begin{equation}\label{eq:noS1}
    \mathbf{g}:Q_3(S) \to \mathcal{M}^B/\mathcal{D}_0.
\end{equation}
\end{prop}

\begin{proof}
Denoting $g=\widetilde{\mathbf{g}}(J, q)$, we  need to verify that $ K(\psi^*g)=-1+2|\psi^*q|^2_{\psi^*g}$. By uniqueness, this will imply that $\psi^*g$ is the solution of  Wang's equation \eqref{WangEquation} associated to the pair $\psi^* (J, q)$. Since $\psi$ is an isometry between $\psi^*g$ and $g$, we know $$K(\psi^*g)=\psi^*K(g)=\psi^*(-1+ 2|q|^2_{g}).$$
It now suffices to show that $ |\psi^*q|^2_{\psi^*g}=\psi^*|q|^2_{g}$ on $S$.
 Let $z$ be conformal coordinates with respect to $J$ and write $z=\psi(w)$, with $w$ conformal coordinates with respect to $\psi^*J$.
 Then if $g=e^{u(z)}dzd\bar{z}$ and $q=f(z)dz^3$, we have $\psi^*g=e^{u(\psi(w))}|dz/dw|^2 dwd\bar{w}$ and $\psi^* q=f(\psi(w))(dz/dw)^3 dw^3$, so that locally 
 \begin{equation*}
    |\psi^*q|^2_{\psi^*g}=|f(\psi(w))|^2/e^{3u(\psi(w))}=\psi^*|q|^2_{g}.
 \end{equation*}
 This is a local computation but we have invariantly defined global objects on both sides, so we have the claim.
\end{proof}

\begin{rem}\label{orientation_cplx}
    Since the solution of Wang's equation corresponding to $(J,q)$ is the same as the one corresponding to $(-J,\bar{q})$, the proof of Proposition \ref{prop_equiv_bl} shows that if one views $\widetilde{\mathbf{g}}$ as a map from the bundle of holomorphic cubic differentials over all complex structures (not necessarily positively oriented), then it is also equivariant with respect to the action of $\cal{D}$  by pullback and inverse pullback, so it descends to an extended mapping class group equivariant map when taking $\cal{D}_0$ quotients. 
    This point of view will not be needed or used in the sequel, except briefly in  the proof of Lemma \ref{lem fixedPointSet}.
\end{rem}

Besides the $\cal{D}_0$ action on  $\tilde{Q}_3(S)$, there is also a natural action of $S^1$ on $\tilde{Q}_3(S)$ which is given by $e^{2\pi i \theta}\cdot (\sigma, q)=(\sigma, e^{2\pi i \theta} q)$.  From now on, we will write elements in $\tilde{Q}_3(S)/S^1$  as $(\sigma,\langle q\rangle)$ for the class of $(\sigma,q)\in \tilde{Q}_3(S)$.
Because the $S^1$ action on  $\tilde{Q}_3(S)$ commutes with the $\cal{D}_0$ action on  $\tilde{Q}_3(S)$, it descends to an action on $Q_3(S)$.
Moreover, the map $\widetilde{\mathbf{g}}$ is constant on the orbits of this action as seen from \eqref{WangEquation}. Thus by Proposition \ref{prop_equiv_bl}, the maps $\widetilde{\mathbf{g}}$ and $\mathbf{g}$ descend to a map
\begin{equation}\label{eq:withS1}
    \mathbf{g}_0:Q_3(S)/S^1 \to \mathcal{M}^B/\mathcal{D}_0.
\end{equation}
The following proposition, whose proof we include for its importance, shows that the map $\mathbf{g}_0$ is bijective.

\begin{prop}[{\cite[Proposition 4.1]{OuyangTamburelli}}] \label{equalBlaschke}
Suppose that $[(\sigma_j,q_j)]\in Q_3(S)$, for $j=1,2$, and that $[g_j]=\mathbf{g}([(\sigma_j,q_j)])$  are the associated Blaschke metrics. Then $[g_2]=[g_1]$ if and only if $[(\sigma_2,q_2)]= [(\sigma_1,e^{2\pi i\theta}q_1)]$ for some $\theta\in[0,1)$.
\end{prop}
\begin{proof}
 The ``if'' direction follows from Propositions \ref{pr:wang} and \ref{prop_equiv_bl}. For the converse, suppose that $[g_1]=[g_2]$, which implies that
 $g_1=\psi^* g_2$ for some $\psi\in \cal{D}_0$. 
For $j=1,2$, write  $g_j=e^{u_j}\sigma_j$ for suitable smooth functions $u_j$.
Then 
$$e^{u_1}\sigma_1=\psi^* (e^{u_2}\sigma_2)=e^{\psi^* u_2}\psi^*\sigma_2,$$
which shows that $\sigma_1 $ and $\psi^*\sigma_2$ are hyperbolic metrics in the same conformal class and thus equal.
We now show that $q_1=e^{2\pi i \theta }\psi^* q_2$ for some $\theta\in [0,1)$, which will imply the claim.
Since $g_1=\psi^* g_2$, from Wang's equation \label{{WangEquation}} we have
\begin{align*}
|q_1|^2_{g_1}=\psi^*(|q_2|_{g_2}^2)=|\psi^{*}q_2|^2_{\psi^{*}g_2}=|\psi^{*}q_2|^2_{g_1}.   
\end{align*}
Therefore, if we write $q_1=f_1(z)dz^3$ and $\psi^*q_2=f_2(z)dz^3$ in conformal coordinates  with respect to the complex structure $J(g_1)$, where $f_j$ is holomorphic for $j=1,2$, we see that $|f_1|=|f_2|$  pointwise.
If $f_1$ or $f_2$ is identically zero then they both are and we are done.
Else, writing $f_2=\lambda  f_1$, where $\lambda: S\to S^1$ is a meromorphic function, we conclude that $\lambda\in S^1$ must be a constant. Thus  $q_1= e^{2\pi i\theta} \psi^*q_2$ for some $\theta\in[0,1)$ and we obtain the claim.
\end{proof}

\subsection{Smoothness of the Blaschke locus \texorpdfstring{$\mathcal{M}^B /\mathcal{D}_0$}{MB/D0}
} \label{subsec:smoothnessBlaschke}

Our goal in this subsection is to prove Theorem \ref{smoothnessBlaschkeLocus}, namely that the Blaschke locus $\mathcal{M}^B /\mathcal{D}_0$ has the structure of a finite dimensional smooth manifold away from the Teichm\"uller space $\cal{T}(S)=\cal{M}_{-1}/\cal{D}_0$. 
Our strategy is based on the construction of smooth charts for $\cal{T}(S)$ as outlined in \cite{Tromba_Book} (see also Subsection \ref{subsec: TeichPressure}). Briefly, one can construct local compatible smooth charts for the Blaschke locus $\mathcal{M}^B /\mathcal{D}_0$ using the smooth manifold structure of  $Q_3(S)$. %

Throughout, $k$ is a large integer and $\alpha\in (0,1)$.
The superscript $k,\alpha$ will indicate that the relevant space of functions or sections of a bundle have H\"older regularity of this order. 
The following theorem collects the results in \cite[Chapter 2]{Tromba_Book} regarding the smooth manifold structure of Teichmüller space.
It is stated in terms of $C^{k,\alpha}$ instead of Sobolev regularity, with the proof being the same as in the Sobolev case.
\begin{thm}\label{thm:slice}
    Let $\sigma_0\in \cal{M}_{-1}$. There exists a smooth local submanifold $\cal{S}$ of $\cal{M}_{-1}^{k,\alpha}$ of dimension $6\mathscr{G}-6$ passing through $\sigma_0$ which contains only $C^\infty$ metrics and satisfies $T_{\sigma_0}\cal{S}=\sf{S}^{\sigma_0,TT}_2(S)$ (see \eqref{TangentTeich}).
    Moreover, when $\cal{S}$ is sufficiently small, the map
    \begin{equation}
        \Theta_{\cal{S}}:\cal{S}\times \cal{D}_{0}^{k+1,\alpha}\to \cal{M}_{-1}^{k,\alpha},\quad  (\sigma,\psi)\mapsto \psi^{*}\sigma, \label{eq:theta_map}
    \end{equation}
    is a diffeomorphism onto its image. The slices $\cal{S}$ locally parametrize $\cal{T}(S)=\cal{M}_{-1}/\cal{D}_0$ 
    (that is, each $\cal{S}$ does not contain two distinct metrics in the same $\cal{D}_0$ orbit)
    and define smoothly compatible charts, giving $\cal{T}(S)$ the structure of a smooth manifold of dimension $6\mathscr{G}-6$, homeomorphic to $\bb{R}^{6\mathscr{G}-6}$.
\end{thm}

A local slice for $\cal{T}(S)$ through a metric $\sigma_0\in \cal{M}_{-1}$ in Theorem \ref{thm:slice} is constructed via the Poincar\'e map
\begin{equation}
    \lambda:\cal{M}^{k,\alpha}_{-}\to \cal{M}^{k,\alpha}_{-1},
\end{equation}
which takes a metric to the unique hyperbolic metric in its conformal class.
More specifically, $\lambda$ defines a smooth diffeomorphism between a neighborhood of $0$ in $\sf{S}_2^{\sigma_0,TT}(S)$ and $\cal{S}$, given by writing
$\cal{S}\ni \sigma =\lambda (\sigma_0 +h)$, where $h\in\sf{S}_2^{\sigma_0,TT}(S)$ is sufficiently small. Since $\sf{S}_2^{\sigma_0,TT}(S)$ is a  $6\mathscr{G}-6$ dimensional vector space consisting of smooth tensor fields, we can write $h$ in terms of a basis $\{h_j\}_{j=1}^{6\mathscr{G}-6}$. Smallness of $h=\sum_{j=1}^{6\mathscr{G}-6} a^j h_j$ means exactly smallness of the coefficients $a^j$.

\begin{rem}\label{rem,equivalentTopology} 
 We remark that different choices of $k,$ $\alpha$ (for $k$ sufficiently large) result in equivalent topologies underlying the smooth structures induced by the slices in Theorem \ref{thm:slice}.
 To see this, one can express $\sigma\in \cal{S}$ in terms of a basis for $ \sf{S}_2^{\sigma_0,TT}(S)$ via the Poincar\'e map. 
Explicitly, if in $\cal{S}$ we have $\sigma_n=\lambda(\sigma_0+\sum_{j=1}^{6\mathscr{G}-6} a^j_n h_j)\overset{n\to\infty}{ \to} \sigma=\lambda(\sigma_0+\sum_{j=1}^{6\mathscr{G}-6} a^j h_j)$ in the $C^{k,\alpha}$ topology,
we  have $a^j_n \overset{n\to\infty}{ \to} a^j$ for all $j$, therefore $\sigma_n \overset{n\to\infty}{ \to} \sigma$ in $C^\infty $ topology as well. 
Moreover, since $\cal{T}(S)$ is homeomorphic to $\bb{R}^{6\mathscr{G}-6}$ with respect to the topology induced from any $\cal{M}_{-1}^{k,\alpha}$ and there exist no exotic smooth structures on  $\bb{R}^d$ for $d\neq 4$, the smooth structures obtained for $\cal{T}(S)$ by means of Theorem \ref{thm:slice} are equivalent for all values of $k$, $\alpha$.
\end{rem}

Recall that we defined the space $\tilde{Q}_3(S)$ in \eqref{eq:Q3} as a space of holomorphic cubic differentials with respect to varying Riemann surfaces.
Since the projection $\tilde{Q}_3(S)\to \cal{M}_{-1}$
is $\cal{D}_0$-equivariant, we obtain a well defined surjective map 
\begin{equation}\label{eq:Q_bundle}
    p:Q_3(S)\to \cal{T}(S). 
\end{equation}
Each fiber of $p$ naturally carries the structure of a complex vector space of  dimension $5(\mathscr{G}-1)$. It is  known that $Q_3(S)$ is a smooth vector bundle (it is actually holomorphic, see for example \cite{bers}) over a contractible space, so it is globally trivial. 
Its topology was discussed in Remark \ref{rmk_topology}.
Below we discuss an identification of $Q_3(S)$ with slices in $\tilde{Q}_3(S)$.

\begin{lem}\label{lm:sliceq}
Let $\sigma_0\in \cal{M}_{-1}$ and let $\cal{S}$ be a slice in $\cal{M}_{-1}^{k,\alpha}$ passing through $\sigma_0$ as in Theorem \ref{thm:slice}. We can identify $Q_3(S)$ locally near a point $[(\sigma_0,q)]$ with a slice
\begin{equation}\label{eq:slice_sq}
    \cal{S}^{\cal{Q}}:=\bigsqcup_{\sigma\in \cal{S}}H^0(X_{J(\sigma)}, K_{J(\sigma)}^{ 3})
\end{equation} 
over  $\cal{S}$.
\end{lem}

\begin{proof}
Given $[(\sigma,q)]\in Q_3(S)$ near $[(\sigma_0,q_0)]$ and the unique $\sigma\in [\sigma]$ lying in $\cal{S}$, there exists a unique $q\in H^0(X_{J(\sigma)}, K_{J(\sigma)}^3)$ such that $ (\sigma,q)\in [(\sigma,q)]$.
Indeed, if $(\sigma,q)=\psi^*(\sigma,q')$ for $\psi\in \cal{D}_0$, then because the $\cal{D}_0$-action is free on $\cal{M}_{-1}$, we have that $\sigma =\psi^*\sigma$ implies $\psi=id$.
\end{proof}

Using the identification described in Lemma \ref{lm:sliceq} we can obtain trivializations for $\cal{S}^{\cal{Q}}$  via a smooth global trivialization $\phi_0$  of $Q_3(S)$, thus giving each slice $\cal{S}^{\cal{Q}}$ a smooth bundle structure so that the projection $(\sigma,q)\mapsto \sigma$ is smooth.
That is, if we let 
\begin{equation}\label{projections}
    \pi_{\cal{S}^{\cal{Q}} }:\cal{S}^{\cal{Q}}\to \pi_{\cal{S}^{\cal{Q}}} ( \cal{S}^{\cal{Q}})\subset Q_3(S) \hspace{.2in} \textnormal{and} \hspace{.2in}   \pi_{\cal{S}}:\cal{S}\to\pi_{\cal{S}}(\cal{S})\subset \cal{T}(S)
\end{equation}
be the respective $\cal{D}_0$ quotient maps and  $\phi_0:Q_3(S)\to \cal{T}(S)\times \mathbb{C}^{5\mathscr{G}-5}$ be a smooth global trivialization for $Q_3(S)$, then 
\begin{equation}
    \tilde{\phi}_{\cal{S}}=(\pi_{\cal{S}}^{-1}\times id)\circ \phi_0 \circ \pi_{\cal{S}^{\cal{Q}}}:\cal{S}^{\cal{Q}}\to \cal{S}\times \bb{C}^{5\mathscr{G}-5}
\end{equation}
is a (global) smooth trivialization for $\cal{S}^{\cal{Q}}$.
Using it, we can write a smooth local frame $\{(\sigma, q_\ell(\sigma))\}_{\ell=1}^{5\mathscr{G}-5}$ for the fiber of $\tilde{Q}_3(S)$ over $\sigma\in \cal{S}$, by choosing a frame $\{[(\sigma, q_\ell(\sigma))]\}_{\ell=1}^{5\mathscr{G}-5}$ for the  bundle $Q_3(S)$ and considering the unique representatives of $[(\sigma, q_\ell(\sigma))]$ over the slice $\cal{S}$.
Then we can express an element $q(\sigma)$ of $\cal{S}^\cal{Q}$ as $q(\sigma)=\sum_{\ell=1}^{5\mathscr{G}-5}b^\ell q_\ell(\sigma)$, $b^\ell\in \bb{C}$.
Moreover, such a section $q(\sigma)$ is smooth when viewed as a map into $\sf{S}_3^k(S)^{\mathbb{C}}$, for any $k\geq 0$. 
This follows from the construction in  \cite[Theorem II]{bers}, where it is shown that for each element $f(z)dz^3\in Q_3(S)\big|_{\tau_0}$ lifted to the universal cover $\bb{D}$ (the unit disk), where $\tau_0\in \cal{T}(S)$,  one has $f(z)=F(z,\tau_0)$ for some jointly holomorphic function $F(z,\tau)$, with respect to the complex structure on $\bb{D}\times \cal{T}(S)$, and the fact that $\cal{S}$ is a smooth chart for $\cal{T}(S)$.

It follows by Theorem \ref{thm:slice} that for two slices $\cal{S}$ and $\cal{S}'$ with $U:=\pi_{\cal{S}}(\cal{S})\cap \pi_{\cal{S}'}(\cal{S}')\neq \emptyset$,
there exists a smooth map $\Psi_{\cal{S},\cal{S}'}:\pi_{\cal{S}}^{-1}(U)\to \pi_{\cal{S}'}^{-1}(U)$ which maps $\sigma\in \cal{S}$ to the unique $\sigma'\in \cal{S}'$ such that $[\sigma]=[\sigma']$.
One has $\Psi_{\cal{S},\cal{S}'}(\sigma)=\psi_\sigma^*\sigma $, where $\psi_\sigma=\pi_2\circ\Theta_{\cal{S}'}^{-1}(\sigma)$ with $\Theta_{\cal{S}'}$ as in \eqref{eq:theta_map} and $\pi_2$ the projection onto the second factor.
(A priori, $\psi_\sigma\in\cal{D}_0^{k+1,\alpha}$, but because the slices $\cal{S}$, $\cal{S}'$ consist of smooth metrics, it is actually smooth, see the proof of \cite[Theorem 2.3.1]{Tromba_Book}.)
Thus we similarly obtain a map 
\begin{equation}\label{Psimap}
    \tilde{\Psi}_{\cal{S},\cal{S}'}:\cal{S}^{\cal{Q}}\to\cal{S'}^{\cal{Q}}, \quad (\sigma,q)\mapsto (\psi_\sigma^*\sigma,\psi_\sigma^*q),
\end{equation}
and  one checks that  $\tilde{\phi}_{\cal{S}'}\circ \tilde{\Psi}_{\cal{S},\cal{S}'}\circ \tilde{\phi}_{\cal{S}}^{-1}(\sigma,b)=(\Psi_{\cal{S},\cal{S}'}(\sigma),b)$, and therefore $\tilde{\Psi}_{\cal{S},\cal{S}'}$ is smooth.

\smallskip

In what follows we construct charts for the Blaschke locus $\cal{M}^B/\cal{D}_0$, away from the Teichm{\"u}ller space $\cal{T}(S)$, using the map $\widetilde{\mathbf{g}}$ in \eqref{eq:g_tilde}.
Similarly to the approach in \cite{Tromba_Book}, we will  use the $C^{k,\alpha}$ topology for $\cal{M}_{B}/\cal{D}_0$ to give it a smooth structure, although the metrics we are interested in are actually smooth. The key for the construction is Wang's equation \eqref{WangEquation}. We will now write this equation slightly differently. Consider the logarithmic density $u$ of a Blaschke metric $g$ given by $g=\widetilde{\mathbf{g}}(\sigma, q)=e^{u(\sigma,q)}\sigma$. 
Then Wang's equation \eqref{WangEquation} is equivalent to the following semilinear elliptic partial differential equation for $u$:
\begin{equation}\label{eq:Blaschke_family_no_t}
\Delta_{\sigma}u= 2e^{u}-4e^{-2u} \frac{|q|^2}{\sigma^3}-2.
\end{equation}
Equation \eqref{eq:Blaschke_family_no_t} has a unique smooth solution $u$ for any given $(\sigma,q)\in \tilde{Q}_3(S)$ (Recall Proposition \ref{pr:wang}).

\begin{lem}\label{lem,smoothness}
    Let $\cal{S}^{\cal{Q}}\subset \tilde{Q}_3(S)$ be a slice as in \eqref{eq:slice_sq}.
    Then the map $\widetilde{\mathbf{g}}\big|_{\cal{S}^{\cal{Q}}}: \cal{S}^{\cal{Q}}\to \cal{M}^{k,\alpha}_{-}$ is smooth.
\end{lem} 
\begin{proof}
Writing $g=e^u\sigma$, where $u$ satisfies \eqref{eq:Blaschke_family_no_t}, it suffices to show that the map $(\sigma,q)\mapsto  u(\sigma,q)\in C^{k,\alpha}(S)$ is smooth on $\cal{S}^{\cal{Q}}$.
Define a map $F: \cal{S}^{\cal{Q}}\times C^{k,\alpha}(S) \to C^{k-2,\alpha}(S)$ by
\begin{equation}
    F((\sigma,q),u)= \Delta_\sigma  u-2e^u+4e^{-2u}\frac{|q|^2}{\sigma^3}+2.
\end{equation}
Then $u(\sigma,q)$ is given implicitly by $F((\sigma,q),u(\sigma,q))=0$

We claim that the map $F$ is smooth. The smoothness of the map $\cal{S}\times C^{k,\alpha}(S)\ni (\sigma,u)\mapsto \Delta_\sigma u\in C^{k-2,\alpha}(S)$ follows from the coordinate expression (with summation convention)
\begin{equation}\label{eq:laplace_in_coords}
	\Delta_\sigma u=\sigma^{k\ell}\partial_{k\ell}^2u+\partial_k\sigma^{k\ell}\partial_\ell u+\frac{\sigma^{k\ell}}{2\det \sigma}\partial_k(\det \sigma) \partial_\ell u.
\end{equation}
The map $C^{k,\alpha}(S)\ni u\mapsto e^u\in C^{k,\alpha}(S)$ is smooth because the exponential map is smooth.
Finally, the map $(\sigma,q)\mapsto \frac{|q|^2}{\sigma^3}$ is seen to be smooth by expressing $q(\sigma)=\sum_{k=1}^{5\mathscr{G}-5}b^k q_k(\sigma)$ in terms of an orthonormal frame, where $b^k\in \bb{C}$, and noting that 
$\displaystyle \frac{|q|^2}{\sigma^3}=\sum_{k,\ell}b^k\bar{b}^\ell \frac{q_k(\sigma)\overline{q_\ell(\sigma)}}{\sigma^3}$.

We then show the partial derivative $\partial_uF\big|_{((\sigma,q),u)}:C^{k,\alpha}(S)\to C^{k-2,\alpha}(S)$ is an isomorphism. Computing the derivative of $F$ with respect to $u$ at a fixed $(\sigma,q)$  gives
\begin{equation}\label{der_first_comp}
\partial_uF\big|_{((\sigma,q),u)}v=\Delta_\sigma  v-2e^u v-8e^{-2u}\frac{|q|^2}{\sigma^3}v=(\Delta_\sigma-2e^u -8e^{-2u}\frac{|q|^2}{\sigma^3}) v.    
\end{equation}
We observe that $P:=\partial_uF\big|_{((\sigma,q),u)} =\Delta_\sigma-2e^u -8e^{-2u}\frac{|q|^2}{\sigma^3}: C^{k,\alpha}(S)\to C^{k-2,\alpha}(S)$ is a  bounded linear  operator.  
It has trivial kernel: if $v\in C^{k,\alpha}(S)$ satisfies $ Pv=0$ we have, by the divergence theorem,
\begin{equation}
    0=\langle P v,v\rangle_{L^2_\sigma(S)}=-\|d v\|_{ L^2_\sigma(S)}^2-\big\| \big({ 2e^u +8e^{-2u}\frac{|q|^2}{\sigma^3}}\big)^{1/2}v\big\|^2_{ L^2_\sigma(S)},
\end{equation}
implying that $v=0$.
To see that it is surjective, consider $w\in C^{k-2,\alpha}(S)\subset \cal{H}^{k-2}(S)$.
Since $P$ is self-adjoint with respect to the $L^2_\sigma$ inner product, it has index 0 as an operator $P:\cal{H}^k(S)\to \cal{H}^{k-2}(S)$ (see \cite[Theorem 8.1]{shubin2001pseudodifferential}) and  thus it is surjective since it is injective.
So there exists $v\in \cal{H}^k(S)$ so that $ P v=w\in C^{k-2,\alpha}(S)$, and by elliptic regularity we see that $v\in C^{k,\alpha}(S)$.
Since $\partial_u F\big|_{((\sigma,q),u)}=P:C^{k,\alpha}(S) \mapsto C^{k-2,\alpha}(S)$ is a continuous bijection between Banach spaces, it is an isomorphism by the open mapping theorem.
Since $\cal{S}^{\cal{Q}}$ is a finite dimensional manifold, by the Banach manifold implicit function theorem (see, e.g.
\cite[Theorem I.5.9]{LangSerge2012FoDG}), the function $u=u(\sigma,q)$ is smooth and therefore $\widetilde{\mathbf{g}}(\sigma, q)=e^{u(\sigma,q)}\sigma$ is smooth.
\end{proof}

We now continue our discussion involving the $S^1$ action on $\tilde{Q}_3(S)$ and $Q_3(S)$ mentioned in Section \ref{subsec, BlaschkeHitchin}. 
The $S^1$ action on $\tilde{Q}_3(S)$ restricts to an action on each slice
$\cal{S}^{\cal{Q}}$, and the identification map $\pi_{\cal{S}^{\cal{Q}}}$  is equivariant with respect to it.
Moreover it  is smooth, proper and free on $\cal{S}^{\cal{Q}*}=\cal{S}^{\cal{Q}}\setminus \cal{O}$ and on $Q_{3}^*(S):=Q_{3}(S)\setminus \cal{O}$ (where $\cal{O}$ is the zero section, which is canonically identified with $\cal{T}(S)$).
We obtain

\begin{cor}
The respective quotients $Q_{3}^*(S)/S^1$ and $\cal{S}^{\cal{Q}*}/S^1$ are smooth manifolds of real dimension $16\mathscr{G}-17$.
\end{cor}The following viewpoint of $Q_{3}^*(S)/S^1$ will be used in later proofs.
\begin{rem}
Since the $\cal{D}_0$ and $S^1$ actions on $\tilde{Q}_{3}(S)$ commute with each other, we can identify
$Q_{3}^*(S)/S^1\cong ((\tilde{Q}_{3}(S)\setminus \cal{O})/S^1)/\cal{D}_0$, that is, with the $S^1$ quotient taken before the $\cal{D}_0$ quotient.
\end{rem}

Next we show in Lemma \ref{lem,injectiveDiff} below that 
$ \widetilde{\mathbf{g}}\big|_{ \cal{S}^{\cal{Q}^*}}$ descends to a map on  $\cal{S}^{\cal{Q}^*}/S^1$ which is an immersion into $\cal{M}^{k,\alpha}_{-}$. 
Recall that, as in the finite dimensional case, a $C^k$ map $f:E_1\to E_2$ between $C^k$ Banach manifolds is an \emph{immersion} if for all $z\in E_1$ there exists an open neighborhood $Z$ of $z$ such that $f\big|_{Z}$ induces a $C^k$ diffeomorphism of $Z$ onto a submanifold of $E_2$;
equivalently, if for every $z\in E_1$ it has injective differential with image that \emph{splits} (see \cite{LangSerge2012FoDG}).
The latter condition means that $df\big|_z$ has closed range $ \mathrm{Ran}(df\big|_z)$ and there exists a closed Banach space $B\subset T_{f(z)}E_2$ which is complementary to $\mathrm{Ran}(df\big|_z)$; this condition is satisfied automatically when $E_1$ is finite dimensional.

We remind our reader of our different notations $\widetilde{\mathbf{g}}, \mathbf{g}, \widetilde{\mathbf{g}}_0, \mathbf{g}_0$: the tilde is used for maps before the $\cal{D}_0$ quotient and the subscript $0$ is used for maps after taking the $S^1$ quotient.  %

\begin{lem}[Injective differential]\label{lem,injectiveDiff} 
Let $(\sigma_0,q_0 )\in \tilde{Q}_3^*(S)=\tilde{Q}_3(S)\setminus \cal{O}$ and consider a slice $\cal{S}^{\cal{Q}_*}:=\cal{S}^{\cal{Q}}\cap \tilde{Q}_3^*(S)$ containing it, where $\cal{S}^{\cal{Q}}$ is as in \eqref{eq:slice_sq} with $\cal{S}$ a slice around $\sigma_0$ as in Theorem \ref{thm:slice}.
Suppose that for  $X\in T_{(\sigma_0, q_0)}\cal{S}^{\cal{Q}_*}$ one has
\begin{equation}\label{eq:differential_g}
    d \widetilde{\mathbf{g}}\big|_{(\sigma_0, q_0)}X=D_{g_0} \chi , \quad \chi\in \sf{S}_1^{k+1,\alpha}(S),\quad g_0:=\widetilde{\mathbf{g}}(\sigma_0, q_0)
\end{equation} 
where $\widetilde{\mathbf{g}}$ is viewed as a map into $\cal{M}^{k,\alpha}_-$ and $D_{g_0}$ is the symmetric differential (see Section \ref{subsec,symmetricTensors}).
Then $X$ is tangent to the $S^1$ orbit through $(\sigma_0,q_0)$ in $\cal{S}^{\cal{Q}_*}$ and so it projects under the quotient map to the zero vector in $T_{(\sigma_0, \langle q_0 \rangle)}(\cal{S}^{\cal{Q}_*}/S^1)$. 
Therefore, the descended map $\widetilde{\mathbf{g}}_0:\cal{S}^{\cal{Q}*}/S^1\to \cal{M}_-^{k,\alpha}$ has injective differential at $(\sigma_0,\langle q_0\rangle)$ whose image splits. The map $\widetilde{\mathbf{g}}_0$ is an immersion from a sufficiently small neighborhood $\cal{U}$ of $(\sigma_0,\langle q_0\rangle)\in \cal{S}^{\cal{Q}_*}/S^1$ into $\cal{M}^{k,\alpha}_{-}$.
The image  $\cal{V}:=\widetilde{\mathbf{g}}_0(\cal{U})$ is a $16\mathscr{G}-17$ dimensional submanifold of $\cal{M}_{-}^{k,\alpha}$ consisting of smooth Blaschke metrics and satisfying $T_g\cal{V}\cap T_g\sf{O}_g^{k+1,\alpha}=\{0\}$ for $g\in \cal{V}$, where $\sf{O}_g^{k+1,\alpha}$ is the $\cal{D}_0^{k+1,\alpha}$-orbit through $g$.
\end{lem}
\begin{proof}
By means of a trivialization for $\cal{S}^{\cal{Q}}$ we can identify  $X\in T_{(\sigma_0,q_0)}\cal{S}^{\cal{Q}}$ with a pair $(h^{TT},v)$, where $h^{TT}\in T_{\sigma_0}\cal{S}=\sf{S}_2^{\sigma_0,TT}(S)$ and $v\in \bb{R}^{10\mathscr{G}-10}$.
 Consider a curve $\alpha_t=(\sigma_t,q_t):(-\epsilon,\epsilon)\to \cal{S}^{\cal{Q}*}$ with $\alpha_0=(\sigma_0,q_0)$ and $\dot{\alpha}_0=(h^{TT},v)$.
We  write $u(\sigma,q)$ for the smooth solution of \eqref{eq:Blaschke_family_no_t} for given $(\sigma,q) \in \cal{S}^{\cal{Q}*}$, and we also write $u_t=u(\alpha_t)$,  $\dot{u}_t=\frac{d}{dt}u_t$, so that $g_0=\widetilde{\textbf{g}}(\sigma_0,q_0)=e^{u_0}\sigma_0$.
By our hypothesis,
\begin{equation}\label{eq:F_der}
D_{g_0}\chi=\frac{d}{dt}\widetilde{\mathbf{g}}(\sigma_t,q_t)\big|_{t=0}= \frac{d}{dt}(e^{u(\sigma_t,q_t)}\sigma_t)\big|_{t=0}    =
e^{u_0}h^{TT}+e^{u_0}\dot{u}_0\sigma_0.
\end{equation}
Taking the $L^2_{g_0}(S)$ inner product of equation \eqref{eq:F_der} with  $h^{TT},$
\begin{equation}\label{eq:paired}
    \frac{1}{\mathrm{Area}(S,g_0)}\int_S e^{u_0}|h^{TT}|^2_{g_0} dv_{g_0}+ \langle g_0 \dot{u}_0, h^{TT}\rangle_{L^2_{g_0}(S)} -\langle D_{g_0}\chi,h^{TT}\rangle_{L^2_{g_0}(S)}=0.
\end{equation}
Since we are in dimension two, the tensor $h^{TT}$ is trace free and divergence free with respect to any metric that is conformal to $ \sigma_0$, in particular with respect to $g_0$. So the second and third term of \eqref{eq:paired} vanish. 
We conclude that $h^{TT}=0$ and so  $ \dot{u}_0g_0 =D_{g_0}\chi$.

We  now apply the X-ray transform $I_2^{g_0}$ on both sides of the equation $ \dot{u}_0g_0 =D_{g_0}\chi$.
Since $I_2^{g_0} (D_{g_0}\chi)(c)=0$ for every $\chi\in \sf{S}_1^{k+1,\alpha}(S)$ and for every closed orbit $c$ of the geodesic flow of $g_0$ (see \cite{GeodesicStretch}), we have
\begin{equation}
  0=\frac{1}{L(c)}\int_0^{L(c)} \dot{u}_0(\gamma(s))g({\gamma}'(s),{\gamma}'(s))\mathrm{d}s
  =\frac{1}{L(c)}\int_0^{L(c)}\dot{u}_0(\gamma(s))\mathrm{d}s=  I_0^{g_0}(\dot{u}_0)(c),
\end{equation}
where the orbit $c$ is parametrized as $(\gamma(s),{\gamma}'(s)):[0,L(c)]\to T^1_{g_0}S$.
By \cite[Theorem 3.6]{Guillemin1980},  $I_0^{g_0}$ is injective on a compact Riemannian surface $(S,g_0)$ with ${g_0}\in \cal{M}_{-}$ (also see \cite{Croke1998} for the higher dimensional case), therefore $\dot{u}_0=0$.

In summary, $(h^{TT},v)=(0,v)$ and $du\big|_{(\sigma_0,q_0)}(h^{TT},v)=du\big|_{(\sigma_0,q_0)}(0,v)=0$.
So $\dot{\alpha}_0=\dot{\tilde{\alpha}}_0$, where $\tilde{\alpha}_t$ is of the form $\tilde{\alpha}_t=(\sigma_0,\tilde{q}_t)$, and $\frac{d}{dt}u\big(\tilde{\alpha}_t\big)\big|_{t=0}=0$.
Differentiating  Wang's equation (\ref{WangEquation}) along $\tilde{\alpha}_t$ and using that $\dot{u}_0=0$, we find that
\begin{equation}
    0=\frac{d}{dt}|\tilde{q}_t|^2_{g_0}\Big|_{t=0}=
    \frac{\tilde{q}_t\overline{\partial_t{\tilde{q}}_t}+\overline{\tilde{q}_t}{\partial_t{\tilde{q}}_t}}{g^3_0}\Big|_{t=0}
    =
   |\tilde{q}_0|^2_{g_0}\left( \frac{\partial_t\tilde{q}_t}{\tilde{q}_t}+\overline{\frac{\partial_t\tilde{q}_t}{\overline{\tilde{q}_t}}}\right) \Big|_{t=0}.
\end{equation}
Note that $\tilde{q}_t$ are all holomorphic cubic differentials with respect to $J=J(\sigma_0)$, so $\frac{\tilde{q}_t}{\tilde{q}_0}$ is a family of meromorphic functions with respect to $J$. So the above equation implies that the meromorphic function $\partial_t (\frac{\tilde{q}_t}{\tilde{q}_0})\big|_{t=0}$ on $S$ is purely imaginary, and therefore constant. 
We conclude that $\partial_t\tilde{q}_{t}|_{t=0}=\lambda i {\tilde{q}_0}$, $\lambda\in \bb{R}$, which precisely  characterizes elements in the tangent space of the $S^1$ orbit through $\tilde{q}_0=q_0$.
Thus  we have the  claim.
It is immediate now that $d \widetilde{\mathbf{g}}_0\big|_{T_{(\sigma_0,q_0)}(\cal{S}^{\cal{Q}^*}/S^1)}$ is injective. Moreover, its image splits because $T_{(\sigma_0, q_0)}(\cal{S}^{\cal{Q}^*}/S^1)$ is finite dimensional.
Thus $\widetilde{\mathbf{g}}_0$ is an immersion at $(\sigma_0,\langle q_0\rangle)$, and thus also in a sufficiently small neighborhood of it by definition of an immersion, with its image being a submanifold of $\cal{M}^{k,\alpha}_-$ which consists of smooth Blaschke metrics.

For the last statement, recall that $T_{g}\sf{O}^{k+1,\alpha}_{g}$ consists exactly of potential symmetric 2-tensors. So
\begin{equation}
    T_{g_0}\cal{V}\cap T_{g_0}\sf{O}^{k+1,\alpha}_{g_0}=\{0\} ,\label{transv}
\end{equation} by the first statement of the lemma.
This implies that
the same holds for a sufficiently small neighborhood  of $g_0$. 
To see this, consider a smooth nonvanishing section $Y(g)$ of $T\cal{V}\subset T\cal{M}^{k,\alpha}_-=\sf{S}_2^{k,\alpha}(S)$.
Then by \eqref{transv} we have
$\|\pi_{\ker _{D_{g_0}^*}}  Y(g_0)\|_{C^{k,\alpha}}>0$, where $\pi_{\ker _{D_g^*}}$ denotes the orthogonal projection with respect to the decomposition \eqref{eqtn, Divergences}.
This is because the potential tensor fields are precisely those which are annihilated by $\pi_{\ker _{D_{g_0}^*}} $.
The operator $\pi_{\ker _{D_{g}^*}}$ is a bounded operator on $\sf{S}_2^{k,\alpha}(S)$ as a pseudodifferential operator of order 0, and it depends continuously on the metric $g$ in a $C^\infty$ neighborhood of $g_0$  (see for example the proof of \cite[Lemma 2.2]{GeodesicStretch}).
Therefore, $\|\pi_{\ker _{D_{g}^*}}  Y(g)\|_{C^{k,\alpha}}$ is nonzero for $g$ in a neighborhood of $g_0$, which implies $T_{g}\cal{V}\cap T_{g}\sf{O}^{k+1,\alpha}_{g}=\{0\}$ near $g_0$.
\end{proof}

Recall that our final goal is to construct smooth charts for the Blaschke locus $\cal{M}^B/\cal{D}_0$ away from $\cal{T}(S)$. 
In Lemma \ref{lem,injectiveDiff} we constructed slices $\cal{V}$ in $\cal{M}^B\subset \cal{M}^{k,\alpha}_-$.
In the following lemma, we show that if $\cal{V}$ are taken sufficiently small, then they locally parametrize $\cal{M}^B/\cal{D}_0$.

\begin{lem}\label{lm:orbits}
If the slice $\cal{V}$ in Lemma \ref{lem,injectiveDiff} is taken to be sufficiently small, then each point of $\cal{V}$ corresponds to exactly one orbit of $\cal{D}_0$.
\end{lem}
\begin{proof}
Let $g_i=e^{u(\sigma_i, \langle q_i\rangle))}\sigma_i\in \cal{V}$, $i=1,2$, be two  metrics in the same $\cal{D}_0$ orbit, that is, $g_2=\psi^* g_1$ for some $\psi\in \cal{D}_0$.
From  $e^{u(\sigma_2, \langle q_2\rangle)}\sigma_2=\psi^* (e^{u(\sigma_1, \langle q_1\rangle)}\sigma_1)$, it follows that both $\sigma_2$ and $\psi^*{\sigma_1}$ are hyperbolic metrics in the same conformal class, so it must be the case that $\sigma_2=\psi^*{\sigma_1}$.
Since $\sigma_1, \sigma_2\in \cal{S}$ and both are in the same $\cal{D}_0$ orbit, we conclude that $\sigma_1=\sigma_2$ by Theorem \ref{thm:slice}. So $\psi=id$ by freeness of the $\cal{D}_0$ action. Thus $g_1=g_2$.
\end{proof}

We now state our main theorem in this section.

\begin{thm}\label{smoothnessBlaschkeLocus}
Away from the Teichmüller space $\cal{T}(S)=\cal{M}_{-1}/\cal{D}_0$, the Blaschke locus $\cal{M}^B/\cal{D}_0$ has the structure of a smooth manifold of real dimension $16\mathscr{G}-17$. Moreover, the map $\mathbf{g}_0:Q_3^*(S)/S^1 \to (\cal{M}^B/\cal{D}_0)\setminus \cal{T}(S)$ is a  diffeomorphism. 
\end{thm}

\begin{proof}
At this point, for each metric $g_0\in \cal{M}^B$ we have constructed a local slice $\cal{V}$ containing it and consisting of smooth Blaschke metrics.
A slice corresponding to $g_0$ is constructed as the image under $\widetilde{\mathbf{g}}_0$ of a neighborhood $\cal{U}\subset \cal{S}^{\cal{Q}*}/S^1$ of $\widetilde{\mathbf{g}}_0^{-1}(g_0) $, and via a trivialization of $\cal{S}^{\cal{Q}*}$ we have $\cal{U}\cong \cal{S}\times \bb{R}^{10\mathscr{G}-11}$.
We have also shown that if $\cal{V}$ is sufficiently small, it is in bijective correspondence with the  $\cal{D}_0$ orbits passing through it. So  $\cal{V}$ parametrizes $\cal{M}^B/\cal{D}_0$ near $[g_0]$. 
To show that the slices $\cal{V}$ together with the corresponding maps $\widetilde{\mathbf{g}}_0^{-1}:\cal{V}\to \cal{U}$ define smooth charts, we have to show that they are smoothly compatible.

Write $\pi_B:\cal{M}^B
\to 
\cal{M}^B/\cal{D}_0$ for the natural projection with respect to the $\cal{D}_0$ quotient. Note that this map is a bijection onto its image when restricting to each slice $\cal{V}$. So we can set
\begin{equation}
    \cal{F}_{\cal{V}}:
\pi_B(\cal{V})\to \cal{U} \cong \cal{S}\times  \bb{R}^{10\mathscr{G}-11}, \quad [g]\mapsto
\widetilde{\mathbf{g}}_0^{-1}(g),
\end{equation}
where $g$ is the unique element in $\cal{V}$ with $g\in [g]$. 
We want to show that if $\cal{V}=
\widetilde{\mathbf{g}}_0(\cal{U})$ and $\cal{V}'=
\widetilde{\mathbf{g}}_0(\cal{U}')$ are different slices whose images $
\pi_B(\cal{V})$, $
\pi_B(\cal{V}')$  have nonempty overlap, then the change of charts $ \cal{F}_{\cal{V}'}\circ  \cal{F}_{\cal{V}}^{-1}: {\cal{U}} \to \cal{U}'$ is smooth where it is defined.
Here $\cal{U}'\subset (\cal{S}')^{\cal{Q}*}/S^1$ is a chart over a different slice $\cal{S}'$.
To this end, write 
$$\cal{F}_{\cal{V}'}\circ \cal{F}_{\cal{V}}^{-1}(\sigma,\langle q\rangle)=\cal{F}_{\cal{V}'}([e^{u(\sigma,\langle q\rangle)}\sigma]).$$
Now because $[g]=[e^{u(\sigma,\langle q\rangle)}\sigma] \in \pi_B(\cal{V})\cap \pi_B({\cal{V}}{'})$, there exists a unique element 
$e^{u(\sigma',\langle q'\rangle)}\sigma'\in \cal{V}{'}$ such that $[e^{u(\sigma',\langle q'\rangle)}\sigma']=[e^{u(\sigma,\langle q\rangle)}\sigma]$.
Therefore, there exists an element $\psi=\psi(\cal{V},\cal{V}',[g])\in \cal{D}_0$ 
such that $e^{u(\sigma',\langle q'\rangle)}\sigma'=\psi^*(e^{u(\sigma,\langle q\rangle)}\sigma).$
So we have
\begin{equation}
    (\sigma',\langle q'\rangle)=
    \cal{F}_{\cal{V}'}\circ \cal{F}_{\cal{V}}^{-1}(\sigma,\langle q\rangle)= \widetilde{\mathbf{g}}_0^{-1}\big(e^{u(\sigma',\langle q'\rangle)}\sigma'\big) =\widetilde{\mathbf{g}}_0^{-1}\big(\psi^*(e^{u(\sigma,\langle q\rangle)}\sigma)\big) \in \cal{U}'.
\end{equation}
In other words, $(\sigma',\langle q'\rangle)$ is the unique element in $\cal{U}'$ such that 
\begin{equation}
    e^{u(\sigma',\langle q'\rangle)}\sigma'= \psi^*(e^{u(\sigma,\langle q\rangle)}\sigma)=e^{\psi^*(u(\sigma,\langle q\rangle))}\psi^*\sigma.
\end{equation}
Now $\sigma'$ and $ \psi^*\sigma$ are both hyperbolic metrics in the same conformal class. Therefore $\sigma'= \psi^*\sigma$. So by the freeness of the $\cal{D}_0$ action on $\cal{M}_{-1}$ we see that $\psi=\psi_{\sigma}$ is unique and is determined only by the pair $\sigma$ and $\sigma'$. In other words, it does not depend on $\langle q\rangle$.
This implies that
\begin{equation}
    u(\sigma',\langle q'\rangle)=\psi_{\sigma}^*\big( u(\sigma,\langle q\rangle)\big)\implies u(\sigma',\langle q'\rangle)=  u(\psi_{\sigma}^*(\sigma,\langle q)\rangle)
    =u(\sigma',\psi_{\sigma}^*\langle q\rangle)
\implies
\langle q'\rangle =\psi_{\sigma}^*\langle q\rangle
\end{equation}
by injectivity of the map $\langle q\rangle\mapsto u( \sigma,\langle q\rangle)$ for each $\sigma$.
Therefore by \eqref{Psimap},
\begin{equation}
    (\sigma',\langle q'\rangle)=\cal{F}_{\cal{V}'}\circ \cal{F}_{\cal{V}}^{-1}(\sigma,\langle q\rangle)=\psi_\sigma ^*(\sigma,\langle q\rangle)=(\psi_\sigma ^*\sigma,\langle \psi_\sigma ^* q\rangle)=\langle\tilde{\Psi}_{\cal{S},\cal{S}'}(\sigma, q)\rangle.
\end{equation}  Hence $\cal{F}_{\cal{V}'}\circ \cal{F}_{\cal{V}}^{-1}$ is smooth.

The second statement is a consequence of Lemmas \ref{lem,smoothness} and \ref{lem,injectiveDiff} by taking $\cal{U}\subset \cal{S}^{Q^*}/S^1$ as a chart for $Q_3^*(S)/S^1$ near $[(\sigma_0,\langle q_0\rangle)]$ and $\cal{V}=\widetilde{\mathbf{g}}_0(\cal{U})$ as a chart for $\cal{M}^B/\cal{D}_0$ near $[g_0]=[e^{u_0}\sigma_0]=\mathbf{g}_0([(\sigma_0,\langle q_0\rangle)])$.
Then since $\widetilde{\mathbf{g}}_0$ is an immersion at $(\sigma_0,\langle q_0\rangle )$, it is a diffeomorphism onto its image upon shrinking $\cal{U}$ if necessary. 
Thus $\mathbf{g}_0:Q_3^*(S)/S^1 \to (\cal{M}^B/\cal{D}_0)\setminus \cal{T}(S)$ is a local diffeomorphism. Because it is bijective (Proposition \ref{equalBlaschke}), it is in fact a global diffeormorphism from $Q_3^*(S)/S^1$ to $(\cal{M}^B/\cal{D}_0)\setminus \cal{T}(S)$.
\end{proof}

\begin{cor}
The quadratic form $G_g(\cdot, \cdot)$ defined in Section \ref{sec:PressureMetric} defines a $C^{\infty}$ Riemannian metric on $(\cal{M}^B/\cal{D}_0)\setminus \cal{T}(S)$.
\end{cor}

\begin{proof}
We saw in Proposition  \ref{def:pressure} that $G_{[g]}(\cdot, \cdot)$ defines a Riemannian metric on $\cal{M}_-/\cal{D}_0$, so in particular on the Blaschke locus. 
Since the slices $\cal{V}$ we constructed are smooth charts for $\cal{M}^B /\cal{D}_0$ and they are also finite dimensional smooth submanifolds of $\cal{M}^{k,\alpha}_{-}$, we conclude by Proposition \ref{prop,smoothness} that $G_g$ restricts to a $C^{k-3}$ quadratic form on each slice.
Thus it is a $ C^{k-3}$ metric on $\cal{M}^B /\cal{D}_0 $ away from $\cal{T}(S)$. 
Notice however that since by Theorem \ref{smoothnessBlaschkeLocus} the map $\mathbf{g}_0:Q_3^*(S)/S^1 \to (\cal{M}^B/\cal{D}_0)\setminus \cal{T}(S)$ is a
diffeomorphism regardless of the space $\cal{M}_-^{k,\alpha}$ of which $\cal{M}^B$ 
was viewed as a subset, we see that we can choose the $k$ arbitrarily large, concluding that the metric $G_g(\cdot,\cdot)$ on $\cal{M}^B/\cal{D}_0$ is smooth away from $\cal{T}(S)$.
\end{proof}

\subsection{Topology of the Blaschke locus}\label{subsec:TopologyBlaschke}

The results of the previous section allow us to elaborate on the topological structure of the Blaschke locus.

\begin{prop}
\label{prop,homeo}
The quotient space $Q_3(S) /S^1 $ is homeomorphic to $ \mathcal{M}^B /\mathcal{D}_0$, when the two spaces are endowed with quotient topologies resulting from $C^\infty$ topologies.
\end{prop}

\begin{proof}

The map $\mathbf{g}_0:Q_3(S)/S^1 \to \cal{M}^B/\cal{D}_0$ is bijective by Proposition \ref{equalBlaschke}, and we will first show that it is continuous.
Each slice $\cal{S}^{\cal{Q}}$ parametrizing $Q_3(S)$ carries the smooth structure (and topological structure) of $Q_3(S)$, see  Remark \ref{rmk_topology}, and we showed in Lemma \ref{lem,smoothness}
that the map
$\widetilde{\mathbf{g}}:\cal{S}^{\cal{Q}}\to \cal{M}^{k,\alpha}_-$ is smooth, so in particular continuous (note that the smoothness also works at the zero section).
This is true for any $k$, $\alpha$, so 
$\widetilde{\mathbf{g}}:\cal{S}^{\cal{Q}}\to \cal{M}_-$ is continuous.
Therefore, writing $\mathbf{g}$ locally as
\begin{equation}
    \mathbf{g}=\pi_B\circ \widetilde{\mathbf{g}}\circ \pi_{\cal{S}^{\cal{Q}}}^{-1}:\pi_{\cal{S}^{\cal{Q}}}(\cal{S}^{\cal{Q}})\subset Q_3(S)\to     \cal{M}^B/\cal{D}_0\subset \cal{M}_-/\cal{D}_0
\end{equation}
(see \eqref{projections}), we see that it is continuous.
Since $\mathbf{g}$ is constant on the orbits of the $S^1$ action on $Q_3(S)$, the descended map $\mathbf{g}_0:Q_3(S)/S^1\to \cal{M}^B/\cal{D}_0$ is continuous.

We then show that $\mathbf{g}^{-1}_0: \cal{M}^B/\cal{D}_0 \to Q_3(S)/S^1 $ is continuous. 
Suppose that $[g_j]\overset{j\to \infty}{\to}[g]\in \cal{M}^B/\cal{D}_0$ with the quotient topology of $\cal{M}_-/\cal{D}_0$.
Then we can consider lifts $g_j=e^{u_j}\sigma_j$ of $[g_j]$ to a slice $\cal{W}\subset \cal{M}_-$ about a lift $g=e^u\sigma\in \cal{M}_-$ of $[g]$ as in Section \ref{subsec: negMetrics}, which converge to $g$ in $C^\infty$.
Here $\sigma_j$, $\sigma\in \cal{M}_{-1}$.
Then we have that $\sigma_j=\lambda(g_j)\overset{j\to \infty}{\to }\lambda(g)=\sigma$ in $C^\infty$, where $\lambda $ is the Poincar\'e map.
Now fix a large integer $k$, and $\alpha\in (0,1)$.
The $\sigma_j$ do not necessarily lie in a slice $\cal{S}\subset \cal{M}_{-1}^{k,\alpha}$ through $\sigma$ as in Theorem \ref{thm:slice}, but for such a slice there exist diffeomorphisms $\psi_{\sigma_j}\in \cal{D}_0$ such that\footnote{A priori, $\psi_{\sigma_j}$ are only $C^{k+1,\alpha}$, but because $\sigma_j$ are all smooth and the slice $\cal{S}$ consists of smooth metrics, they are actually smooth, see the proof of \cite[Theorem 2.3.1]{Tromba_Book}.} $\psi_{\sigma_j}^*\sigma_j \in \cal{S}$.
Moreover, $\psi_{\sigma_j}$ depend continuously on $\sigma_j$ in the $C^{k+1,\alpha}$ topology. So since $\sigma_j\to \sigma\in \cal{S}$, it follows that $\psi_{\sigma_j}\to id$ in $C^{k+1,\alpha}$.
Therefore, $\psi_{\sigma_j}^*g_j\to g$ in $C^{k,\alpha}$, and we have
 $\lambda(\psi_{\sigma_j}^*g_j)=\psi_{\sigma_j}^*\lambda(g_j)\in \cal{S} $ and $$\mathbf{g}_0^{-1}([g_j])=[\widetilde{\mathbf{g}}_0^{-1}(g_j) ]=[\psi_{\sigma_j}^*\widetilde{\mathbf{g}}_0^{-1}(g_j) ]=[\widetilde{\mathbf{g}}_0^{-1}(\psi_{\sigma_j}^*g_j) ] .$$
Thus we may replace our assumption that $g_j\to g $ in $C^\infty$ with the assumption  $g_j\to g $ in $C^{k,\alpha}$, where $\sigma_j=\lambda({g_j})\in \cal{S}$ converge to $\sigma =\lambda(\sigma)$ in $C^{k,\alpha}$, and we want to show that $[\widetilde{\mathbf{g}}_0^{-1}(g_j) ]\to [\widetilde{\mathbf{g}}_0^{-1}(g) ]=:  [(\sigma,\langle q\rangle)]$. 
As already mentioned in Remark \ref{rem,equivalentTopology}, $\cal{S}\ni \sigma_j\to\sigma $ in $C^{k,\alpha}$ actually implies that $\sigma_j\to \sigma $ in $C^\infty$.

So now we would like to show that $\langle q_j\rangle \to \langle q\rangle$ in $\cal{S}^{\cal{Q}}/S^1$.
Notice
that by Wang's equation \eqref{WangEquation},
for any $q_j$, $q$ such that $g_j=\widetilde{\mathbf{g}}(\sigma_j,q_j)$, $g=\widetilde{\mathbf{g}}(\sigma,q)$, one has
$|q_j|^2_{g_j}\to |q|^2_g$ in $C^{k-2,\alpha}$ and therefore in particular in $C^0$.
Moreover, since $g_j=e^{u_j}\sigma_j\to g=e^u\sigma$ and $\sigma_j\to \sigma$ in $C^0$  we conclude that $u_j\to u$ in $C^0$.
Therefore, $|q_j|^2_{\sigma_j}=e^{3u_j}|q_j|^2_{g_j} \to |q|^2_\sigma=e^{3u}|q|^2_g$ in $C^0$.
In particular, $\int_S |q_j|^2_{\sigma_j}dv_{\sigma_j}\to \int_S |q|^2_\sigma dv_{\sigma}$.
Now 
\begin{equation}\label{hermitian}
    (q,q')\mapsto\int _S \frac{q\overline{q}'}{\sigma^3}dv_{\sigma}
\end{equation}
is a continuous Hermitian inner  product on the fibers of  $\cal{S}^{\cal{Q}}$, so we can choose a local orthonormal frame $\{\hat{q}_\ell(\sigma)\}_{\ell=1}^{5\mathscr{G}-5}$ with respect to it and write $ q_j=\sum_{\ell} b_j^\ell \hat{q}_\ell(\sigma_j)$, $ q=\sum_{\ell} b^\ell \hat{q}_\ell(\sigma)$ for $b^\ell,b_j^\ell\in \mathbb{C}$.
Then 
\begin{equation}
    \sum _\ell |b_j^\ell |^2=\int _S|q_j|_{\sigma_j} ^2dv_{\sigma_j}\overset{j\to \infty}{\to} \int _S|q|_{\sigma} ^2dv_{\sigma}=\sum_{\ell }|b^\ell|^2.
\end{equation}
In particular, the sequence $b_j=(b_j^1,\dots,b_j^{5\mathscr{G}-5})\in \bb{C}^{5\mathscr{G}-5}$ is bounded, and for any limit point $\tilde{b}\in \bb{C}^{5\mathscr{G}-5}$
of it, the cubic differential $\tilde{q}(\sigma)=\sum_{\ell}\tilde{b}^\ell \hat{q}_\ell (\sigma)$ is holomorphic with respect to the conformal structure determined by $\sigma$ and satisfies $|\tilde{q}|^2_\sigma=|q|_\sigma^2$.
If $q\equiv 0$, then we have $\tilde{b}=0$, and thus $q_j\to q$.
If $q\not\equiv 0$, then the ratio $\tilde{q}/q$ defines a meromorphic function on $(S,J(\sigma))$, and since $|\tilde{q}/q|^2=|\tilde{q}|^2_\sigma/|{q}|^2_\sigma=1$,
it has values in $S^1$.
This implies that it is constant, and therefore $\tilde{q}=e^{2\pi i \theta}q$, $\theta \in [0,1)$.
We conclude that $\langle q_j\rangle \to \langle q\rangle$ in $\cal{S}^{\cal{Q}}$ in the $S^1$ quotient topology originating from $C^0$, which is equivalent to the one originating from $C^\infty$.
Therefore,
$[(\sigma_j,\langle q_j\rangle)]$ converges to $[(\sigma,\langle q \rangle)]$ in $Q_3(S) /S^1 $.
\end{proof}

Using Proposition \ref{prop,homeo}, we now have:

\begin{thm}\label{thm:contratible}
$\mathcal{M}^B /\mathcal{D}_0 $ is a $16\mathscr{G}-17$ dimensional connected contractible space.
\end{thm}

\begin{proof}

Each  fiber of $Q_3(S)/S^1$ is isomorphic to $\mathbb{C}^{n}/S^1$, where $n=5\mathscr{G}-5$, by the Riemann-Roch Theorem. Taking away $0$ from $\mathbb{C}^{n}$, we have the following homeomorphism (actually diffeomorphism),
\begin{align*} 
     (\mathbb{C}^{n}\setminus \{0\})/S^1 &\longrightarrow (0,\infty) \times \mathbb{C}P^{n-1},\\
    \big\langle (z_1,\cdots, z_n)\big\rangle &\longmapsto (\big|(z_1,\cdots, z_n)\big|,\big \lfloor z_1,\cdots, z_n \big \rfloor ). 
\end{align*}
Here $\big\langle\cdot\big\rangle$ denotes the equivalence class for the $S^1$ action that identifies $(z_1, z_2, \cdots, z_n)$ with $e^{ 2\pi i\theta}(z_1, z_2, \cdots, z_n)$ for any $\theta\in [0,1)$ and $\big \lfloor \cdot \big \rfloor$ denotes the equivalence class for projective lines in $\mathbb{C}^{n}$. 

Therefore $\mathbb{C}^{n}/S^1$ is homeomorphic to the cone given by $\big([0,\infty) \times \mathbb{C}P^{n-1}\big)/_\sim$. The relation $``\sim"$ glues $\{0\}\times \mathbb{C}P^{n-1}$ to a single point and is trivial otherwise. A deformation retraction of $\big([0,\infty) \times \mathbb{C}P^{n-1}\big)/_\sim$ to a point is given by $f_t(\{(r, x)\})=\{((1-t)r,x)\} \in\big([0,\infty) \times \mathbb{C}P^{n-1}\big)/_\sim$. Here $r\in[0,\infty)$, $x\in\mathbb{C}P^{n-1}$ and $\{(r,x)\}$  denotes the equivalence class of $(r,x)$ in $\big([0,\infty) \times \mathbb{C}P^{n-1}\big)/_\sim$.
Since $\mathcal{M}^B /\mathcal{D}_0 $ is homeomorphic to $Q_3(S)/S^1$ and $Q_3(S)$ is a trivial vector bundle over the contractible space $\cal{T}(S)$, we conclude that $\mathcal{M}^B /\mathcal{D}_0$ is homeomorphic to the product space of $\cal{T}(S)$ and $([0,\infty) \times \mathbb{C}P^{n-1})/_\sim$. The result follows.
\end{proof}

\begin{rem}
    The proof of Theorem \ref{thm:contratible} shows that $\mathcal{M}^B /\mathcal{D}_0$ is topologically the product of $\cal{T}(S)$ with a cone having as base the compact manifold $\bb{C}P^{n-1}$. This implies that it can be viewed as a \emph{wedge} or \emph{``manifold'' with edges}  in the sense of \cite[p. 266]{Schulze1991}, with the difference that here the edge $\cal{T}(S)$ is not a closed manifold.
    It would be interesting to know whether the covariance metric actually behaves like (a conformal multiple of) an edge-type metric near $\cal{T}(S)$ (see \cite[§2]{Mazzeo1991}), though we have not pursued this.
\end{rem}

We conclude with the relation between the Blaschke locus $\mathcal{M}^B/\mathcal{D}_0$  and the Hitchin component $\h3 $. The following corollary is an immediate consequence of Theorem \ref{thm, HtichinMap} and Proposition \ref{prop,homeo}.

\begin{cor}\label{cor, mappingClassEquiv}
The $S^1$ action on $ Q_3(S) $ induces a $S^1$ action on $\h3 $ by the Hitchin map $\mathbf{H}: Q_3(S) \to \h3 $. 
Denote the quotient space by $\h3 /S^1$ and the descending map by $\mathbf{H}_0 : Q_3(S)/S^1 \to \h3 /S^1$. Then the  composition 
$$\Phi:=   \mathbf{g}_0\circ\mathbf{H}_0^{-1}: \h3 /S^1 \to \mathcal{M}^B /\mathcal{D}_0$$
is a mapping class group equivariant homeomorphism, where  the mapping class group actions on $\mathcal{H}_3(S)/S^1$ (as outer automorphism group action) and on $\mathcal{M}^B /\mathcal{D}_0$ are the left actions introduced in Section \ref{sec: MappingClassGrp}.
\end{cor}

\section{Geodesics in \texorpdfstring{$\mathcal{M}^B /\mathcal{D}_0$}{M B/D 0}}
\label{sec:geodesics}
We will study some families of geodesics with respect to the covariance metric in the Blaschke $\mathcal{M}^B /\mathcal{D}_0$. In Section \ref{subsec:geodesics}, we identify some geodesics in $\mathcal{M}^B /\mathcal{D}_0$ leaving all compacts sets, using the Hitchin orbifold representations introduced in Section \ref{orbifolds}, and we estimate their lengths with respect to covariance metric in Subsection \ref{infiniteLength}.

\subsection{Geodesics in the locus \texorpdfstring{$\mathcal{M}^B /\mathcal{D}_0$}{M B/D 0}} \label{subsec:geodesics}
 
Throughout this section, we fix a presentation $Y \simeq [S/ \Sigma]$ of an orbifold $Y$, where $\Sigma$ is a finite subgroup of the diffeomorphism group $\mathcal{D}$ of the surface $S$. Moreover, we assume that $Y=Y_J$ descends from a Riemann surface $X_J=(S,J)$ so that $Y \simeq [X_J/ \Sigma]$.  In \cite{Orbifolds}, the orbifolds of negative Euler characteristic with $1$-dimensional Hitchin components are classified. These are non-orientable orbifolds. In particular, for $\textnormal{Hit}(\pi_1 Y, \mathrm{PGL}(3,\mathbb{R}))$,  one has

\begin{prop}[{\cite[Theorem 6.5]{Orbifolds}}]\label{OneDimHitchin}
Let $Y$ be a non-orientable orbifold of negative Euler characteristic. Then we have
$\textnormal{dim Hit}(\pi_1 Y, \mathrm{PGL}(3,\mathbb{R}))=1$ 
if and only if the orientable double cover $Y^{+}$ of $Y$ satisfies one of the following:
\begin{enumerate}
    \item $Y^+$ is a sphere with 4 cone points of respective orders $m_1=m_2=m_3=2$ and $m_4\geq 4$.
    \item  $Y^+$ is a sphere with 3 cone points of respective orders $m_1\geq 3$, $m_2\geq 3$ and $m_3\geq 4$.
\end{enumerate}
\end{prop}
By Theorem \ref{thm,orbifoldLocus}, the space $\textnormal{Hit}(\pi_1 Y, \mathrm{PGL}(3,\mathbb{R}))$ is homeomorphic to the one dimensional space $\mathcal{H}_3(Y) :=\textnormal{Fix}_{\Sigma}\h3 $. 
 Among the examples of orbifolds $Y \simeq [X_J/ \Sigma]$ given in Proposition \ref{OneDimHitchin}, 
there are examples of $\cal{H}_3(Y) \subset \cal{H}_3(S)$  
which are parametrized by a single holomorphic cubic differential. These are

\begin{prop}[{\cite[Theorem 5.5, Theorem 6.6]{Orbifolds}}]
\label{CyclicOneDimHitchin}
Suppose dim $\mathcal{H}_3(Y)$=1 and $\mathcal{H}_3(Y)$ is parametrized by a single non-vanishing cubic differential, then $Y^+$ must be a sphere with 3 cone points of respective orders $m_1\geq 3$, $m_2\geq 3$ and $m_3\geq 4$.
\end{prop}

\begin{rem}
The Teichmüller space $\cal{T}(Y^+)$ for the orbifolds $Y^+$ in Proposition \ref{CyclicOneDimHitchin} (spheres with 3 cone points) is of dimension $0$ (see \cite[Corollary 13.3.7]{ThurstonBooks}).  Therefore $Y^+$, which is the orientable double cover of $Y$, has a unique complex structure and $Y$ also inherits a unique ``complex structure'', presented as $Y=Y_J \simeq [X_J/ \Sigma]$ (see Remark \ref{rem:OrbifoldComplexStructure}). 
\end{rem}

From now on, we restrict our interest to orbifolds $Y=Y_J$  for which $\mathcal{H}_3(Y)$ is one dimensional and is parametrized by a single non-vanishing holomorphic cubic differential. Recall that  elements in $\Sigma$ act on $X_J$ as holomorphic or anti-holomorphic maps and $Y \simeq [X_J/ \Sigma]$. For $Y$ arising from Proposition \ref{CyclicOneDimHitchin}, the vector space $ \textnormal{Fix}_{
\Sigma} H^{0}(X_{J}, K_{J}^2)$ is trivial. 
So in these cases, we have a homeomorphism $ \textnormal{Fix}_{
\Sigma} H^{0}(X_{J}, K_{J}^3)\overset{H_J}{\simeq}\mathcal{H}_3(Y)$ given by the Hitchin parametrization (recall Proposition \ref{SigmaEquiHitchinSection}). Because $\mathcal{H}_3(Y)$ is a real one dimensional subspace of $\mathcal{H}_3(S)$,
the vector space $\textnormal{Fix}_{
\Sigma} H^{0}(X_{J}, K_{J}^3)=\{q \in H^{0}(X_{J}, K_{J}^3) \textnormal{ }|\textbf{ } \psi^A\cdot q=q, \forall \psi \in \Sigma \}$ is also real one-dimensional. It is formed by the real span of a single holomorphic cubic differential $q$.

To further consider the counterpart of $\mathcal{H}_3(Y)$ in $\cal{M}^B/\cal{D}_0$, we need to discuss the $S^1$ action on $\mathcal{H}_3(Y)$ and  $ \textnormal{Fix}_{\Sigma} H^{0}(X_{J}, K_{J}^3)\subset \bigoplus\limits_{i=2}^{3}H^0(X_{J}, K_{J}^i)$ (recall Section \ref{subsec, BlaschkeHitchin}). We first need the following lemma which allows us to identify the Hitchin parametrization $H_J$ and the Hitchin map $\mathbf{H}$ on some special fibers:

\begin{lem}\label{lem,identification}
For any complex structure $J$, when restricting to $H^{0}(X_{J}, K_{J}^3)$, the composition of the inverse of the Hitchin map $\mathbf{H}^{-1}$ and the Hitchin parametrization $H_J$
$$\mathbf{H}^{-1}\circ H_J:\bigoplus\limits_{i=2}^{3}H^0(X_{J}, K_{J}^i) \to \cal{H}_3(S) \to Q_3(S) $$ satisfies 
$$ \mathbf{H}^{-1}\circ H_J|_{H^0(X_{J}, K_{J}^3)}= \textnormal{Id}|_{H^0(X_{J}, K_{J}^3)},$$
where  $H^0(X_{J}, K_{J}^3)$ is identified with $H^0(X_{[J]}, K_{[J]}^3)$.
\end{lem}
\begin{proof}
The Hitchin map $\mathbf{H}:Q_3(S) \to \mathcal{H}_3(S)$ is a homeomorphism and  $\mathbf{H}([(J,q)])=H_J(0,q)$ for any representative $(J,q)$  in $[(J,q)]\in Q_3(S)$.
Equivalently, one has that $\mathbf{H}^{-1}\circ H_J(0,q)=[(J,q)]\in H^0(X_{[J]}, K_{[J]}^3)$ is the identity map.
\end{proof}

Regarding to the $S^1$ action on $\mathcal{H}_3(Y)$ with $Y=Y_J$ arising from Proposition \ref{CyclicOneDimHitchin}, we have

\begin{lem}\label{CircleActionY}
The $S^1$ action on $\h3 $ induces a two-to-one identification on $\mathcal{H}_3(Y)$ except at $H_J(0)$. The quotient of $\mathcal{H}_3(Y)$ by the $\mathbb{Z}_2$ action induced from the $S^1$ action, denoted by $\mathcal{H}_3(Y)/ \mathbb{Z}_2$, is homeomorphic to $\mathbb{R_+}= [0,\infty)$.
\end{lem} 

\begin{proof}
 The Hitchin parametrization $H_J$ 
 is $\Sigma$-equivariant and is a homeomorphism between $\textnormal{Fix}_{\Sigma} H^{0}(X_{J}, K_{J}^3)$ and $\mathcal{H}_3(Y)$ by Proposition \ref{SigmaEquiHitchinSection}. After identifying  $ \textnormal{Fix}_{\Sigma} H^{0}(X_{J}, K_{J}^3)$ with  $\textnormal{Fix}_{
\Sigma} H^{0}(X_{[J]}, K_{[J]}^3)$ and $H_J$ with $\mathbf{H}$ by the previous lemma,  the $S^1$ action on $Q_3(S)$ induces a $\mathbb{Z}_2$ action on $ \textnormal{Fix}_{\Sigma} H^{0}(X_{J}, K_{J}^3)= \text{span}_{\bb{R}}(q) \subset \bigoplus\limits_{i=2}^{3}H^0(X_{J}, K_{J}^i)$. Moreover, this $\mathbb{Z}_2$ action identifies $tq$ with $-tq$ for any $t\in \mathbb{R}/\{0\}$. 
 The identification is trivial when $t=0$. We obtain that the quotient $\mathcal{H}_3(Y)/ \mathbb{Z}_2$ is homeomorphic to the half line $\mathbb{R_+}= [0,\infty)$.
\end{proof}

With what we have obtained, we are able to show the following lemma. It suggests that the half line given by $\mathcal{H}_3(Y)/ \mathbb{Z}_2$ provides a (unparametrized) geodesic in $ \mathcal{M}^B /\mathcal{D}_0$ via the map $\Phi:\h3 /S^1 \to \mathcal{M}^B /\mathcal{D}_0$ defined in Corollary \ref{cor, mappingClassEquiv}.

\begin{lem}\label{lem fixedPointSet}
The set $\Phi(\mathcal{H}_3(Y)/\mathbb{Z}_2) \subset \mathcal{M}^B /\mathcal{D}_0$ is a fixed point set of the group action of $\Sigma$, where $\Sigma\leq \cal{D}$ is identified with a finite subgroup of $\Exmod$.
\end{lem}

\begin{proof}
If $\psi\in \Sigma$ is orientation-preserving, then the fact that every point in $\Phi(\mathcal{H}_3(Y)/\mathbb{Z}_2)$ is fixed by $[\psi]$ follows from Corollary \ref{cor, mappingClassEquiv} and the definition of $\mathcal{H}_3(Y)$. Otherwise, if $\psi\in \Sigma$ is orientation-reversing, then it is an anti-holomorphic map with respect to $X_J$. We have $\psi^* q\in H^{0}(X_{-J}, K^3_{-J})$ for $q\in H^{0}(X_{J}, K^3_{J})$. Notice that $\Phi(\mathcal{H}_3(Y)/\mathbb{Z}_2 )= \mathbf{g}_0\circ\mathbf{H}_0^{-1}\big(\mathcal{H}_3(Y)/\mathbb{Z}_2 \big)= \mathbf{g}_0 \Big(\textnormal{Fix}_{
\Sigma} H^{0}(X_{J}, K_{J}^3)/ \mathbb{Z}_2\Big) $ by Lemma \ref{lem,identification}. 
For 
 $q\in \textnormal{Fix}_{
\Sigma} H^{0}(X_{J}, K_{J}^3)$, the fact that $\psi^A \cdot q=q$ implies, by the discussion in Subsection \ref{MappingClassHolomorphicDiff},  
$$\psi^*q=\kappa^{-1}_{\psi}\circ q \circ \psi=\overline{\tau_{\psi^{-1}}\circ q \circ \psi}=\overline{(\psi^A)^{-1}\cdot q}=\overline{q}. $$
Together with the observation that the solution of equation (\ref{WangEquation}) is invariant under the complex conjugation $z\to \overline{z}$ and $q \to \overline{q}$, we obtain $[\psi]^{*}\mathbf{g}_0([(J,\langle q \rangle)])=\mathbf{g}_0([(-J, \langle \overline{q}\rangle)])=\mathbf{g}_0([(J, \langle q\rangle)])$. 
Here we made use of Remark \ref{orientation_cplx}. So $[\psi]\cdot\mathbf{g}_0([(J,\langle q \rangle)])=[\psi^{-1}]^{*}\mathbf{g}_0([(J,\langle q \rangle)])=\mathbf{g}_0([(J,\langle q \rangle)])$ and $\mathbf{g}_0([(J,\langle q \rangle)])$ is a fixed point of the induced left action of $\Sigma$. 
\end{proof}

We further have

\begin{thm} \label{thm,geodesics}
Let $Y$ be a non-orientable orbifold of negative Euler characteristic with orientation double cover $Y^{+}$ given by the cases in Proposition \ref{CyclicOneDimHitchin}. Then $\mathcal{H}_3(Y)/\mathbb{Z}_2$ embeds as a (unparametrized) geodesic in $\mathcal{M}^B/\mathcal{D}_0$ with respect to the covariance metric $G(\cdot,\cdot)$.

\end{thm}
\begin{proof}

By Proposition \ref{CyclicOneDimHitchin} and Lemma \ref{CircleActionY}, the set $\Phi(\mathcal{H}_3(Y)/\mathbb{Z}_2)$ is homeomorphic to a half line $\mathbb{R}_{+}$  in $\mathcal{M}^B/\mathcal{D}_0$. By Lemma \ref{lem fixedPointSet}, the set $\Phi(\mathcal{H}_3(Y)/\mathbb{Z}_2)$ is pointwise fixed by the action of the group $\Sigma \leq \textnormal{Out}(\pi_1S)\cong\Exmod$. Because $\Exmod$ is a subgroup of isometries of the covariance metric $G(\cdot,\cdot)$ on $\mathcal{M}^B/\mathcal{D}_0$ and a one dimensional connected subset pointwise fixed by a subgroup of isometries must be a geodesic (see \cite[Theorem 5.1]{Kobayashi}), we conclude that $\mathcal{H}_3(Y)/\mathbb{Z}_2$ embeds as a (unparametrized) geodesic in $\mathcal{M}^B/\mathcal{D}_0$.
\end{proof}

\subsection{Infinite length geodesics}\label{infiniteLength}

In this subsection, we estimate lengths of certain families of curves in the Blaschke locus $\mathcal{M}^B /\mathcal{D}_0$ with respect to the covariance metric; some of them are geodesics, as mentioned in the last part of Section \ref{subsec:geodesics}.
Specifically, fix a complex structure $J$ on $S$ and let $\sigma\in \mathcal{M}_{-1}$ be the corresponding hyperbolic metric.
Let $q$ be a nonzero holomorphic cubic differential with respect to $J$ and consider the curve $\{g_t=e^{u_t}\sigma:t\geq 0\} \subset\mathcal{M}^B $, 
where, as in \eqref{eq:Blaschke_family_no_t}, 
for each $t\geq 0$ the logarithmic density $u_t$ satisfies
\begin{equation}\label{eq:Blaschke_family}
\Delta_{\sigma}u_t= 2e^{u_t}-4te^{-2u_t} \frac{|q|^2}{\sigma^3}-2.
\end{equation}
With $g_t=e^{u_t}\sigma$, \eqref{eq:Blaschke_family} is equivalent to Wang's equation (\ref{WangEquation}) given by $K_{g_t}=-1+2|\sqrt{t}\,
q|_{g_t}^2$. 
Our goal is to estimate the length of the curve $\{[g_t]\}_{t\geq 0} \subset \cal M^B/\cal D_0$. We start with some preliminaries.

\subsubsection{Some Estimates for Blaschke metrics}

%Moreover, for a fixed $t\geq 0$, if $\dot{u}_t(p)=0$ for some $p\in S$, then $p$ must be a zero of the holomorphic cubic differential $q$.

The following Lemma benefits from communication with Michael Wolf.
\begin{lem}\label{lem:MonotoneDensity}
The logarithmic density $u_t$ of the Blaschke metric $g_t$ is monotone increasing with respect to the parameter $t$ when $t\geq 0$.
\end{lem}

\begin{proof}
This is proved more generally in  \cite[Proposition 4.1]{CyclicHiggsBundle}. For completeness, we include a proof for the case we need. Taking a time derivative of \eqref{eq:Blaschke_family},
we find that
\begin{equation}\label{eq:DevBlaschke}
    \Delta_{\sigma}\dot{u}_t=2\Big(e^{u_t} +4e^{-2u_t}\frac{t|q|^2}{\sigma^3}\Big)\dot{u}_t-4e^{-2u_t}\frac{|q|^2}{\sigma^3}.
\end{equation}
By elliptic regularity, equation \eqref{eq:DevBlaschke} implies that $\dot{u}_t\in C^\infty(S)$.
We prove that $\dot{u}_t > 0$ for all $t \geq 0$. Suppose that this is not the case for some $t \geq 0$, so that for this $t$ one has $\dot{u}_t(p)\leq 0$ for a $p$ on $S$ where the minimum of $\dot{u}_t$ is attained.
 Consider the operator $L=-\Delta_\sigma +2e^{u_t}$. We rewrite equation \eqref{eq:DevBlaschke} as 
\begin{equation}\label{eq:L_operator}
    L\dot{u}_t=4e^{-2u_t}\frac{|q|^2}{\sigma^3}(1-2t\dot{u}_t).
\end{equation}
In a small neighborhood $U$ of $p$,  we have $(1-2t\dot{u}_t)\geq 0$, which implies that $ L\dot{u}_t\big|_U\geq 0$. So by the strong maximum principle \cite[Section 6.4, Theorem 4(ii)]{EvansPDE},  $\dot{u}_t $ is constant on $U$. If $\dot{u}_t\big|_U \equiv 0$, then $q=0$ on $U$, and thus everywhere, which contradicts our assumption. If on the other hand $\dot{u}_t\big|_U<0$, then the left hand side of \eqref{eq:L_operator} is negative on $U$ whereas the right hand side is non-negative, and this is  a contradiction.
%\noteN{
%Maybe one can do something like the following, set 
%\begin{equation}
%    L=-\Delta_\sigma +2e^{u_t},
%\end{equation}
%then 
%\begin{equation}\label{eq:L_operator}
%    L\dot{u}_t=4e^{-2u_t}\frac{|q|^2}{\sigma^3}(1-2t\dot{u}_t).
%\end{equation}
%If $p$ is a point where $\dot{u}_t$ attains a non-positive minimum, there is a small neighborhood $U$ of it where $(1-2t\dot{u}_t)\geq 0$, so by the strong maximum principle \cite[Section 6.4, Theorem 4(ii)]{EvansPDE} $\dot{u}_t $ is constant on $U$.
%In this case, if $\dot{u}\big|_U=0$, then $q=0$ on $U$ (and thus everywhere, which contradicts our assumption).
%}
%
%
%
%
%
%
\end{proof}

We will need the following result, due to Loftin:
\begin{prop}[{\cite[Proposition 4.02]{Loftin_AffineSphere}}]\label{prop:SuperSub}
If $u_t$ satisfies equation \eqref{eq:Blaschke_family} with respect to the hyperbolic metric $\sigma$ and the holomorphic cubic differential $q$ for $t\geq0$, then 
\begin{equation}\label{eq:supersub}
    0 \leq u_t \leq  \log \Big(R\Big(\max_S\Big\{\frac{t|q|^2}{\sigma^3}\Big\}\Big)\Big),
\end{equation}  
where $R(a)$ is the largest positive root of the polynomial $p_a(x)=2x^3-2x^2-4a$.
\end{prop}

\begin{lem}\label{lm:root}
The polynomial $p_a(x)=2x^3-2x^2-4a$ has a unique positive root provided $a\geq 0$, and if we set
  $x_t$
  as the positive root of $ x^3-x^2-2\max\limits_{S}\left\{\frac{t|q|^2}{\sigma^3}\right\}$, we have
 \begin{equation}
 \lim_{t\to\infty}\frac{x_t}{t^{1/3}}= \left(2\max_S\left\{\frac{|q|^2}{\sigma^3}\right\}\right)^{1/3}.
 \end{equation}
\end{lem}

\begin{proof}
The fact that the polynomial $p_a(x)$ has a unique positive root for $a\geq 0$ is clear if $a=0$. If $a>0$, we know that a real root exists and any real root $x_a$ satisfies $x_a^2(x_a-1)=2a$. Thus a real root $x_a$ satisfies $x_a>1$.
Taking the derivative, we find that  $p_a'(x_a)>0$. Therefore there cannot be more than one positive real roots, because between any two  roots for which $p_a'>0$ there has to be at least one more root and $p_a$ has at most three real roots in total.
Write $\displaystyle b=2\max_S\left\{\frac{|q|^2}{\sigma^3}\right\}>0$; we look for $x>0$ satisfying $x^3-x^2-bt=0$, for $t>0$ large.
Set $z=x/t^{1/3}$ and $s=t^{-1/3}$.
Notice that $(x,t)\mapsto (z,s)$ is a bijection on $(0,\infty)\times (0,\infty)$, so $x^3-x^2-bt=0$ with $x$, $t>0$ exactly when 
\begin{equation}
    F(z,s):=z^3-z^2s-b=0, \quad z,\; s>0.
\end{equation} 
The function $F$ satisfies $F(b^{1/3},0)=0$ and is smooth near $(z,s)=(b^{1/3},0)$. Since $\partial_z F (b^{1/3},0)=3b^{2/3}$, by the implicit function theorem we conclude that in a neighborhood of $(b^{1/3},0)$, the equation $F(z,s)=0$ holds if and only if  $z=f(s)$ for a smooth function $f$ defined in a neighborhood of $s= 0$. Denoting by $x_t$  the positive solution of $x^3-x^2-bt=0$, 
\begin{equation}
    \lim_{t\to \infty}\frac{x_t}{t^{1/3}}=\lim_{t\to \infty} f(t^{-1/3})=\lim_{s\to0} f(s)=b^{1/3}.
\end{equation}  
\end{proof}

Loftin proves the following result regarding the asymptotic behavior of the Blaschke metrics $g_t$ when $t\to \infty$ (see also  \cite[Theorem 3.8]{VolumeEntropy-Tholozan}).
\begin{thm} [{\cite[Proposition 1]{Loftin_FlatMetrics}}] \label{thm, singularFlat}
The family of Blaschke metrics $g_t$ given by equation \eqref{eq:Blaschke_family} degenerates to the singular flat metric $|q|^{\frac{2}{3}}$ associated to the cubic differential $q$ (up to some scaling) in the following sense:  when $t\to\infty$, we have
$$t^{-\frac{1}{3}} g_t \to 2^{\frac{1}{3}}|q|^{\frac{2}{3}},$$
uniformly on every compact subset of the complement of the zeros of $q$ in $S$. 
\end{thm}

\subsubsection{Length estimates}

The following proposition shows that any curve in the space $\cal M^B/\cal D_0$ corresponding to a ray in a fixed fiber of the bundle $ Q_3(S)$
has infinite length with respect to the covariance metric.

\begin{thm}\label{thm,infiniteLength}
Let $\sigma $ be a hyperbolic metric on $S$ and $q$ be a nonzero cubic differential which is holomorphic with respect to the complex structure determined by $\sigma$ and an orientation.
Then the projection $\{[g_t]\}_{t\geq0} \subset \cal M^B/\cal D_0$  of the curve
 $g_t=\{e^{u_t}\sigma\}_{t\geq 0}\subset \cal M^B$, where $u_t$ satisfies equation \eqref{eq:Blaschke_family},
has infinite length with respect to the  covariance metric.

\end{thm}
\begin{proof}
By Lemma \ref{lem, KillPotential}, for given class $[g]\in \mathcal{M}_-/\cal D_0$, the norm of an element $\hat{h}\in T_{[g]}(\mathcal{M}_-/\cal D_0)$ with respect to the covariance metric is given by $G_{g}(h,h)^{1/2}$, 
where $g$ is a representative in $[g]$, 
$h\in T_g\mathcal{M}_-$ is a  lift of $\hat{h}$,
and $G_{g}(\cdot,\cdot)$ is the bilinear form defined in Definition \ref{def:bilinear}.
Thus the length of the segment  $\gamma_T:=\{[g_t]:t\in [0,T]\}\subset \cal M^B/\cal D_0$ is given by 
\begin{equation}                
    L(\gamma_T)=\int_0^T \sqrt{G_{g_t}\big(\partial_t(g_t),\partial_t(g_t)\big)}dt.
\end{equation}
We will show that 
\begin{equation*}\label{eq:limit}
    \lim_{T\to \infty}L(\gamma_T)=\infty.
\end{equation*}

To simplify notation, we denote $h_t:=\partial_t(g_t)$. We have, according to Definition \ref{def:bilinear},
\begin{align*}
    G_{g_t}(h_t,h_t)=\langle \Pi^{g_t}_2 h_t,h_t \rangle_{L^2(T^1S_{g_t})}
    =\langle \Pi^{g_t}\pi_2^*h_t,\pi_2^*h_t\rangle+|\langle 1,\pi_2^*h_t \rangle_{L^2(T^1S_{g_t})}|^2
\end{align*}
Note that $h_t=\dot{u}_t g_t$. Since $\pi_2^*h_t(v)=\pi_2^*(\dot{u}_tg_t)(v)=\pi_0^*( \dot{u}_t)(v)$ for $v\in T^1S_{g_t}$,  we have\footnote{Recall that the Liouville measure is of total mass $1$, so locally  $d\mu^L_g= \frac{d\theta dv_g }{2\pi \text{Area}(S, g) }$.}
\begin{equation}
\begin{aligned}
    G_{g_t}(h_t,h_t)
    =&\langle \Pi^{g_t}\pi_0^*\dot{u}_t,\pi_0^*\dot{u}_t\rangle +|\langle 1,\pi_0^*\dot{u}_t \rangle_{L^2(T^1S_{g_t})}|^2\nonumber\\
    &\geq |\langle 1,\pi_0^*\dot{u}_t\rangle_{L^2(T^1S_{g_t})}|^2
   =\frac{|\langle 1,\dot{u}_t\rangle_{L^2(S,dv_{g_t})}|^2}{\text{Area}(S,g_t)^2 } ,
\end{aligned} 
\label{eq:lower_bound}
\end{equation}
 where the inequality follows by Theorem \ref{thm,PiProperty}.
Note that above we wrote $\langle\cdot,\cdot\rangle_{L^2(S,dv_{g_t})}=\mathrm{Area}(S,g_t) \langle\cdot,\cdot\rangle_{L^2_{g_t}(S)}$.

We now examine $\dot{u}_t$ in more detail.
Since our manifold is two dimensional, we have $\Delta_{\sigma}=e^{u_t}\Delta_{g_t}$, thus from  \eqref{eq:DevBlaschke} it follows that
\begin{equation}
    \Delta_{g_t}\dot{u}_t=\left(2 +8\frac{t|q|^2}{g_t^3}\right)\dot{u}_t-4\frac{|q|^2}{g_t^3}.\label{eq:derivative}
\end{equation}

Now integrate \eqref{eq:derivative} over $S$ with respect to the $g_t$ area form: by the divergence theorem, $\int_S \Delta_{g_t}\dot{u}_t dv_{g_t}=0$ and therefore 
\begin{equation}
   \left \langle \left(2 +8\frac{t|q|^2}{g_t^3}\right)\dot{u}_t,1\right\rangle_{L^2(S,dv_{g_t})}=4\left\langle\frac{|q|^2}{g_t^3},1\right\rangle_{L^2(S,dv_{g_t})}.
\end{equation}
The curvature of $g_t$ is given by $K_{g_t}= 2t\frac{|q|^2}{g_t^3}-1$ (recall equation \eqref{WangEquation}), therefore
\begin{equation}
    \langle (6+4K_{g_t})\dot{u}_t,1\rangle_{L^2(S,dv_{g_t})}=4\left\langle\frac{|q|^2}{g_t^3},1\right\rangle_{L^2(S,dv_{g_t})}.
\end{equation}
Using the fact that $\dot{u}_t\geq 0$ (by Lemma \ref{lem:MonotoneDensity}) and that $K_{g_t}<0$ (by Proposition \ref{prop:neg_curv}), we find 
\begin{equation}
    \langle \dot{u}_t,1\rangle_{L^2(S,dv_{g_t})}\geq \langle (1+\frac{2}{3}K_{g_t})\dot{u}_t,1\rangle_{L^2(S,dv_{g_t})}=\frac{2}{3}\left\langle\frac{|q|^2}{g_t^3},1\right\rangle_{L^2(S,dv_{g_t})}.\label{eq:average}
\end{equation}
We now seek for a lower bound for the right hand side.
Since $dv_{g_t}=e^{u_t}dv_{\sigma}$, we have 
\begin{equation*}
    \left\langle\frac{|q|^2}{g_t^3},1\right\rangle_{L^2(S,dv_{g_t})}=\int_{S} e^{-2u_t}\frac{|q|^2}{\sigma^3}dv_{\sigma}\geq \frac{1}{x_t^2} \int_{S} \frac{|q|^2}{\sigma^3}dv_{\sigma};
\end{equation*}
the inequality is by Proposition \ref{prop:SuperSub}, where  $x_t$ is the positive root of the  polynomial 
$ x^3-x^2-2\max\limits_{S}\left\{\frac{t|q|^2}{\sigma^3}\right\}$.
Thus by Lemma \ref{lm:root}, there exists a positive constant $C_{q,\sigma}$ and $T_0>0$ depending on $C_{q,\sigma}$ such that for  $t\geq T_0$ one has 
\begin{equation*}
     \left\langle\frac{|q|^2}{g_t^3},1\right\rangle_{L^2(S,dv_{g_t})}\geq C_{q,\sigma}t^{-2/3}.
\end{equation*}

On the other hand, we have $\text{Area}(S, g_t) \leq D_q t^{\frac{1}{3}} + 2\pi|\chi(S)| $ (see  \cite[Lemma 3.4]{OuyangTamburelli}) where $D_q$ is a positive constant depending only on $q$ and $\chi(S)$ is the Euler characteristic of $S$.
 Combining this with equation \eqref{eq:lower_bound} and equation \eqref{eq:average}, and adjusting $T_0$ if necessary, we can find a positive constant $C'_{q,\sigma}$ so that for $t\gg T_0$,
$$\sqrt{G_{g_t}(h_t,h_t)}\geq C'_{q,\sigma}t^{-1}.$$
Hence we obtain, for $T\geq T_0$, 
\begin{equation*}
    L(\gamma_T)\geq L(\gamma_{T_0})+C'_{q,\sigma}\int^T_{T_0}t^{-1}dt\overset{T\to\infty}{\to}\infty.
\end{equation*}
\end{proof}

\begin{cor}
The geodesics $\mathcal{H}_3(Y)$ given in Theorem \ref{thm,geodesics} have infinite length with respect to the covariance metric.
 \end{cor}

 \appendix 
\section{Some further estimates for the covariance metric in \texorpdfstring{$\cal M^B/\cal D_0$}{MB/D0}} \label{furtherEstimates}

We collect in this appendix some estimates for the covariance metric in the Blaschke locus $\cal M^B/\cal D_0$. One first interesting question is to understand the behavior of the covariance metric $G(\cdot, \cdot)$ at the Fuchsian locus $\cal{T}(S) \subset \cal M^B/\cal D_0$. Given $[\sigma] \in \cal{T}(S)$, we know that this metric restricts to a scalar of the Weil-Petersson metric on $T_{[\sigma]}\cal{T}(S)$; However, the  behavior of the covariance metric $G(\cdot, \cdot)$ at $[\sigma]$ on directions that are transverse to Fuchsian locus $\cal{T}(S)$ in $\cal M^B/\cal D_0$ remains unclear. 

Our first estimate is along directions tangential to the family of equivalence classes of Blaschke metrics $\{[g_t] \}_{t\geq 0}\subset \mathcal{M}^B /\mathcal{D}_0$ that lifts to the curve given by $\{g_t=e^{u_t}\sigma:t\geq 0\} \subset\mathcal{M}^B $ described by equation \eqref{eq:Blaschke_family}, so that $\sigma$ is a representative in the class $[\sigma]$. These represent vectors tangential to fiber directions of the bundle $ Q_3(S)/S^1$. 

\begin{prop}\label{FuchsianEstimate}
Let $\sigma $ be a hyperbolic metric on $S$. Fix a holomorphic cubic differential $q$ with respect to the complex structure corresponding to $\sigma$ and an orientation. Suppose $\widehat{h}_0\in T_{[g_0]}\cal M_{-}/\cal D_0 $ is a vector that lifts to  $h_0:=\partial_t(g_t)|_{t=0}\in T_{\sigma}\cal M_{-}$, where $\{g_t\}_{t\geq 0}$ are solutions of equation \eqref{eq:Blaschke_family} with $g_0=\sigma$. 
Then  
$$G_{[\sigma]}\big(\widehat{h}_0,\widehat{h}_0\big)\geq \frac{1}{\pi^2|\chi(S)|^2} \|q\|_{\sigma}^4.$$
where $\|q\|_{\sigma}$ is the $L^2$ norm of $q$ with respect to inner product \eqref{hermitian} and $\chi(S)$ is the Euler characteristic.
\end{prop}

\begin{proof}
By Lemma \ref{lem, KillPotential}, the computation can be lifted to $\cal M^B$, and we just need to estimate $G_{\sigma}(h_0,h_0)$.
Recall that 
\begin{align*}
   G_{\sigma}(h_0,h_0)=\mathrm{Var}(P_{\sigma}(\pi_2^*h_0), \mu^L_{\sigma})+\langle\pi_2^*h_0,1\rangle_{L^2(T^1S_\sigma)}^2.
\end{align*}
The first part, which equals $\langle \Pi^{\sigma}  \pi_2^*h_0, \pi_2^*h_0\rangle$, is nonnegative (Theorem \ref{thm,PiProperty}). We estimate the second part. Note that $g_t=e^{u_t}\sigma$ and $h_0=\dot{u}_0 \sigma$. From equation \ref{eq:derivative}, we have
\begin{equation}
    (\Delta_{\sigma}-2)\dot{u}_0=-4|q|^2_\sigma,
\end{equation}
so
\begin{equation}
    \langle\pi_2^*h_0,1\rangle_{L^2(T^1S_\sigma)}
    =\int_{T^1S_{\sigma}}\pi_0^*\big( -4(\Delta_{\sigma}-2)^{-1}\big(|q|^2_\sigma\big)\big) \pi_2^*\sigma d\mu_{\sigma}^L
     =\int_{T^1S_{\sigma}}\pi_0^*\big( -4(\Delta_{\sigma}-2)^{-1}\big(|q|^2_\sigma\big)\big) d\mu_{\sigma}^L.
\end{equation}
Since $-(\Delta_{\sigma}-2)^{-1}$ is a self-adjoint positive operator and $-(\Delta_{\sigma}-2)^{-1}(1)=\frac{1}{2}$, we obtain
\begin{align*}
    \langle\pi_2^*h_0,1\rangle _{L^2(T^1S_\sigma)}
     =& \frac{1}{\text{Area}(S, \sigma)} \int_{S} 4|q|^2_\sigma \big(-(\Delta_{\sigma}-2)^{-1}(1)\big) dv_{\sigma}
     =\frac{1}{\pi|\chi(S)|} \int_{S}|q|^2_\sigma dv_{\sigma}
     =\frac{1}{\pi|\chi(S)|} \|q\|^2_{\sigma},
     \end{align*}
using the Gauß-Bonnet theorem. This gives the conclusion.
\end{proof}

The inequality can be strengthened to strict inequality because of the following two lemmas.

\begin{lem}\label{PiConstant}
Suppose $g\in \mathcal{M}_-$ and $h\in \sf{S}_{2}(S)$ is conformal to $g$, then 
$$\langle \Pi^{g} \pi_2^* h,\pi_2^* h\rangle=0$$
if and only if $h=C g$, where $C$ is a constant.
\end{lem}

\begin{proof}
Let's assume $h=f g$, with $f\in C^\infty (S)$. If suffices to show that $$\langle \Pi^{g} \pi_0^* f,\pi_0^* f\rangle = 0$$ is equivalent to  $f$ being a constant function.
The proof for this is similar to the proof for \cite[Lemma 4.6]{Thibault_MicrolocalAnalysis}. Let $\Lambda=(1-\Delta_{G})^{1/2} $, where $ \Delta_{G}$ is the (negative) Laplace-Beltrami operator with respect to the Sasaki metric on $T^1S_g$; then for any $m,s\in \bb{R}$, $\Lambda^s:\cal{H}^m(T^1S_g)\to \cal{H}^{m-s}(T^1S_g)$ is an invertible bounded self-adjoint operator.
So by Theorem \ref{thm,PiProperty}, the composition $\Lambda^{-2s}\Pi^{g} :\cal{H}^s(T^1S_g)\to \cal{H}^{s}(T^1S_g)$ is bounded for $s>0$.
Moreover, we know from equation \eqref{SobolevInner_product} and Theorem \ref{thm,PiProperty},
\begin{equation}\label{eq:computation} 
    \langle \Lambda^{-2s}\Pi^{g} \pi_0^* f',\pi_0^*f'\rangle_{\cal{H}^s(T^1S_g)}=\langle \Pi^{g} \pi_0^* f',\pi_0^*f'\rangle_{L^2(T^1S_g)}\geq0  \qquad   \text{ for all } f'\in \cal{H}^s(S),
\end{equation} 
so
 $\Lambda^{-2s}\Pi^{g}$ is also positive on $ \cal{H}^s(T^1S_g)$.
It follows that $\Lambda^{-2s}\Pi^{g}$ is self-adjoint on $ \cal{H}^s(T^1S_g)$, and as such it has a self-adjoint square root $R$ which is bounded on $\cal{H}^s(T^1S_g)$, with the property $\Lambda^{2s} \Pi^{g}=R^2=R^*R $. Thus,
\begin{equation*}
    0=\langle \Pi^{g} \pi_0^* f,\pi_0^*f\rangle_{L^2(T^1S_g)}=\langle \Lambda^{-2s}\Pi^{g} \pi_0^* f,\pi_0^*f\rangle_{\cal{H}^s(T^1S_g)}=\| R \pi_0^* f\|_{\cal{H}^s(T^1S_g)}^2
\end{equation*}
Thus our hypothesis implies that $ \Pi^{g}\pi_0^*f=\Pi^{g}(P_g(\pi_0^*f))=0$ (recall that $\Pi^{g}1=0$ and  $P_g(\pi_0^*f) = \pi_0^*f-\langle \pi_0^*f,1\rangle_{L^2(T^1S_g)}$).
Then by Theorem \ref{thm,PiProperty},  there exists $w\in \cal{H}^{s}(T^1 S_g)$ such that $X^g w=P_g(\pi_0^*f)$, where $X^g$ is the geodesic vector field on $T^1S_g$.
Therefore, for every closed orbit $c$ of the flow of $X^g$, parametrized as $\phi_\cdot (z):[0,L(c)]\to T^1S_g$, 
\begin{equation}
    0=\frac{1}{L(c)}\int_0^{L(c)} (P_g(\pi_0^*f))(\phi_t(z))\mathrm{d}t=
    \frac{1}{L(c)}\int_0^{L(c)} \pi_0^*\big(f-\langle f,1\rangle_{L^2_g(S)}\big)(\phi_t(z)\big)\mathrm{d}t=I_0^g\big(f-\langle f,1\rangle_{L^2_g(S)}\big)(c).
\end{equation} 
By the injectivity of the X-ray transform for negatively curved metrics (\cite{Guillemin1980}), $f=\langle f,1\rangle_{L^2_g(S)}$, i.e., a constant. 
\end{proof}

\begin{lem}\label{DensityNotConst}
The solution $\dot{u}_t$ of  equation \eqref{eq:DevBlaschke}
\begin{equation*}
    \Delta_{\sigma}\dot{u}_t=\left(2e^{u_t} +8e^{-2u_t}\frac{t|q|^2}{\sigma^3}\right)\dot{u}_t-4e^{-2u_t}\frac{|q|^2}{\sigma^3}
\end{equation*} 
cannot be constant for any $t\geq 0$ unless $q \equiv 0$.
\end{lem}
\begin{proof}
We prove the claim by contradiction. Suppose $\dot{u}_t=c_t$ and $q\not\equiv 0$. Recall that $g_t=e^{u_t}\sigma $ satisfies 
\begin{equation*}
    \Delta_{g_t}\dot{u}_t=\left(2 +8t\frac{|q|^2}{g_t^3}\right)\dot{u}_t-4\frac{|q|^2}{g_t^3}.
\end{equation*}
This yields:
\begin{equation}
(4-8tc_t)\frac{|q|^2}{g_t^3}=2c_t\implies \frac{|q|^2}{g_t^3}=\frac{2c_t}{(4-8tc_t)}.
\end{equation}
This is impossible for any $t\geq 0$, because $q$ is nonzero and  the left hand side only vanishes at  the finite zeros of $q$, while the right hand side is constant. This yields the claim.
\end{proof}

\begin{cor}
Fix a holomorphic cubic differential $q$ with respect to the complex structure corresponding to the hyperbolic metric $\sigma$ and an orientation. Suppose $[g_t]\in \cal M^B/\cal D_0$ lifts to the solutions $\{g_t\}$ of equation \eqref{eq:Blaschke_family} and $\widehat{h}_t\in T_{[g_t]}\cal M^B/\cal D_0 $  lifts to $h_t= \partial_t(g_t)$. Then 
$$\langle \Pi^{g_t} \pi_2^*h_t,\pi_2^*h_t \rangle > 0.$$
In particular, this implies
$$G_{[\sigma]}(\widehat{h}_0,\widehat{h}_0) > \frac{1}{\pi^2|\chi(S)|^2} \|q\|_{\sigma}^4.$$
\end{cor}
\begin{proof}
Since $g_t=e^{u_t}\sigma $ and  $h_t=\dot{u}_t g_t$, the first statement follows from Lemma \ref{PiConstant} and Lemma \ref{DensityNotConst}.  The second statement is a special case at $t=0$ combined with Proposition \ref{FuchsianEstimate}. 
\end{proof}

It would be desirable to obtain a concise explicit formula for $G_{[\sigma]}\big(\widehat{h}_0,\widehat{h}_0\big) $ mentioned above. We conclude with a partial answer to this question following from \cite{GuilMon}. 
It would be of interest to know whether the formula below can be further simplified.

\begin{prop} \label{prop: FiberDirection}
Under the same assumptions of Proposition \ref{FuchsianEstimate}, we have 
\begin{align*} 
    &G_{[\sigma]}(\widehat{h}_0,\widehat{h}_0)=\left\langle  4\dfrac{\Gamma(1/4-\mathcal{L})\Gamma(1/4+\mathcal{L}) }{\Gamma(3/4-\mathcal{L})\Gamma(3/4+\mathcal{L})}(-\Delta_{\sigma}+2)^{-1}4\frac{|q|^2}{\sigma^3},(-\Delta_{\sigma}+2)^{-1}4\frac{|q|^2}{\sigma^3}\right\rangle_{L^2_{\sigma}(S)} 
    +\frac{1}{\pi^2|\chi(S)|^2} \|q\|_{\sigma}^4,\label{eq:gamma}
\end{align*}
where $\Gamma(\cdot, \cdot)$ is the Euler Gamma function and $\mathcal{L}=\frac{i}{2}\sqrt{-\Delta_{\sigma}(1-P_0)-1/4}$, where $P_0$ is the orthogonal projector onto $\ker \Delta_{\sigma}$.
\end{prop}

\begin{proof}
The proof is a direct application of \cite[Lemma A.1, Remark A.2]{GuilMon} combined with Proposition \ref{FuchsianEstimate}. 
\end{proof}

\printbibliography
% \bibliography{ref}
% \bibliographystyle{alpha}

\end{document}